\documentclass{amsart}
\usepackage[utf8]{inputenc}
\usepackage[utf8]{inputenc}
\usepackage{amsfonts}
\usepackage{amsmath}
\usepackage{amssymb}
\usepackage{mathrsfs}
\usepackage{amsthm}
\usepackage{float}
\usepackage{upgreek}
\usepackage{enumitem}
\usepackage{mathabx}
\usepackage{color}
\usepackage{thmtools}
\usepackage{thm-restate}
\usepackage{stmaryrd}
\usepackage[dvipsnames]{xcolor}
\usepackage{caption}
\usepackage{subcaption}
\usepackage{thmtools}
\usepackage{hyperref}
\usepackage{enumitem}
\usepackage{thm-restate}
\usepackage{cleveref}
\usepackage{xcolor}
\newcommand{\subsetsim}{\mathrel{\ooalign{\raise.4ex\hbox{$\subset$}\cr$\raise-.9ex\hbox{$\sim$}$}}}

\usepackage[left=1.03in, right=1.03in, top=1.05in, bottom=1.05in]{geometry}
\numberwithin{equation}{section}
\DeclareMathOperator{\calt}{\mathcal T}

\DeclareMathOperator{\Q}{\mathcal{Q}}

\DeclareMathOperator{\T}{\mathbb{T}}

\DeclareMathOperator{\Z}{\mathcal{Z}}
\DeclareMathOperator{\J}{\mathcal{J}}
\DeclareMathOperator{\cal}{\mathcal{L}}
\DeclareMathOperator{\W}{\mathcal{W}}
\DeclareMathOperator{\A}{\mathcal{A}}
\DeclareMathOperator\supp{supp }
\DeclareMathOperator{\wid}{wid}
\newcommand{\dd}[1]{\mathrm{d}{#1}}
\DeclareMathOperator{\R}{\mathbb{R}}
\DeclareMathOperator{\B}{\mathbb{B}}
\newcommand{\U}{U_{\alpha,\beta}^{\mathrm{br}}}
\newtheorem{theorem}{Theorem}[section]
\newtheorem{lemma}[theorem]{Lemma}

\newtheorem{prop}[theorem]{Proposition}
\newtheorem{corollary}[theorem]{Corollary}
\newtheorem{remark}[theorem]{Remark}
\theoremstyle{definition}
\newtheorem{definition}[theorem]{Definition}
\theoremstyle{definition}
\newtheorem{example}[theorem]{Example}
\newtheorem{notation}[theorem]{Notation}

\usepackage{graphicx}
\newtheorem{ltheorem}{Theorem}

\DeclareMathOperator{\len}{len}

\title{Uniform Decoupling for Convex Curves}
\author{Hrit Roy}
\date{May 2025}
\begin{document}
\begin{abstract}
    Using a high/low argument, we prove a universal $\ell^2L^6$ decoupling estimate with constant $C_\epsilon R^{\epsilon}$ for general convex curves in the plane. These curves have no additional regularity assumptions, and the constant $C_\epsilon$ is uniform across all such curves.
\end{abstract}
\maketitle
\section{Introduction}\label{sec: intro}
\subsection{Background}
Let $R\geq 1$ be a large parameter, and let $f\in\mathcal{S}(\R^2)$ be any Schwarz function with Fourier support in $\mathcal{N}_{R^{-1}}(\Gamma_0)$, the $R^{-1}$-neighbourhood of the parabola $\Gamma_0:=\big\{(t,t^2):t\in[0,1]\big\}$. Let $\{\theta\}$ be the canonical partition of $\mathcal{N}_{R^{-1}}(\Gamma_0)$ into rectangular boxes of size roughly $R^{-1/2}\times R^{-1}$, and let $f_\theta:=\widecheck{\chi}_\theta*f$ denote the Fourier projection of $f$ onto $\theta$. For $p\geq 2$ let $D_{p,\Gamma_0}(R)$ denote the smallest constant such that 
\begin{equation}
    \label{eqn: parabola decoupling constant}
    \|f\|_{L^p(\R^2)}\leq D_{p,\Gamma_0}(R)\big(\sum_{\theta}\|f_\theta\|_{L^p(\R^2)}^2\big)^{1/2}
\end{equation}
holds for any $f\in\mathcal{S}(\R^2)$ with $\supp{\widehat{f}}\subseteq\mathcal{N}_{R^{-1}}(\Gamma_0)$.

An application of Cauchy--Schwarz shows that $D_{p,\Gamma_0}(R)\lesssim R^{1/4}$ for all $2\leq p\leq \infty$. While this is sharp for $p=\infty$, for other values of $p$ we can do much better. Plancherel's theorem shows that $D_{2,\Gamma_0}(R)\lesssim 1$, which uses the disjointedness of the boxes $\{\theta\}$. Using special scaling properties and the curvature of the parabola, Bourgain and Demeter \cite{BD} established the following stronger result.
\begin{ltheorem}(\cite[Theorem 1.1]{BD})
\label{thm: BD}
    For all $2\leq p\leq 6$ and all $\epsilon>0$, there exists a constant $C_\epsilon\geq 1$ such that 
    \begin{equation}
        \label{eqn: BD bound}
        D_{p,\Gamma_0}(R)\leq C_\epsilon R^\epsilon.
    \end{equation}
\end{ltheorem}
Interpolating \eqref{eqn: BD bound} with the sharp $L^\infty$ bound, Theorem \ref{thm: BD} also gives us $D_{p,\Gamma_0}(R)\lesssim_\epsilon R^{1/4-3/2p+\epsilon}$ for $6< p\leq\infty$. Various subsequent results improved the rate of blow-up (see \cite{li1,li2,gmw}), and provided simpler proofs of Theorem \ref{thm: BD} (see \cite{guo}). We want to focus on the result of Guth--Maldague--Wang \cite{gmw}, who improved the rate of blowup to a log-power.
\begin{ltheorem}(\cite[Theorem 1.1]{gmw})
    \label{thm: gmw}
    For all $2\leq p\leq 6$, there exists constants $c,C>0$ such that 
    \begin{equation}
        \label{eqn: GMW bound}
        D_{p,\Gamma_0}(R)\leq C(\log R)^{c}.
    \end{equation}
\end{ltheorem}
Besides providing the best known bound for $D_{p,\Gamma_0}(R)$ to date, Theorem \ref{thm: gmw} provides a strong proof technique for decoupling that is extremely versatile, which has many applications beyond decoupling (see \cite{GWZ,smallcapdec,improved_discrete_restriction,GGGHMW,gan_wu,gsw}). Here we observe that the technique, known as the \textit{high/low} argument, provides a way to adapt the proof of decoupling for the parabola to general convex curves in a way that, to our knowledge, the methods in \cite{BD,guo,li1,li2} are unable to. Specifically, we adapt the high/low argument used in Theorem \ref{thm: gmw} in order to extend the result of Theorem \ref{thm: BD} to obtain a universal decoupling estimate for all convex curves $\Gamma$ in $\R^2$. In \cite{gmw}, the authors discuss how to extend Theorem \ref{thm: gmw} to $C^2$ curves of bounded, positive curvature. Although the curvature of a general convex curve is defined almost everywhere, the set where the curvature does not exist can have large Hausdorff dimension. Our result applies to such low-regularity curves.
\subsection{Main result}
Let $\Gamma:=\big\{(t,\gamma(t)):t\in[0,1]\big\}$ be the graph of a convex function $\gamma:[0,1]\to\R$ satisfying 
\begin{equation}
    \label{eqn: gamma normalisation}
    0\leq\gamma_R'(0)\leq\gamma_L'(1)\leq 1.\footnote{Recall that convex functions are left and right-differentiable everywhere.}
\end{equation}
Before stating the main theorem, we introduce some notation and definitions.
\begin{definition}[Affine dimension]
    \label{def: caps}
    For $R\geq 1$, we define an $R^{-1}$-cap to be a collection of points on $\Gamma$ of the form $$B(\ell,R^{-1}):=\{\xi\in\Gamma: \text{dist}(\xi,\ell)<R^{-1}\},$$ where $\ell$ is any supporting line for $\Gamma$. We use $N(\Gamma,R^{-1})$ to denote the minimum number of $R^{-1}$-caps required to cover $\Gamma$. The \textit{affine Minkowski dimension} (or affine dimension, for short) of $\Gamma$ is defined as $$\kappa_\Gamma:=\limsup_{R\to\infty}\frac{\log N(\Gamma,R^{-1})}{\log R}.$$ This quantity was introduced by Seeger and Ziesler in \cite{SZ}, which was in turn motivated by a work of Bruna, Nagel, and Wainger \cite{BNW}.
\end{definition}
\begin{definition}
    [Vertical neighbourhoods]
    \label{def: vertical nbhds}
     For a subinterval $J\subseteq [0,1]$, define the vertical neighbourhood $$\mathcal{N}_{R^{-1}}(\Gamma,J):=\big\{\xi\in\R^2:\xi_1\in J,\,|\xi_2-\gamma(\xi_1)|\leq R^{-1}\big\}.$$ When $J=[0,1]$, we denote this by simply $\mathcal{N}_{R^{-1}}(\Gamma)$.
\end{definition}
For all $f\in\mathcal{S}(\R^2)$ satisfying $\supp(\widehat{f})\subseteq\mathcal{N}_{R^{-1}}(\Gamma)$, we define $f_J$ by 
 \begin{equation}
     \label{eqn: strip Fourier projection}
     \widehat{f}_J(\xi):=\widehat{f}(\xi)\cdot\chi_{J}(\xi_1).
 \end{equation}
 Note that $f_J$ is Fourier-supported in $\mathcal{N}_{R^{-1}}(\Gamma,J)$.
\begin{remark}
    \label{remark: neighbourhoods}
    The vertical neighbourhood $\mathcal{N}_{R^{-1}}(\Gamma)$ is comparable to the usual $R^{-1}$-neighbourhood of $\Gamma$.
\end{remark}
 \begin{definition}
     [Ideal partition]
     \label{def: ideal partition}
     Let $\Gamma$ be a convex curve as above, and let $\epsilon>0$ and $R\geq 1$. Let $\mathcal{J}$ be a collection of subintervals covering $[0,1]=\bigcup_{J\in\mathcal{J}}J$. We call $\mathcal{J}$ an ideal partition of $[0,1]$ (associated with $\Gamma$) if the following conditions hold.
     \begin{enumerate}[label=(J\arabic*)]
        \item For all $[a,b]\in\mathcal{J}$, we have $$(b-a)(\gamma_L'(b)-\gamma_R'(a))\leq 2R^{-1}.$$
        \item The number of intervals is bounded by $$\#\J\leq c_\epsilon R^{\epsilon}N(\Gamma,R^{-1}),$$ for some constant $c_\epsilon$ depending only on $\epsilon$.
        \item The interval lengths satisfy $$|J|\geq R^{-1}\quad \text{for all}\quad J\in\J.$$
    \end{enumerate}
 \end{definition}
 When $\Gamma=\Gamma_0$ is the parabola, an ideal partition of $[0,1]$ would be its decomposition into congruent subintervals of length $R^{-1/2}$. For each decomposing subinterval $[a,b]$, we would have $$(b-a)(\gamma_0'(b)-\gamma_0'(a))=2(b-a)^2=2R^{-1}.$$ The number of decomposing intervals is clearly $\#\J=R^{1/2}$, and it is easy to verify that the parabola has affine dimension $1/2$.
 
 Condition (J1) is analogous to \cite[(2.8.1)]{SZ}, as well as the `$P(r)$ property' appearing in \cite[\S 1.3]{yang}. It is a natural condition, and a consequence is that each vertical neighbourhood $\mathcal{N}_{R^{-1}}(\Gamma,J)$ is comparable to a $|J|\times R^{-1}$ rectangular box. These are colloquially known as \textit{canonical boxes}. These boxes tile $\mathcal{N}_{R^{-1}}(\Gamma)$, and condition (J2) says that the covering is essentially optimal in the number of rectangles. Unlike the parabolic case, in general these rectangles may have varying lengths. However, they lie in the range dictated by (J3), and so the number of `dyadically distinct' lengths is still $\log R$.
Our result is the following generalization of Bourgain--Demeter in the plane.
\begin{theorem}
    \label{thm: convex decoupling}
    For all $\epsilon>0$ and $R\geq 1$, there exists an ideal partition $\J$ (depending on $\epsilon,R$) and a constant $C_\epsilon\geq 1$, such that 
     \begin{equation}
         \label{eqn: main bound}
          \|f\|_{L^6(\R^2)}\leq C_\epsilon R^\epsilon\big(\sum_{J\in\mathcal{J}}\|f_J\|_{L^6(\R^2)}^2\big)^{1/2},
     \end{equation}
    for all $f\in\mathcal{S}(\R^2)$ satisfying $\supp(\widehat{f})\subseteq\mathcal{N}_{R^{-1}}(\Gamma)$.
\end{theorem}
When $\Gamma=\Gamma_0$ is the parabola, Theorem \ref{thm: convex decoupling} recovers Theorem \ref{thm: BD}. We remark that the constant $C_\epsilon$ depends only on $\epsilon$. Thus, \eqref{eqn: main bound} provides a uniform decoupling estimate for all $\Gamma$ satisfying \eqref{eqn: gamma normalisation}. There is an obvious extension of Theorem \ref{thm: convex decoupling} to piecewise convex curves, by applying \eqref{eqn: main bound} to each convex component. However, the resulting decoupling constant would not be uniform, but rather depend on the number of such components. For example, if $\gamma$ is a $d$-degree polynomial, then $\gamma''$ has at most $d-2$ zeros. Then $\Gamma$ can be partitioned into at most $d-1$ convex components. Applying Theorem \ref{thm: convex decoupling} to each such component recovers a result of Yang \cite[Theorem 1.4]{yang} which provides an estimate analogous to \eqref{eqn: main bound} for polynomial curves with constant $C_{\epsilon,d}$.
\begin{remark}
    Although the method used to prove Theorem \ref{thm: convex decoupling} follows the same scheme used to prove Theorem \ref{thm: gmw}, there are a few technical issues that prevent us from obtaining a $(\log R)^{c}$ bound similar to \eqref{eqn: GMW bound} for all convex curves $\Gamma$. As a result of this, our aim is to establish the weaker bound \eqref{eqn: main bound}. An upside to this is that a lot of the technical work done in \cite{gmw} to get the better dependence can be avoided here.
\end{remark}
\begin{remark}
    Unlike most results on decoupling, our partition $\J$ will be chosen depending on $\epsilon$. This is done as a matter of convenience to simplify the argument, and it is likely that this dependence can be removed by following a more careful analysis. We note, however, that this dependence poses no issue for practical applications of decoupling.
\end{remark}
\subsection{Motivation}
The question of extending Bourgain--Demeter's result \cite{BD} to general convex curves was posed by Christoph Thiele in a recent Oberwolfach meeting. Our result is also motivated by \cite{SZ}, where the authors generalized the classical square function estimate of C\'ordoba--Fefferman in the plane to all convex curves. We explain this presently. For the parabola $\Gamma=\Gamma_0$, a classical result of C\'ordoba and Fefferman states that $$\|f\|_{L^4(\R^2)}\lesssim\big\|\big(\sum_{J\in\J}|f_J|^2\big)^{1/2}\big\|_{L^4(\R^2)},$$ for all $f\in\mathcal{S}(\R^2)$ with $\supp{\widehat{f}}\subseteq \mathcal{N}_{R^{-1}}(\Gamma_0)$, where $\J$ is the partition of $[0,1]$ into congruent subintervals of length $R^{-1/2}$. An estimate of this form is known as a \textit{square function estimate}, and is related to decoupling via Minkowski's integral inequality. In \cite{SZ}, the authors extended the parabolic square function bound to general convex curves, proving $$\|f\|_{L^4(\R^2)}\lesssim(\log R)^{1/4}\big\|\big(\sum_{J\in\J}|f_J|^2\big)^{1/2}\big\|_{L^4(\R^2)},$$ for an ideal partition $\J$ of $[0,1]$. In a similar vein, our estimate \eqref{eqn: main bound} extends the classical parabolic decoupling bound \eqref{eqn: BD bound} to all convex curves.

\subsection{Features of the low-regularity problem} Decoupling theory for the parabola utilizes special properties of the curve such as non-vanishing curvature, homogeneity, and self-similarity. The lack of these desirable features in the general setting makes the problem harder to study at the level of arbitrary convex curves. At the same time, some low-regularity `fractals' satisfy decoupling bounds much stronger than the parabola. 
\subsubsection*{Regularity}
The curvature of a general convex curve can be very irregular, which we demonstrate by the following example.
\begin{example}\label{example: devil's staircase}
    Let $d\in(0,1)$, and let $\mu$ be a $d$-Ahlfors regular measure supported on a set of zero Lebesgue measure in $[0,1]$. Define the associated `devil's staircase' function $f_\mu$ by $$f_\mu(x):=\mu([0,x]),\quad\text{for all}\quad x\in [0,1].$$ The function $f_\mu$ is non-decreasing, and as such, it is the derivative of some convex function $\gamma_\mu$. Let us consider the curvature of the graph of $\gamma_\mu$ over $[0,1]$. Since the support of $\mu$ has zero Lebesgue measure, it follows that $\gamma_\mu''=0$ almost everywhere. Furthermore, \cite[Corollary 2]{falconer} asserts that the Hausdorff dimension of the set of points of non-differentiability of $f_\mu$ is $d^2$, so that $\gamma_\mu''$ does not exist in a set of Hausdorff dimension $d^2$, which can be arbitrarily close to $1$. 
\end{example}
The irregularity of the curvature means that the low-regularity problem cannot be treated using parabolic decoupling; and while all convex curves admit smooth approximations by strictly convex curves, the approximation is not uniform in the magnitude of the second derivatives.
\subsubsection*{Self-similarity} An essential feature of decoupling estimates of the form \eqref{eqn: main bound} is its self-similarity, where the $L^p$ norm appears on both sides of the inequality. On the left we have the $L^p$ norm of the original function $f$, and on the right we have the $L^p$ norms of the Fourier localised pieces $f_J$. This suggests that one could iterate these estimates. The affine self-similarity of the parabola plays an important role here. For $R>r$, let $\Theta(r)$ and $\Theta(R)$ respectively denote the collections of canonical boxes of width $r^{-1}$ and $R^{-1}$. If $\tau\in\Theta(r)$, then using \textit{parabolic rescaling}, we can roughly map $$\{\theta\in\Theta(R): \theta\subset\tau\}\mapsto\Theta(R/r).$$ In other words, the relative behaviour of the boxes between two scales $r^{-1}$ and $R^{-1}$ is the same as the behaviour of the full collection of boxes at scale $(R/r)^{-1}$. Most proofs of parabolic decoupling \cite{BD,guo,li2} leverage this fact in a multi-scale inductive argument. An analogous statement holds for curves of the form $(t,\gamma_\mu(t))$ associated to self-similar fractal measures $\mu$. Using a similar inductive argument, \cite[Theorem 1.1]{chang} obtains estimates that provide decoupling for these self-similar fractal curves. However, their work does not address curves that lack self-similarity e.g., when $\mu$ exhibits different behaviour at different scales.

A feature of the high/low argument is that it does not rely too heavily on induction on scales, and very minimal information between different scales is needed in order for the argument to work. 
\subsubsection*{Homogeneity}
The homogeneity of the parabola implies that all decomposing boxes have similar geometry, and that adjacent boxes are comparable, in the sense that if $\theta_1$ and $\theta_2$ are neighbouring boxes, then $\theta_1\subseteq c\cdot\theta_2$ and $\theta_2\subseteq c\cdot \theta_1$, for some suitable constant $c\lesssim 1$. These properties do not hold for general curves. Following is such an example.
\begin{example}
   Consider Example \ref{example: devil's staircase} when $\mu$ is the natural measure on the middle-third Cantor set, in which case $f_\mu=\gamma_\mu'$ is the Cantor function. It can be shown that for $R=3^{2K}$, the natural decomposition of the curve $(t,\gamma_\mu(t))$ consists of $R^{-1}$-rectangles of lengths $3^{-1},3^{-2},\dots,3^{-K+1},$ and $3^{-K}$ (cf. \cite[\S 3.2]{roy}). Adjacent rectangles $\theta_1$ and $\theta_2$ of length $3^{-K}$ will have a direction separation of $2^{-K}$. Consequently, we have $\theta_1\subseteq c\cdot\theta_2$ only if $c\gtrsim (3/2)^K=R^{\frac{1}{2}(1-\frac{\ln 2}{\ln 3})}$. Furthermore, the number of such boxes violating the comparability condition is about $2^K=R^{\frac{1}{2}\frac{\ln 2}{\ln 3}}$.
\end{example}
As a result, the analysis in our general setting is quite sensitive to the choice of the canonical boxes, which we define in a particular way (see \S\ref{sec: box definition}). To fix the issue with varying side lengths, one can dyadically pigeonhole. However, due to the multi-scale nature of the argument, we need to pigeonhole at every scale. By working with a coarse sequence of scales (see \S\ref{sec: intermediate scales}), we can pigeonhole in the lengths at every scale simultaneously, without incurring significant losses.
\subsubsection*{$\Lambda(p)$ phenomenon} Certain low-regularity curves exhibit stronger orthogonality, which lead to better bounds for decoupling over the parabola and our universal result. Bourgain \cite{Bourgain_lambdap} famously showed that a generic sparse subset $A\subset\{1,2,\dots,n\}$ of size $|A|=[n^{2/p}]$ satisfies the discrete restriction bound $$\|\sum_{k\in A}a_ke^{2\pi ikx}\|_{L^p([0,1])}\leq c_p\big(\sum_{k\in A}|a_k|^2\big)^{1/2}.$$
In a similar vein, one can show that a generic self-similar fractal curve of affine dimension $1/p$ satisfies a decoupling bound of the form $$\|f\|_{L^{3p}(\R^2)}\leq c_{p,\epsilon} R^\epsilon\big(\sum_{J\in\J}\|f_J\|_{L^{3p}(\R^2)}^2\big)^{1/2}.$$ These examples were obtained in our previous work \cite{roy}, and uses a decoupling theorem of Chang et al. \cite{chang} for self-similar fractals. We give a more detailed discussion in \S\ref{sec: applications and examples}.
\subsection{Notation} Unless the base is specified, all logarithms are taken base $2$. For $X,Y\in\mathbb{R}$, if there exists an absolute constant $C>0$ such that $X\leq CY$, we shall write $X\lesssim Y$ or equivalently $Y\gtrsim X$. If this constant $C$ is not absolute, but depends on some parameters of our problem, say $\alpha,\beta,\gamma$, we shall specify this by writing $X\lesssim_{\alpha,\beta,\gamma}Y$ or equivalently $Y\gtrsim_{\alpha,\beta,\gamma} X$. We also write $X\approx Y$ when both $X\lesssim Y$ and $X\gtrsim Y$ hold. Similarly, we write $X\approx_{\alpha,\beta,\gamma}Y$ when both $X\lesssim_{\alpha,\beta,\gamma} Y$ and $X\gtrsim_{\alpha,\beta,\gamma} Y$ hold. Furthermore, $R\gg 1$ shall stand for `$R$ sufficiently large'. For example, the statement $$``P(R)\quad\text{holds for all}\quad R\gg 1"$$ should be interpreted as $$``\exists R_0>0\quad\text{such that}\quad P(R)\quad\text{holds for all}\quad R\geq R_0".$$ Similarly, $R\gg_\epsilon 1$ shall stand for `$R$ sufficiently large depending on $\epsilon$'. For example, the phrase $$``P(R)\quad\text{holds for all}\quad R\gg_\epsilon 1"$$ should be interpreted as $$``\exists R_\epsilon>0\quad\text{such that}\quad P(R)\quad \text{holds for all}\quad R\geq R_\epsilon".$$ 
Let $\B$ denote the unit ball $B(0,1)$ in $\R^2$. For a convex set $D$, we denote its centroid by $c(D)$. For $s>0$, we let $s\cdot D$ denote the $s$-dilate of $D$ about its centroid. That is, $$s\cdot D:=c(D)+s(D-c(D)).$$ For instance, $2\cdot\B=B(0,2)$.
\subsection*{Acknowledgments} I would like to thank my advisor Jonathan Hickman for providing helpful feedback on the project.
\section{Preliminaries}\label{sec: prelim}
 Let $\Gamma$ be the graph of a convex function $\gamma$ satisfying \eqref{eqn: gamma normalisation}, and $R\geq 1$ be a large parameter. Since the quantity $\epsilon$ appearing in the statement of Theorem \ref{thm: convex decoupling} is arbitrary, we fix an $\epsilon\in (0,1/2)$, and construct an ideal partition $\mathcal{J}$ such that \eqref{eqn: main bound} holds with $100\epsilon$. Without loss of generality, we also assume that $1/\epsilon\in\mathbb{N}$.

Due to a technical limitation of the high/low argument, we must partition our original curve (depending on $R$ and $\epsilon$) into components that are more amenable to the method. We describe these curves below. First we introduce the following functions. For the graph $\Gamma$ of a convex function $\gamma$ over $I:=[a,b]$, we define $$\A(\Gamma,I):=\frac{\gamma_L'(b)-\gamma_R'(a)}{1+\gamma_L'(b)\gamma_R'(a)},\quad\cal(\Gamma,I):=|\Gamma(b)-\Gamma(a)|,\quad\text{and}\qquad\W(\Gamma,I):=\cal(\Gamma,I)\A(\Gamma,I).$$
With this notation, we define the following.
\begin{definition}
    [Admissible curves]
    \label{def: admissible pair}
    Let $\Gamma$ be the graph of a convex function $\gamma$ over $I=[a,b]$. We call $(\gamma,I)$ an \textit{admissible pair} and $\Gamma$ an \textit{admissible curve} if the following hold:
    \begin{enumerate}[label=(C\arabic*)]
        \item $\A(\Gamma,I)\leq R^{-2\epsilon}$;
        \item $\cal(\Gamma,I)\geq R^{2\epsilon}$;
        \item $\gamma(a)=\gamma_R'(a)=0$;
        \item $\W(\Gamma,I)\leq 1$.
    \end{enumerate}
   By a slight abuse of notation, we shall frequently denote $\Gamma(t):=(t,\gamma(t))$. We record the following consequence of (C1) and (C3): 
    \begin{equation}
        \label{eqn: admissible derivative bound}
        \|\gamma_L'\|_{L^\infty(I)},\|\gamma_R'\|_{L^\infty(I)}\leq R^{-2\epsilon}.
    \end{equation}
\end{definition}
    Recall that an affine transformation $\mathcal{L}:\R^2\to\R^2$ is called a similarity transformation of ratio $r\geq 0$ (or similarity, for short), if $$\mathcal{L}(\xi)=p+rO\xi\quad\text{for all}\quad \xi\in\R^2,$$ for some point $p\in\R^2$ and some orthogonal matrix $O$. Two sets $A\subseteq\R^2$ and $B\subseteq\R^2$ are called similar if $B=\cal(A)$ for some similarity $\cal$.

\subsection{A partition of the original curve}\label{sec: partition of curve}  Let $\Gamma$ be the graph of a convex function $\gamma:[0,1]\to \R$, satisfying \eqref{eqn: gamma normalisation}. We describe a partition of $\Gamma$ below.

Begin by choosing a finite sequence of points $\{s_p\}_{p=0}^P\subset[0,1]$ as follows. Set $s_0:=0$. For each $s_p$, note that the function $\A(\Gamma,[s_p,s])$ is increasing in $s\geq s_p$ due to the convexity of $\Gamma$. We iteratively define $$s_{p+1}:=\max\big\{s_p+R^{-1},\sup\{s\geq s_p:\A(\Gamma,[s_p,s])\leq R^{-2\epsilon}\}\big\}.$$
The process terminates after finitely many steps (see \eqref{eqn: P bound}), and we set $s_P:=1$. This gives the initial partition
\begin{equation}
    \label{eqn: [0,1] partition}
    [0,1]=\bigsqcup_{p=0}^{P-1}I_p\quad\text{and}\quad \Gamma=\bigsqcup_{p=0}^{P-1}\Gamma\vert_{I_p},
\end{equation}
where $I_p:=[s_p,s_{p+1}]$, and $\Gamma\vert_{I_p}:=\{\Gamma(t):t\in I_p\}$. Let us separate the indices into the sets $$\mathcal{P}_1:=\{p:s_{p+1}=s_p+R^{-1}\},\quad\text{and}\quad \mathcal{P}_2:=\{p:s_{p+1}>s_p+R^{-1}\}.$$
By construction, we have 
\begin{equation}
    \label{eqn: I_p length bound} |I_p|\geq R^{-1}\quad\text{for all}\quad p\neq P.
\end{equation}
Furthermore, we have
\begin{align}
    \A(\Gamma,I_p)\leq R^{-2\epsilon}\quad&\,\text{for all}\quad p\in\mathcal{P}_2\label{eqn: I_p total curvature control}\\
    \A(\Gamma,[s_p,s_{p+1}+\eta])>R^{-2\epsilon}\quad&\text{for all}\quad \eta>0\quad\text{and all}\quad p<P-1.\label{eqn: plus eta}
\end{align}
Indeed, \eqref{eqn: I_p total curvature control} follows from the left-continuity of $\gamma_L'$, and \eqref{eqn: plus eta} follows from the definition of suprema. We observe that the number of partitioning intervals satisfies $P\leq 2R^{2\epsilon}$. Indeed, by \eqref{eqn: plus eta} and \eqref{eqn: gamma normalisation} we have 
\begin{align*}
    (P-1)R^{-2\epsilon}<\sum_{p=0}^{P-2}\gamma_L'(s_{p+1}+\eta)-\gamma_R'(s_p)=\sum_{p=0}^{P-2}\gamma_L'(s_{p+1}+\eta)-\gamma_R'(s_{p+1})+\sum_{p=0}^{P-2}\gamma_R'(s_{p+1})-\gamma_R'(s_p).
\end{align*}
The second term on the right-hand side above is a telescoping series, and using the convexity of $\gamma$ and \eqref{eqn: gamma normalisation}, we find  $$\sum_{p=0}^{P-2}\gamma_R'(s_{p+1})-\gamma_R'(s_p)=\gamma_R'(s_{P-1})-\gamma_R'(0)\leq \gamma_L'(1)-\gamma_R'(0)\leq 1.$$ 
For the remaining series, we observe the following consequence of convexity $$\gamma_R'(s_{p+1})\leq\gamma_L'(s_{p+1}+\eta)\leq\gamma_R'(s_{p+1}+\eta).$$
Since the right-derivative of convex functions are right-continuous, it follows that $\gamma_L'(s_{p+1}+\eta)\to \gamma_R'(s_{p+1})$ as $\eta\to 0^+$. Thus, by choosing $\eta>0$ sufficiently small, we get 
\begin{equation}
    \label{eqn: P bound}
    (P-1)R^{-2\epsilon}\leq 2\quad\text{so that}\quad P\leq 3R^{2\epsilon},\quad\text{for all}\quad R\geq 1.
\end{equation} 
For each $p\in\mathcal{P}_2$, define the similarity transformation $\cal_p$ by 
\begin{equation}
    \label{eqn: cal p def}
    \cal_p(\xi):=\frac{r_p}{\sqrt{1+\gamma_R'(s_p)^2}}\begin{pmatrix}
   1&\gamma_R'(s_p)\\-\gamma_R'(s_p)&1
\end{pmatrix}(\xi-\Gamma(s_p)),
\end{equation}
where $r_p:=R^{2\epsilon}\cal(\Gamma,I_p)^{-1}$.
Under $\cal_p$, the arc $\Gamma\vert_{I_p}$ is mapped onto a new convex curve $\Gamma_p$, which is the graph of a convex function $\gamma_p$ over an interval $J_p$ with left end-point at $0$.

\medskip\noindent\textbf{Claim:} $\Gamma_p$ is admissible for all $p\in\mathcal{P}_2$, and satisfies (C2) with an equality.

\medskip\noindent\textit{Verifying} (C1):  As $\cal_p$ preserves angles, it follows that $$\A(\Gamma_p,J_p)=\A(\Gamma,I_p)\leq R^{-2\epsilon},$$ where the last inequality follows from \eqref{eqn: I_p total curvature control}. 

\medskip\noindent\textit{Verifying} (C2): As $\cal_p$ scales uniformly by $r_p$, we have $$\cal(\Gamma_p,J_p)=r_p\cal(\Gamma,I_p)=R^{2\epsilon},$$ by the choice of $r_p$. Thus, (C2) holds with equality.

\medskip\noindent\textit{Verifying} (C3): Recall that $\Gamma_p$ is the graph of a convex function $\gamma_p$ over the interval $J_p$ with left end-point at $0$. Clearly $$\Gamma_p(0)=\cal_p(\Gamma(s_p))=(0,0),\quad\text{so that}\quad \gamma_p(0)=0.$$
Also, $\cal_p$ maps lines with direction $(1,\gamma_R'(s_p))$ onto lines with direction $(1,0)$. Since derivatives commute with linear transformations, it follows that $\gamma_{p,R}'(0)=0$.

\medskip\noindent\textit{Verifying} (C4): We have $$\W(\Gamma_p,J_p)=\cal(\Gamma_p,J_p)\A(\Gamma_p,J_p)\leq 1,$$ by combining (C1) and (C2) from above.

For the curve appearing in Theorem \ref{thm: convex decoupling}, \eqref{eqn: [0,1] partition} gives an initial partition of $[0,1]$. We obtain the ideal partition $\mathcal{J}$ by refining \eqref{eqn: [0,1] partition}. This is done through a multi-scale algorithm, that we present below.

\subsection{A multi-scale algorithm for admissible curves}\label{sec: box definition} 
Let $\Gamma$ be an admissible curve in the sense of Definition \ref{def: admissible pair} associated with the pair $(\gamma,I)$. 
We provide an algorithm that simultaneously produces a fine-scale partition of $I=[a,b]$, as well as a family of canonical boxes. This is done iteratively, first using a single-scale algorithm to produce a coarse-scale partition of $I$, along with an associated family of wide boxes. The partition is gradually refined, producing thinner boxes at each step using a multi-scale algorithm. This involves working with many intermediate scales, that we introduce below.

In light of the decomposition \eqref{eqn: [0,1] partition}, applying the present algorithm to each admissible curve $\Gamma_p$ followed by applying $\cal_p^{-1}$ will produce the ideal partition $\mathcal{J}$ as well as the canonical boxes appearing in Theorem \ref{thm: convex decoupling}.
\subsubsection{Introducing intermediate scales}\label{sec: intermediate scales}
Similar to previous results on decoupling mentioned in \S\ref{sec: intro}, our argument involves working with a sequence of intermediate scales between $R^{-1}$ and $1$.
 The authors of \cite{gmw} worked with a fine set of scales $$R^{-1}=(\log R)^{-cN}<(\log R)^{-c(N-1)}<\dots<(\log R)^{-c}<1,\quad\text{where}\quad N=c^{-1}\log R/\log\log R,$$ which is essential in order to achieve the improved bounds. As mentioned in \S\ref{sec: intro}, we are aiming for the weaker sub-polynomial loss, and as such choose to work with the following coarse set of scales
 \begin{equation}
     \label{eqn: intermediate scales}
     R^{-1}=R^{-N\epsilon}<R^{-(N-1)\epsilon}<\dots<R^{-\epsilon}<1,\quad\text{where}\quad N=1/\epsilon.
 \end{equation}
  This reduces the number of steps to $1/\epsilon$, making it completely independent of the parameter $R$. 
\begin{notation}\label{not: lambda 0 and R_k}
Denote $\lambda_0:=R^{2\epsilon}$, and for $k\in\{0,1,\dots,N\}$, denote $R_k:=R^{k\epsilon}$. For $t\in I$, we define the directions $$\mathbf{t}(t):=\frac{\big(1,\gamma_R'(t)\big)}{\|\big(1,\gamma_R'(t)\big)\|},\quad\text{and}\quad \mathbf{n}(t):=\frac{\big(\gamma_R'(t),-1\big)}{\|\big(\gamma_R'(t),-1\big)\|}.$$
\end{notation}
\subsubsection{Single-scale algorithm}\label{sec: single-scale algorithm} The single-scale algorithm is similar to the scheme from \cite[\S2]{SZ}. Let $\Gamma$ be an admissible curve associated to the pair $(\gamma, I)$ for $I=[a,b]$. We construct a family of subintervals $[s^\nu,u^\nu]\subseteq[a,b]$ as follows. 

\medskip\noindent\underline{Step 0:} Set $s^0:=a$. 

\medskip\noindent\underline{Case I:} If $\W(\Gamma,[s^0,b])\leq R^{-\epsilon}$, we set $u^0:=b$, and the algorithm terminates.

\medskip\noindent\underline{Case II:} If $\W(\Gamma,[s^0,b])> R^{-\epsilon}$, we set $$u^0:=\sup\{s\geq s^0:\W(\Gamma,[s^0,s])\leq R^{-\epsilon}\},$$ and proceed to Step 1.

\medskip\noindent\underline{Step $\nu$:} At the end of step $\nu-1$, we have defined the points $s^{\nu-1},u^{\nu-1}$. 

\medskip\noindent\underline{Case I:} If $\W(\Gamma, [u^{\nu-1},b])\leq R^{-\epsilon}$, we set 
\begin{equation}
    \label{eqn: right box def}
    s^{\nu}:=\inf\{s\leq u^{\nu-1}:\W(\Gamma,[s,b])\leq R^{-\epsilon}\} \quad\text{and}\quad u^{\nu}:=b,
\end{equation}
and the algorithm terminates.

\medskip\noindent\underline{Case II:} If $\W(\Gamma, [u^{\nu-1},b])> R^{-\epsilon}$, we set, and 
\begin{equation}
    \label{eqn: left box def}
    s^{\nu}:=u^{\nu-1}\quad\text{and}\quad u^{\nu}:=\sup\{s\geq s^{\nu}:\W(\Gamma,[s^{\nu},s])\leq R^{-\epsilon}\},
\end{equation}
 and proceed to step $\nu+1$.

Our selection process terminates after finitely many steps (see Lemma \ref{lemma: number of boxes}) with the last pair of points $s^{\nu_1},u^{\nu_1}$. Note that $s^0=a$ and $u^{\nu_1}=b$.
For each $0\leq\nu\leq\nu_1$, we define the box 
\begin{equation}
\label{eqn: sigma def}
\sigma^\nu:=\Gamma(s^{\nu})+\big(
    \mathbf{t}(s^{\nu})\;\,\vert\;\,\mathbf{n}(s^{\nu})
\big)\begin{pmatrix}
    \len(\sigma^\nu)&0\\0&\wid(\sigma^\nu)
\end{pmatrix}\big([0,1]\times[-1,1]\big),
\end{equation}
where \begin{equation}
    \label{eqn: len wid sigma def}
    \len(\sigma^\nu):=\cal(\Gamma,[s^\nu,u^\nu]),\quad\text{and}\quad \wid(\sigma^\nu):=R^{-\epsilon}.
\end{equation}
Each $\sigma^\nu$ is a rectangle of length $\len(\sigma^\nu)$ and width $2\wid(\sigma^\nu)$. Denote $\calt_1(\Gamma):=\{\sigma^\nu\}_{\nu=0}^{\nu_1}$. If $\nu_1=0$ (that is, the algorithm terminates at Step 0), we call $\sigma^0$ an \textit{exceptional box}. If $\nu_1>0$ (that is, the algorithm proceeds to Step 1), we call every $\sigma^\nu$ a \textit{typical box}. Among the typical boxes, we call $\sigma^{\nu_1}$ a \textit{right box}, and call every other $\sigma^\nu$ a \textit{left box}.\footnote{These terms reflect how the intervals $[s^\nu,u^\nu]$ are defined.} We denote $I\vert_{\sigma^\nu}:=[s^\nu,u^\nu]$ and call it the \textit{interval associated with} $\sigma^\nu$.
For each $\sigma\in\calt_1(\Gamma)$, we also define $\Gamma\vert_{\sigma}$ as the arc of $\Gamma$ given by $$\Gamma\vert_{\sigma}:=\{(t,\gamma(t)):t\in I\vert_{\sigma}\}.$$
We also define the similarity transformation 
\begin{equation}
    \label{eqn: cal sigma def}
    \cal_{\sigma}:=(\wid\sigma)^{-1}(\mathbf{t}(s)\;\vert\;\mathbf{n}(s))^{-1}(\xi-\Gamma(s)),
\end{equation}
 where $s$ denotes the left endpoint of $I\vert_{\sigma}$. This definition will be useful for the multi-scale algorithm. 
\subsubsection{Properties of the single-scale decomposition}\label{sec: single scale properties}
The following properties are satisfied by the collection $\calt_1(\Gamma)$ constructed above.
\begin{lemma}\label{lemma: wid/len sigma}
    For all $\sigma\in\calt_1(\Gamma)$, we have
    \begin{equation}
        \label{eqn: wid/len formula}
        \W(\Gamma,I\vert_\sigma)\leq R^{-\epsilon}.
    \end{equation}
    Furthermore, if $I\vert_\sigma=[s,u]$, then for all $\eta>0$ sufficiently small, we have
    \begin{align}
    \label{eqn: sigma angle sep}
    \nonumber \W(\Gamma,[s,u+\eta])> R^{-\epsilon}\quad&\text{if $\sigma$ is a left box};\\ 
   \W(\Gamma,[s-\eta,u])>R^{-\epsilon}\quad&\text{if $\sigma$ is a right box}.
\end{align}
\end{lemma}
\begin{proof}
  For the proof of this lemma, we shall use the elementary fact that the left-derivative of a convex function is left-continuous and similarly the right-derivative is right-continuous.
  
  We start with the top inequality.  First suppose the algorithm terminates at Step 0. Then $\sigma=\sigma^0$ is the only box, and we have $\W(\Gamma,[s^0,u^0])\leq R^{-\epsilon}$, which is the desired relation. Now suppose the algorithm proceeds to Step 1.  Let $\sigma=\sigma^\nu$. If $\sigma^\nu$ is a left box, we consider all $s\geq s^\nu$ satisfying $\W(\Gamma,[s^\nu,s])\leq R^{-\epsilon}$. By \eqref{eqn: left box def}, $u^\nu$ is the supremum of all such $s$, so by taking limit as $s$ approaches $u^\nu$ from below and using the left-continuity of $\gamma_L'$, we get $$R^{-\epsilon}\geq \lim_{s\to {u^{\nu}}^-}\W(\Gamma,[s^\nu,s])=\W(\Gamma,[s^\nu,u^\nu])=\W(\Gamma,I\vert_\sigma).$$ If $\sigma$ is a right box, we consider all $s\leq u^{\nu-1}$ satisfying $\W(\Gamma,[s,u^\nu])\leq R^{-\epsilon}$. By \eqref{eqn: right box def}, $s^\nu$ is the infimum of all such $s$, so by taking limit as $s$ approaches $s^\nu$ from above and using the right-continuity of $\gamma_R',$ we get $$R^{-\epsilon}\geq \lim_{s\to {s^\nu}^+}\W(\Gamma,[s,u^\nu])=\W(\Gamma,[s^\nu,u^\nu])=\W(\Gamma,I\vert_\sigma).$$
 For the first inequality in \eqref{eqn: sigma angle sep} we use \eqref{eqn: left box def} and the definition of supremum. For the second inequality we use \eqref{eqn: right box def} and the definition of infimum. 
\end{proof}
\begin{lemma}\label{lemma: sigma length}
For all $\sigma\in\calt_1(\Gamma)$, we have
\begin{equation}
    \label{eqn: tau 1 length relative range}
     R^{-\epsilon}\cal(\Gamma,I)\leq\cal(\Gamma,I\vert_{\sigma})\leq \cal(\Gamma,I). 
\end{equation}
\end{lemma}
\begin{proof}
If the algorithm stops at Step 0, then $\sigma=\sigma^0, I\vert_{\sigma}=I$, and \eqref{eqn: tau 1 length relative range} follows immediately. 

Now suppose the algorithm proceeds to Step 1. For all $\sigma$ we have $I\vert_{\sigma}\subseteq I$, so the upper bound follows from the monotonicity of the function $\cal(\Gamma,\cdot)$, which in turn follows from convexity and the conditions (C1)-(C4). For the lower bound we argue as follows.
Let $I\vert_\sigma=[s,u]$. First, consider the case that $\sigma$ is a left box. There exists $\eta>0$ satisfying $u+\eta< b$. By the first inequality in \eqref{eqn: sigma angle sep}, we have $$\W(\Gamma,[s,u+\eta])>R^{-\epsilon}.$$ But $\W(\Gamma,[s,u+\eta])=\A(\Gamma,[s,u+\eta])\cal(\Gamma,[s,u+\eta])$, and by convexity and (C4), we have $$\A(\Gamma,[s,u+\eta])\leq\A(\Gamma,I)\leq \cal(\Gamma,I)^{-1}.$$ 
Combining both, we have $$\cal(\Gamma,[s,u+\eta])>R^{-\epsilon}\cal(\Gamma,I).$$ Letting $\eta\to 0^+$ and using the continuity of $\Gamma$, we get the desired lower bound in this case. 

We argue similarly when $\sigma$ is a right box, this time observing that there exists $\eta>0$ satisfying $s-\eta>a$, and using the second inequality in \eqref{eqn: sigma angle sep}.
\end{proof}
The following corollary says that the boxes $\sigma$ have some eccentricity, which is a technical requirement for the high/low argument (specifically, Lemma \ref{lemma: high lemma}).
\begin{lemma}
    \label{lemma: sigma eccentricity}
      For each $\sigma\in\calt_1(\Gamma)$, we have 
       \begin{equation}
      \label{eqn: sigma eccentricity}
      \len(\sigma)\geq R^{2\epsilon}\wid(\sigma).
  \end{equation}
\end{lemma}
\begin{proof}
By Lemma \ref{lemma: sigma length} and \eqref{eqn: len wid sigma def} we have $$\len(\sigma)=\cal(\Gamma,I\vert_{\sigma})\geq R^{-\epsilon}\cal(\Gamma,I)=\wid(\sigma)\cal(\Gamma,I).$$ But by (C2), we have $\cal(\Gamma,I)\geq R^{2\epsilon}$, and so the result follows.
\end{proof}
The following gives a bound on the number of $\sigma\in\calt_1(\Gamma)$. This follows directly from \cite[Lemma 2.3 (i)]{SZ}, but we include the proof for completeness.
\begin{lemma}\label{lemma: number of R1 boxes}
    For all admissible curves $\Gamma$, we have $$\#\calt_1(\Gamma)\leq 32R^{\epsilon/2}.$$ 
\end{lemma}
\begin{proof}
    Recall the relation $$\W(\Gamma,\cdot)=\A(\Gamma,\cdot)\cal(\Gamma,\cdot).$$ 
    If the single-scale algorithm stops at Step 0, then $\#\calt_1(\Gamma)=1$, and we are done. Suppose we are not in this case. For all $0\leq\nu\leq\nu_1-3$, $\sigma^\nu$ is a left box with $s^{\nu+2}>s^{\nu+1}=u^\nu$. Consequently, by the first inequality in \eqref{eqn: sigma angle sep}, we have $$R^{-\epsilon}<\W(\Gamma,[s^\nu,s^{\nu+2}])\quad\text{for all}\quad 0\leq\nu\leq\nu_1-3.$$
    Applying \eqref{eqn: admissible derivative bound} gives $$\W(\Gamma,[s^\nu,s^{\nu+2}])\leq 2(s^{\nu+2}-s^\nu)(\gamma_L'(s^{\nu+2})-\gamma_R'(s^\nu)).$$
    Combining the above two and using Cauchy--Schwarz, it follows that $$(\nu_1-2)R^{-\epsilon/2}\leq 2\sum_{\nu=0}^{\nu_1-3}(s^{\nu+2}-s^\nu)^{1/2}(\gamma_L'(s_{\nu+2})-\gamma_R'(s^\nu))^{1/2}\leq2\big(\sum_{\nu=0}^{\nu_1-3}s^{\nu+2}-s^\nu\big)^{1/2}\big(\sum_{\nu=0}^{\nu_1-3}\gamma_L'(s^{\nu+2})-\gamma_R'(s^\nu)\big)^{1/2}.$$
    Now $$\sum_{\nu=0}^{\nu_1-3}s^{\nu+2}-s^\nu\leq 2(b-a),\quad\text{and}\quad \sum_{\nu=0}^{\nu_1-3}\gamma_L'(s^{\nu+2})-\gamma_R'(s^\nu)\leq 2(\gamma_L'(b)-\gamma_R'(a)).$$ Since $(\gamma,[a,b])$ is an admissible pair, by \eqref{eqn: admissible derivative bound} and (C4) we have $(b-a)(\gamma_L'(b)-\gamma_R'(a))\leq 2$. It follows that $\nu_1\leq 2+16R^{\epsilon/2}$, and the result follows since $R\geq 1$.
\end{proof}
\begin{lemma}\label{lemma: sigma containment}
For all $\sigma\in\calt_1(\Gamma)$, we have $$\sigma\subseteq 2\cdot(I\times [-1,1]),$$ provided $R\geq 2^{1/\epsilon}$.
\end{lemma}
\begin{proof}
Let $I=[a,b]$ and $I\vert_{\sigma}=[s,u]$. Then by \eqref{eqn: sigma def} we have $$\sigma=\Gamma(s)+\big(
    \mathbf{t}(s)\;\,\vert\;\,\mathbf{n}(s)
\big)\begin{pmatrix}
    \len(\sigma)&0\\0&\wid(\sigma)
\end{pmatrix}\big([0,1]\times[-1,1]\big).$$ Our goal is to show that $\sigma\subseteq[a-\frac{b-a}{2},b+\frac{b-a}{2}]\times [-2,2]$. Let $\xi\in\sigma$. Then 
\begin{align}
    \label{eqn: point in tau_{1,nu}}
   \nonumber \xi_1&=s+\rho_1|\Gamma(u)-\Gamma(s)|-\rho_2R^{-\epsilon}\gamma_R'(s)\\
    \xi_2&=\gamma(s)+\rho_1|\Gamma(u)-\Gamma(s)|\gamma_R'(s)+\rho_2R^{-\epsilon},
\end{align}
for some $0\leq\rho_1\leq 1$, and $|\rho_2|\leq 1$. By \eqref{eqn: admissible derivative bound}, we have 
\begin{equation}
    \label{eqn: sigma containment derivative bound}
    |\Gamma(u)-\Gamma(s)|\leq (1+R^{-4\epsilon})^{1/2}(u-s)\quad\text{and}\quad \gamma_R'(s)\leq R^{-2\epsilon}.
\end{equation} 
Using \eqref{eqn: sigma containment derivative bound} in the first equation in \eqref{eqn: point in tau_{1,nu}}, we get 
 $$s-2^{-3}\leq \xi_1\leq u+(1/4)(u-s)+2^{-3}\quad\text{for all}\quad R\geq 2^{1/\epsilon}.$$ 
 Now $[s,u]\subseteq[a,b]$, and by  \eqref{eqn: admissible derivative bound} and (C2), we have $$b-a\geq (1/2)|\Gamma(b)-\Gamma(a)|\geq (1/2)R^{2\epsilon}.$$
 It follows that
 \begin{equation}
    \label{eqn: xi_1 containment}
    \xi_1\in [a-\frac{b-a}{2},b+\frac{b-a}{2}].
\end{equation}
 
On the other hand, using \eqref{eqn: sigma containment derivative bound} in the second equation in \eqref{eqn: point in tau_{1,nu}}, we find
\begin{equation}
    \label{eqn: |xi_2| bound}
    |\xi_2|\leq \gamma(s)+(1+R^{-4\epsilon})^{1/2}(u-s)\gamma_R'(s)+R^{-\epsilon}\leq (1+R^{-4\epsilon})^{1/2}\big(\gamma(s)+(u-s)\gamma_R'(s)\big)+R^{-\epsilon},
\end{equation}
since $\gamma\geq 0$.
 By convexity and (C3), we have $$\gamma(s)+(u-s)\gamma_R'(s)\leq \gamma(u)\leq \gamma(a)+(u-a)\gamma_L'(u)=\gamma(a)+(u-a)(\gamma_L'(u)-\gamma_R'(a)).$$  
 Since $u\leq b$, by convexity and \eqref{eqn: admissible derivative bound} we have $$(u-a)(\gamma_L'(u)-\gamma_R'(a))\leq (b-a)(\gamma_L'(b)-\gamma_R'(a))\leq (1+R^{-4\epsilon})\W(\Gamma,I).$$
 Using the above with (C4) in \eqref{eqn: |xi_2| bound}, we get $$|\xi_2|\leq (1+R^{-4\epsilon})^{3/2}+R^{-\epsilon}.$$ It follows that for all $R\geq 2^{1/\epsilon}$ we have 
\begin{equation}
    \label{eqn: xi_2 containment}
    \xi_2\in[-2,2].
\end{equation}
Combining \eqref{eqn: xi_1 containment} and \eqref{eqn: xi_2 containment}, we get that $\xi\in 2\cdot (I\times [-1,1])$ for all $R\geq 2^{1/\epsilon}$. 
\end{proof}

\subsubsection{Multi-scale algorithm}\label{sec: multi-scale algorithm}
We now use the single-scale algorithm recursively to construct box decompositions at all scales $R_k^{-1}$, and the associated intervals that provide finer partitions of $I=[a,b]$. 

\medskip\noindent\underline{Step $0$:}
Let us define $\len(\tau_0):=\cal(\Gamma,I)$, and $\wid(\tau_0):=1$. Define the unit-scale box $\tau_0:=[a,a+\len(\tau_0)/\wid(\tau_0)]\times[-1,1]$ and $\calt_0(\Gamma):=\{\tau_0\}$.

\medskip\noindent\underline{Step $1$:} We use the single-scale algorithm to define $\calt_1(\Gamma;\tau_0):=\calt_1(\Gamma)$, where $\calt_1(\Gamma)$ is the decomposition obtained in \S\ref{sec: single-scale algorithm}. The associated intervals $(I\vert_{\tau_1})_{\tau_1\in\calt_1(\Gamma)}$ produce an initial partition of $I$.

\medskip\noindent\underline{Step $k+1$:} At the end of Step $k$, we will have obtained $\calt_k(\Gamma)$, which is the collection of rectangles of the form
\begin{equation}
    \label{eqn: tau_k form}
    \tau_k=\Gamma(t_k)+\big(\mathbf{t}(t_k)\;\vert\;\mathbf{n}(t_k)\big)\begin{pmatrix}
        \len(\tau_k)&0\\0&\wid(\tau_k)
    \end{pmatrix} \big([0,1]\times[-1,1]\big), 
\end{equation} for some $t_k\in I$, and $\wid(\tau_k):=R_k^{-1}.$ Let $I\vert_{\tau_k}=[t_k,v_k]\subseteq[a,b]$ be the interval associated with $\tau_k$, which is defined at Step $k$. Let $\Gamma\vert_{\tau_k}$ be the arc of $\Gamma$ defined by 
\begin{equation}
    \label{eqn: Gamma vert tau k def}
    \Gamma\vert_{\tau_k}:=\{\Gamma(t):t\in I\vert_{\tau_k}\}.
\end{equation}
 Further suppose that the following conditions are met at Step $k$:
\begin{enumerate}
    [label=\Roman*(k).]
    \item For all $\tau_{k}\in\calt_{k}(\Gamma)$ we have 
        $$\len(\tau_{k})=\cal(\Gamma,I\vert_{\tau_k})\geq R_{k-2}^{-1}.$$
    \item For all $\tau_k\in\calt_k(\Gamma)$, we have $$\W(\Gamma,I\vert_{\tau_k})\leq R_k^{-1}.$$
\end{enumerate}
Now we define the $R_{k+1}^{-1}$ boxes as follows. Fix $\tau_k\in\calt_k(\Gamma)$, and denote $\upsilon_k:=[0,\len(\tau_k)/\wid(\tau_k)]\times[-1,1]$. Define the similarity transformation \begin{equation}
    \label{eqn: cal tau k def}
    \cal_{\tau_k}(\xi):=(\wid\tau_k)^{-1}\big(\mathbf{t}(t_k)\;\vert\;\mathbf{n}(t_k)\big)^{-1}(\xi-\Gamma(t_k)),
\end{equation} 
mapping $\tau_k$ onto $\upsilon_k$. Define the rescaled curve
\begin{equation}
    \label{eqn: Gamma tau k def}
    \Gamma_{\tau_k}:=\cal_{\tau_k}(\Gamma\vert_{\tau_k}),
\end{equation}
 which is the graph of some convex function $\gamma_{\tau_k}$ over some interval $I_{\tau_k}=[a_k,b_k]$.
\begin{lemma}
     \label{lemma: rescaled curves}
   $\Gamma_{\tau_k}$ is an admissible curve.
 \end{lemma}
 \begin{proof}
     (C1): Since $\cal_{\tau_k}$ preserves angles, it follows immediately that 
\begin{equation}
\label{eqn: angle preservation}
 \A(\Gamma_{\tau_k},I_{\tau_k})=\A(\Gamma,I\vert_{\tau_k}).
\end{equation}
Since $I\vert_{\tau_k}\subseteq I$, by convexity we have $$\A(\Gamma,I\vert_{\tau_k})\leq\A(\Gamma,I).$$ Combining the above two with the fact that $\Gamma$ satisfies $(C1)$, proves the claim.

\medskip\noindent (C2): We have $\cal(\Gamma_{\tau_k},I_{\tau_k})=|\Gamma_{\tau_k}(b_k)-\Gamma_{\tau_k}(a_k)|$. Clearly, $$\Gamma_{\tau_k}(a_k)=\cal_{\tau_k}(\Gamma(t_k)),\quad\text{and}\quad \Gamma_{\tau_k}(b_k)=\cal_{\tau_k}(\Gamma(v_k)).$$ Since $\cal_{\tau_k}$ is a similarity of ratio $R_k$, it follows that $|\Gamma_{\tau_k}(b_k)-\Gamma_{\tau_k}(a_k)|=R_k|\Gamma(t_k)-\Gamma(v_k)|$, in other words 
\begin{equation}
    \label{eqn: length scaling}
    \cal(\Gamma_{\tau_k},I_{\tau_k})=R_k\cal(\Gamma,I\vert_{\tau_k}).
\end{equation}
By I(k), we have $$\cal(\Gamma,I\vert_{\tau_k})=\len(\tau_k)\geq R_{k-2}^{-1}.$$ Using this in the display above it yields (C2).

\medskip\noindent (C3):  Let $I_{\tau_k}=[a_k,b_k]$, and $I\vert_{\tau_k}=[t_k,v_k]$. By \eqref{eqn: cal tau k def}, we see that $\cal_{\tau_k}$ maps $\Gamma(t_k)$ to the origin. As such, $\gamma_{\tau_k}(a_k)=\gamma_{\tau_k}(0)=0$.
    From \eqref{eqn: cal tau k def}, we also see that
     $\cal_{\tau_k}$ maps lines with direction $(1,\gamma_R'(t_k))$ onto lines with direction $(1,0)$. Since derivatives commute with linear transformations, it follows that $\gamma_{\tau_k,R}'(a_k)=\gamma_{\tau_k,R}'(0)=0$. 
     
\medskip\noindent (C4): 
Recall that $$\W(\Gamma_{\tau_k},I_{\tau_k})=\cal(\Gamma_{\tau_k},I_{\tau_k})\A(\Gamma_{\tau_k},I_{\tau_k}).$$ Combining \eqref{eqn: angle preservation} and \eqref{eqn: length scaling}, we get $$\W(\Gamma_{\tau_k},I_{\tau_k})=R_k\W(\Gamma,I\vert_{\tau_k})\leq 1,$$ where in the last step we have used II(k). 
\end{proof}

Since $\Gamma_{\tau_k}$ is an admissible curve, we may apply the single-scale algorithm to $\Gamma_{\tau_k}$ and obtain $\calt_1(\Gamma_{\tau_k})$ for each $\tau_k\in\calt_k(\Gamma)$. The $R_{k+1}^{-1}$ boxes are defined as follows:
\begin{equation}
    \label{eqn: calt k+1 def}
    \calt_{k+1}(\Gamma;\tau_k):=\mathcal{L}_{\tau_k}^{-1}(\calt_1(\Gamma_{\tau_k})),\quad\text{and}\quad\calt_{k+1}(\Gamma):=\bigcup_{\tau_k\in\calt_k(\Gamma)}\calt_{k+1}(\Gamma;\tau_k).
\end{equation}
Here we have made a slight abuse of notation to denote $$\mathcal{L}_{\tau_k}^{-1}(\calt_1(\Gamma_{\tau_k}))=\{\cal_{\tau_k}^{-1}(\sigma_{k+1}):\sigma_{k+1}\in\calt_1(\Gamma_{\tau_k})\}.\footnote{Quite often we would like to think of $\cal_{\tau_k}$ and its inverse as acting on these boxes rather than points in the plane, which would lead to similar abuse of notation.}$$
We show that each $\tau_{k+1}\in\calt_{k+1}(\Gamma;\tau_k)$ has the following form: 
\begin{equation}
    \label{eqn: tau k+1 form}
    \tau_{k+1}=\Gamma(t_{k+1})+\big(\mathbf{t}(t_{k+1})\;\vert\;\mathbf{n}(t_{k+1})\big)\begin{pmatrix}
        \len(\tau_{k+1})&0\\0&\wid(\tau_{k+1})
    \end{pmatrix} \big([0,1]\times[-1,1]\big), 
\end{equation}
for some $t_{k+1}\in I\vert_{\tau_k}$. By \eqref{eqn: calt k+1 def}, we have $\tau_{k+1}=\cal_{\tau_k}^{-1}(\sigma_{k+1})$ for some $\sigma_{k+1}\in\calt_1(\Gamma_{\tau_k})$. Let $I_{\tau_k}\vert_{\sigma_{k+1}}=[s_{k+1},u_{k+1}]$ denote the interval associated with $\sigma_{k+1}$. Define $I\vert_{\tau_{k+1}}:=[t_{k+1},v_{k+1}]$ to be the interval associated with $\tau_{k+1}$, where 
\begin{equation}
    \label{eqn: t k+1 v k+1 def}
    \Gamma(t_{k+1}):=\cal_{\tau_k}^{-1}(\Gamma_{\tau_k}(s_{k+1})),\quad\text{and}\quad\Gamma(v_{k+1}):=\cal_{\tau_k}^{-1}(\Gamma_{\tau_k}(u_{k+1})).
\end{equation}
By the single-scale algorithm, $\sigma_{k+1}$ has the following form: $$\sigma_{k+1}=\Gamma_{\tau_k}(s_{k+1})+\begin{pmatrix}
   \cos\phi&-\sin\phi\\\sin\phi&\cos\phi
\end{pmatrix}\begin{pmatrix}
    \len(\sigma_{k+1})&0\\0&\wid(\sigma_{k+1})
\end{pmatrix}\big([0,1]\times[-1,1]\big),$$ where $\phi$ is the argument of the vector $(1,\gamma_{\tau_k,R}'(s_{k+1}))$. Using the definitions \eqref{eqn: cal tau k def} and \eqref{eqn: t k+1 v k+1 def}, we get 
\begin{equation}
    \label{eqn: tau k+1 form calculation 1}
    \cal_{\tau_k}^{-1}(\sigma_{k+1})=\Gamma(t_{k+1})+\begin{pmatrix}
        \cos\theta_k&-\sin\theta_k\\\sin\theta_k&\cos\theta_k
    \end{pmatrix}\begin{pmatrix}
   \cos\phi&-\sin\phi\\\sin\phi&\cos\phi
\end{pmatrix}\begin{pmatrix}
    \len(\tau_{k+1})&0\\0&\wid(\tau_{k+1})
\end{pmatrix}\big([0,1]\times[-1,1]\big),
\end{equation}
 where $\theta_k$ is the argument of the vector $\mathbf{t}(t_k)$, and  
 \begin{equation}
      \label{eqn: wid k+1 vs wid k}
  \len(\tau_{k+1}):=\len(\sigma_{k+1})\wid(\tau_{k}),\quad\text{and}\quad \wid(\tau_{k+1}):=\wid(\sigma_{k+1})\wid(\tau_k).
 \end{equation} 
 Clearly, $$\begin{pmatrix}
        \cos\theta_k&-\sin\theta_k\\\sin\theta_k&\cos\theta_k
    \end{pmatrix}\begin{pmatrix}
   \cos\phi&-\sin\phi\\\sin\phi&\cos\phi
\end{pmatrix}=\begin{pmatrix}
    \cos(\theta_{k+1})&-\sin(\theta_{k+1})\\\sin(\theta_{k+1})&\cos(\theta_{k+1})
\end{pmatrix},$$ where $\theta_{k+1}:=\theta_k+\phi$. We claim that $\theta_{k+1}$ is the argument of the vector $\mathbf{t}(t_{k+1})$. Indeed, $\phi$ is the angle between the long directions of $\sigma_k=\cal_{\tau_k}(\tau_k)$ and $\sigma_{k+1}=\cal_{\tau_k}(\tau_{k+1})$. Since $\cal_{\tau_k}$ is a similarity, it preserves angles. As such, $\phi$ is also the angle between the long directions of $\tau_k$ and $\tau_{k+1}$. Now $\theta_k$ is the angle between the long direction of $\tau_k$ and the $\xi_1$-axis. Hence, $\theta_{k+1}$ is the angle between the long direction of $\tau_{k+1}$ and the $\xi_1$-axis, proving the claim above. In light of this, we have
\begin{equation}
    \label{eqn: rotation matrix k+1}
    \begin{pmatrix}
    \cos(\theta_{k+1})&-\sin(\theta_{k+1})\\\sin(\theta_{k+1})&\cos(\theta_{k+1})
\end{pmatrix}=\big(\mathbf{t}(t_{k+1})\,\vert\,\mathbf{n}(t_{k+1})\big).
\end{equation} Using \eqref{eqn: rotation matrix k+1} in \eqref{eqn: tau k+1 form calculation 1}, we obtain \eqref{eqn: tau k+1 form}. 

Recall the definition of $\cal_{\sigma_{k+1}}$ from \eqref{eqn: cal sigma def}. The above argument also reveals the following useful identity:
 \begin{equation}
        \label{eqn: cal composition law}
        \cal_{\tau_{k+1}}=\cal_{\sigma_{k+1}}\circ\cal_{\tau_k},
    \end{equation} where $\cal_{\tau_{k+1}}$ is defined analogous to the definition of $\cal_{\tau_k}$ given in \eqref{eqn: cal tau k def}. Moreover, we have the following.
\begin{lemma}
    \label{lemma: cal composition law}
    Let $\tau_k\in\calt_k(\Gamma)$, and $\tau_{k+1}=\cal_{\tau_k}^{-1}(\sigma_{k+1})$ for $\sigma_{k+1}\in\calt_1(\Gamma_{\tau_k})$. Then the following diagram commutes:
    \begin{figure}[H]
        \centering
        \includegraphics[width=0.35\linewidth]{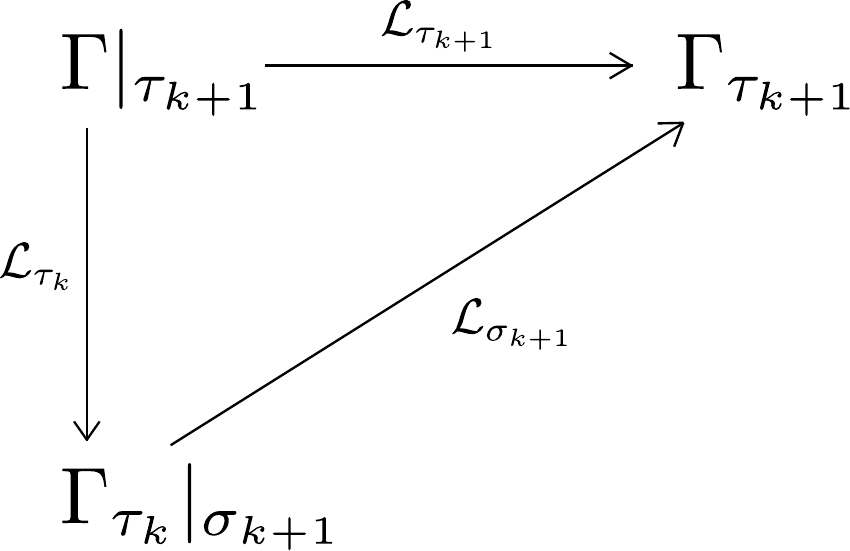}
        \label{fig:commutative diagram}
    \end{figure}
\end{lemma}
\begin{proof}
    By the definition given in \eqref{eqn: Gamma tau k def}, we have 
    \begin{equation}
        \label{eqn: Gamma tau k+1 def}
        \cal_{\tau_{k+1}}(\Gamma\vert_{\tau_{k+1}})=\Gamma_{\tau_{k+1}}.
    \end{equation}  
    Let $I\vert_{\tau_{k+1}}=[t_{k+1},v_{k+1}]$ and $I_{\tau_k}\vert_{\sigma_{k+1}}=[s_{k+1},u_{k+1}]$. Then by \eqref{eqn: Gamma vert tau k def}, $$\Gamma\vert_{\tau_{k+1}}=\big\{\Gamma(t):t\in[t_{k+1},v_{k+1}]\big\},\quad\text{and}\quad\Gamma_{\tau_k}\vert_{\sigma_{k+1}}=\big\{\Gamma_{\tau_k}(s):s\in[s_{k+1},u_{k+1}]\big\}.$$
    By \eqref{eqn: t k+1 v k+1 def}, we have $$\cal_{\tau_k}:\Gamma(t_{k+1})\mapsto\Gamma_{\tau_k}(s_{k+1}),\quad\text{and}\quad \cal_{\tau_k}:\Gamma(v_{k+1})\mapsto\Gamma_{\tau_k}(u_{k+1}).$$ It follows quickly that 
    \begin{equation}
        \label{eqn: Gamma tau k vert sigma k+1}
        \cal_{\tau_k}(\Gamma\vert_{\tau_{k+1}})=\Gamma_{\tau_k}\vert_{\sigma_{k+1}}.
    \end{equation}
     By \eqref{eqn: cal composition law}, \eqref{eqn: Gamma tau k+1 def}, and \eqref{eqn: Gamma tau k vert sigma k+1}, we have $$\Gamma_{\tau_{k+1}}=\cal_{\tau_{k+1}}(\Gamma\vert_{\tau_{k+1}})=\cal_{\sigma_{k+1}}(\cal_{\tau_k}(\Gamma\vert_{\tau_{k+1}}))=\cal_{\sigma_{k+1}}(\Gamma_{\tau_k}\vert_{\sigma_{k+1}}),$$ completing the proof.
\end{proof}
In order to conclude Step $k+1$, we need to verify the following conditions: \begin{enumerate}
    [label=\Roman*(k+1).]
    \item For all $\tau_{k+1}\in\calt_{k+1}(\Gamma)$ we have $$\len(\tau_{k+1})=\cal(\Gamma,I\vert_{\tau_{k+1}})\geq R_{k-1}^{-1}.$$
    \item For all $\tau_{k+1}\in\calt_{k+1}(\Gamma)$ we have $$\W(\Gamma,I\vert_{\tau_{k+1}})\leq R_{k+1}^{-1}.$$
\end{enumerate}
\begin{proof}[Verifying $\mathrm{I(k+1)}$] Let $\tau_{k+1}=\cal_{\tau_k}^{-1}(\sigma_{k+1})$ for $\sigma_{k+1}\in\calt_1(\Gamma_{\tau_k})$. By \eqref{eqn: wid k+1 vs wid k}, we have $$\len(\tau_{k+1})/\wid(\tau_{k+1})=\len(\sigma_{k+1})/\wid(\sigma_{k+1}).$$ By \eqref{eqn: sigma eccentricity}, we have $$\len(\sigma_{k+1})\geq R^{2\epsilon}\wid(\sigma_{k+1}),$$ using which in the above yields $\len(\tau_{k+1})\geq R^{-1}_{k-1}$. It remains to show that $\len(\tau_{k+1})=\cal(\Gamma,I\vert_{\tau_{k+1}})$. By \eqref{eqn: wid k+1 vs wid k} and \eqref{eqn: len wid sigma def}, we have $$\len(\tau_{k+1})=R_k^{-1}\cal(\Gamma_{\tau_k},I_{\tau_k}\vert_{\sigma_{k+1}}).$$ But $\cal(\Gamma_{\tau_k},I_{\tau_k}\vert_{\sigma_{k+1}})=|\Gamma_{\tau_k}(u_{k+1})-\Gamma_{\tau_k}(s_{k+1})|$, where the points $s_{k+1},u_{k+1}$ are related to $t_{k+1},v_{k+1}$ by the relation \eqref{eqn: t k+1 v k+1 def}. Since $\cal_{\tau_k}$ is a similarity of ratio $R_k$, it follows that $\cal(\Gamma_{\tau_k},I_{\tau_k}\vert_{\sigma_{k+1}})=R_k\cal(\Gamma,I\vert_{\tau_{k+1}})$, combining which with the display above establishes I(k+1).
\end{proof}
\begin{proof}[Verifying $\mathrm{II(k+1)}$]
By \eqref{eqn: t k+1 v k+1 def} and the fact that $\cal_{\tau_k}$ is a similarity, we get 
\begin{equation}
\label{eqn: tan(b-a) formula}
   \A(\Gamma,I\vert_{\tau_{k+1}})=\A(\Gamma_{\tau_k},I_{\tau_k}\vert_{\sigma_{k+1}}).
\end{equation}
Since $\cal_{\tau_k}$ scales isotropically by $R_k$, it follows again that 
\begin{equation*}
\cal(\Gamma,I\vert_{\tau_{k+1}})=R_k^{-1}\cal(\Gamma_{\tau_k},I_{\tau_k}\vert_{\sigma_{k+1}}).
\end{equation*}
Combining the above, we have 
\begin{equation*}
    \label{eqn: W under cal tau k}
    \W(\Gamma,I\vert_{\tau_{k+1}})=R_k^{-1}\W(\Gamma_{\tau_k},I_{\tau_k}\vert_{\sigma_{k+1}}).
\end{equation*}
By the single-scale algorithm (specifically, \eqref{eqn: wid/len formula}) applied to $\Gamma_{\tau_k}$, we have $$\W(\Gamma_{\tau_k},I_{\tau_k}\vert_{\sigma_{k+1}})\leq R_1^{-1},$$ from which and the above II(k+1) follows.
\end{proof}
We call $\tau_{k+1}\in\calt(\Gamma;\tau_k)$ \textit{exceptional}, if $\sigma_{k+1}=\cal_{\tau_k}(\tau_{k+1})$ is exceptional (recall the definition from \S\ref{sec: single-scale algorithm}). Otherwise we call $\tau_{k+1}$ \textit{typical}. If $\tau_{k+1}$ is typical, we call it a \textit{left box}, if $\sigma_{k+1}$ is a left box, and we call $\tau_{k+1}$ a \textit{right box} if $\sigma_{k+1}$ is a right box (recall the definitions from \S\ref{sec: single-scale algorithm}). We make a simple but useful observation about exceptional boxes.
\begin{equation}
    \label{eqn: multiscale exceptional box properties}
    \tau_{k+1}\in\calt_{k+1}(\Gamma;\tau_k)\text{ is exceptional }\implies \#\calt_{k+1}(\Gamma;\tau_k)=1\text{ and }I\vert_{\tau_{k+1}}=I\vert_{\tau_k}.
\end{equation}
To see this, by definition we have that $\sigma_{k+1}\in\calt_1(\Gamma_{\tau_k})$ is exceptional so that by \eqref{eqn: calt k+1 def}, $$\calt_{k+1}(\Gamma;\tau_k)=\cal_{\tau_k}^{-1}(\calt_1(\Gamma_{\tau_k}))=\{\cal_{\tau_k}^{-1}(\sigma_{k+1})\}.$$ By Step 0 of the single-scale algorithm, we also get that $I_{\tau_k}\vert_{\sigma_{k+1}}=I_{\tau_k}$. By this, and the definitions of $I_{\tau_k}, I\vert_{\tau_{k+1}}$ (see \eqref{eqn: Gamma tau k def}, \eqref{eqn: t k+1 v k+1 def}) it follows that $I\vert_{\tau_{k+1}}=I\vert_{\tau_k}$.

The construction process terminates at $k=N$, with the collection $\calt_N(\Gamma)$. 
\subsubsection{Properties of the multi-scale decomposition}
 The following properties are satisfied by the collection $\calt_k(\Gamma)$ constructed above. Each of these results follow from their single-scale counterpart in \S\ref{sec: single scale properties}, combined with the recursive nature of the multi-scale algorithm.
  \begin{lemma}
     \label{lemma: wid over len}
    For all $\tau_{k}\in\calt_{k}(\Gamma)$ we have
      \begin{equation}
         \label{eqn: wid over len tau k}
        \W(\Gamma,I\vert_{\tau_k})\leq R_k^{-1}.
    \end{equation}
    Furthermore, if $I\vert_{\tau_k}=[t_k,v_k]$, then for all $\eta>0$ sufficiently small, we have \begin{align}
         \label{eqn: angle sep tau k}
         \nonumber\W(\Gamma,[t_k,v_k+\eta])>R_k^{-1}\quad&\text{if $\tau_k$ is a left box};\\
         \W(\Gamma,[t_k-\eta,v_k])>R_k^{-1}\quad&\text{if $\tau_k$ is a right box}.
     \end{align}
 \end{lemma}
 \begin{proof}
   The inequality \eqref{eqn: wid over len tau k} is just II(k) from the $k$th step of the multi-scale algorithm. 
   
   The $k=1$ case of \eqref{eqn: angle sep tau k} follows directly from \eqref{eqn: sigma angle sep}. For $k\geq 2$, suppose $\tau_{k}\in\calt_k(\Gamma;\tau_{k-1})$, so that $\tau_k=\cal_{\tau_{k-1}}^{-1}(\sigma_k)$ for some box $\sigma_k\in\calt_1(\Gamma_{\tau_{k-1}})$. If $I\vert_{\sigma_{k}}=[s_{k},u_{k}]$, then by \eqref{eqn: t k+1 v k+1 def} $$\cal_{\tau_k}(\Gamma(t_k))=\Gamma_{\tau_k}(s_k)\quad\text{and}\quad \cal_{\tau_k}(\Gamma(v_k))=\Gamma_{\tau_k}(u_k).$$
  By Lemma \ref{lemma: rescaled curves} we know $\Gamma_{\tau_k}$ is an admissible curve. Then \eqref{eqn: sigma angle sep} gives
   \begin{align}
       \label{eqn: sigma k wid/len}
      \nonumber \W(\Gamma_{\tau_k},[s_k,u_k+\eta'])>R_1^{-1}\quad&\text{if $\sigma_k$ is a left box};\\
       \W(\Gamma_{\tau_k},[s_k-\eta',u_k])>R_1^{-1}\quad&\text{if $\sigma_k$ is a right box},
   \end{align}
   for all $\eta'>0$.
    By definition, $\tau_k$ is a left (respectively, right) box if and only if $\sigma_k$ is a left (respectively, right) box. Then \eqref{eqn: angle sep tau k} follows from the fact that $\cal_{\tau_{k-1}}$ is a similarity of ratio $R_{k-1}$, and the fact that $R_{k}=R_{k-1}\cdot R_1$.
 \end{proof}
 Property I(k) from the $k$th step of the multi-scale algorithm is important for proving a forthcoming result (Lemma \ref{lemma: high lemma}). We record it in the form of the following lemma.
 \begin{lemma}
     \label{lemma: tau k eccentricity}
     For each $\tau_k\in\calt_k(\Gamma)$, we have 
     \begin{equation}
     \label{eqn: tau k are eccentric}
        \len(\tau_{k})\geq R^{2\epsilon}\wid(\tau_k).
    \end{equation}
 \end{lemma}
 As mentioned in the introduction, a key difficulty in the general setting is to relate the geometry of the decomposing boxes at two different scales. The following is a simple consequence of our construction, which turns out to be sufficient for the high/low argument to work in our setting.
 \begin{lemma}
     \label{lemma: len tau k range}
     For all $\tau_{k+1}\in\calt_{k+1}(\Gamma;\tau_k)$, we have 
     \begin{equation}
         \label{eqn: len tau k+1 relative}
         R^{-\epsilon}\len(\tau_k)\leq\len(\tau_{k+1})\leq \len(\tau_k).
     \end{equation}
 \end{lemma}
 \begin{proof}
    By construction, $\tau_{k+1}=\cal_{\tau_k}^{-1}(\sigma_{k+1})$ for some $\sigma_{k+1}\in\calt_1(\Gamma_{\tau_k})$, and by \eqref{eqn: wid k+1 vs wid k} we have $$\len(\tau_{k+1})=\len(\sigma_{k+1})R_k^{-1}.$$
   Now we apply Lemma \ref{lemma: sigma length} to the box $\sigma_{k+1}\in\calt_1(\Gamma_{\tau_k})$, and obtain $$R^{-\epsilon}\cal(\Gamma_{\tau_k},I_{\tau_k})\leq\cal(\Gamma_{\tau_k},I_{\tau_k}\vert_{\sigma_{k+1}})\leq\cal(\Gamma_{\tau_k},I_{\tau_k}).$$
   By \eqref{eqn: length scaling} we have $ \cal(\Gamma_{\tau_k},I_{\tau_k})=R_k\cal(\Gamma,I\vert_{\tau_k})$. On the other hand, by \eqref{eqn: len wid sigma def} and I(k) we know that $$\len(\sigma_{k+1})=\cal(\Gamma_{\tau_k},I_{\tau_k}\vert_{\sigma_{k+1}})\quad\text{and}\quad \len(\tau_k)=\cal(\Gamma,I\vert_{\tau_k}).$$
   Combining everything we get the desired estimate.
 \end{proof}
 The following gives a bound on the number of decomposing boxes at each scale by iterating the analogous bound from the single-scale algorithm (Lemma \ref{lemma: number of R1 boxes}).
    \begin{lemma}
        \label{lemma: number of boxes}
        Let $k\in\{1,\dots,N-1\}$. Then $$\#\calt_{k+1}(\Gamma;\tau_{k})\leq 32R^{\epsilon/2},\quad\text{for all}\quad \tau_k\in\calt_k(\Gamma);$$
        and hence $$\#\calt_k(\Gamma)\leq 32^k R^{k\epsilon/2}.$$
    \end{lemma}
    \begin{proof}
        By construction, we have $\calt_{k+1}(\Gamma;\tau_k)=\cal_{\tau_k}^{-1}(\calt_1(\Gamma_{\tau_k}))$. By Lemma \ref{lemma: rescaled curves} we know that $\Gamma_{\tau_k}$ is admissible. Thus, by Lemma \ref{lemma: number of R1 boxes}, we have $$\#\calt_{k+1}(\Gamma;\tau_k)=\#\calt_1(\Gamma_{\tau_k})\leq 32R^{\epsilon/2}.$$      
        For the latter inequality, we argue by induction. The base case $k=1$ is just Lemma \ref{lemma: number of R1 boxes}. Now suppose the result holds for some $k\geq 1$. By \eqref{eqn: calt k+1 def} and the induction hypothesis, we get $$\#\calt_{k+1}(\Gamma)\leq\sum_{\tau_{k}\in\calt_{k}(\Gamma)}\#\calt_{k+1}(\Gamma;\tau_k)\leq\#\calt_k(\Gamma)\sup_{\tau_k\in\calt_k(\Gamma)}\#\calt_{k+1}(\Gamma;\tau_k)\leq 32^kR^{k\epsilon/2}\times 32R^{\epsilon/2}=32^{k+1}R^{(k+1)\epsilon/2},$$ which completes the induction step.
    \end{proof}
    \begin{lemma}
        \label{lemma: tau k+1 containment}
        Let $k\in\{0,1,\dots,N-1\}$, and $\tau_k\in\calt_k(\Gamma)$. For all $R\geq 2^{1/\epsilon}$, we have \begin{equation}
            \label{eqn: tau k+1 containment} \tau_{k+1}\subseteq 2\cdot\tau_k,\quad\text{for all}\quad \tau_{k+1}\in\calt_{k+1}(\Gamma;\tau_k).
        \end{equation}
    \end{lemma}
    \begin{proof}
First we prove the $k=0$ case. Recall from Step 0 of the multi-scale algorithm that $$\tau_0=[a,a+\cal(\Gamma,I)]\times[-1,1].$$ Since $\cal(\Gamma,I)=|\Gamma(b)-\Gamma(a)|\geq b-a$, we have that $[a,b]\subseteq[a,a+\cal(\Gamma,I)]$. The $k=0$ case then follows directly from Lemma \ref{lemma: sigma containment}.

For $k\geq 1$, let $\tau_{k+1}\in\calt_{k+1}(\Gamma;\tau_k)$. Then, there exists $\sigma_{k+1}\in\calt_1(\Gamma_{\tau_k})$, such that $\cal_{\tau_k}(\tau_{k+1})=\sigma_{k+1}$. By Lemma \ref{lemma: rescaled curves}, $\Gamma_{\tau_k}$ is an admissible curve associated with the pair $(\gamma_{\tau_k},I_{\tau_k})$. Let $I_{\tau_k}=[a_k,b_k]$ and $\upsilon_k=[0,\len(\tau_k)/\wid(\tau_k)]\times[-1,1]$ be as defined in Step $k+1$ of the multi-scale algorithm. But $a_k=0$, and by \eqref{eqn: length scaling}, we have $$\len(\tau_k)/\wid(\tau_k)=R_k\cal(\Gamma,I\vert_{\tau_k})=\cal(\Gamma_{\tau_k},I_{\tau_k}).$$ It follows that $\upsilon_k$ has the following form 
\begin{equation}
    \label{eqn: upsilon k form}
    \upsilon_k=[a_k,a_k+\cal(\Gamma_{\tau_k},I_{\tau_k})]\times[-1,1].
\end{equation}
Thus, by the $k=0$ case of this lemma, we get $\sigma_{k+1}\subseteq 2\cdot\upsilon_k$. Now by \eqref{eqn: cal tau k def}, we know that $\cal_{\tau_k}:\tau_k\mapsto \upsilon_k$. But affine transformations preserve centroids, so that $\cal_{\tau_k}:c(\tau_k)\mapsto c(\upsilon_k)$. By definition, we also have $\cal_{\tau_k}:\tau_{k+1}\mapsto\sigma_{k+1}$. Combining the three, we get $$\sigma_{k+1}\subseteq 2\cdot \upsilon_k=c(\upsilon_k)+2(\upsilon_k-c(\upsilon_k))\implies \tau_{k+1}\subseteq c(\tau_k)+2(\tau_k-c(\tau_k))=2\cdot\tau_k,$$ as required.
    \end{proof}
\subsection{Obtaining the ideal partition} Let $\Gamma$ be the graph of a convex function $\gamma:[0,1]\to\R$ satisfying \eqref{eqn: gamma normalisation}. Recall the partition \eqref{eqn: [0,1] partition} of $\Gamma$. Recall also the definitions of the index sets $\mathcal{P}_1$ and $\mathcal{P}_2$. Here we obtain the family $\J$ appearing in the statement of Theorem \ref{thm: convex decoupling} by partitioning each $I_p$ into a family of subintervals $\J_p$. This is done in cases. 

\medskip\noindent\underline{Case 1:} $p\in\mathcal{P}_1$. These are exceptional cases which do not require partitioning, and we simply set $\J_p:=\{I_p\}$. 

\medskip\noindent\underline{Case 2:} $p\in\mathcal{P}_2$. Recall that $r_p=R^{2\epsilon}\cal(\Gamma,I_p)^{-1}$, and define $R_p:=r_p^{-1}R$. By \eqref{eqn: I_p length bound} and \eqref{eqn: gamma normalisation}, we have $$R^{-1}\leq|I_p|\leq\cal(\Gamma,I_p)\leq|\Gamma(1)-\Gamma(0)|\leq\sqrt{2},$$ so that $$R^{-2\epsilon}\leq R_p\leq R^{1-\epsilon}.$$ 
We further divide this into two sub-cases.

\medskip\noindent\underline{Subcase 2A:} $R_p\geq R^\epsilon$. This is the typical case, and we use the multi-scale algorithm to the admissible curve $\Gamma_p$ for all such $p$. This produces the families $\mathcal{T}_k(\Gamma_p)$ where $k=0,1,\dots,N=1/\epsilon$. In order to decouple $\Gamma$ at scale $R^{-1}$, we would like to decouple each $\Gamma_p$ at scale $R_p^{-1}$. Without loss of generality, we may assume that $$R_p=R^{k_p\epsilon},\quad\text{for some}\quad 1\leq k_p\leq N-1.$$ If not, we have $R_p=R^{k_p\tilde\epsilon}$ for some $\tilde\epsilon\in[\epsilon/2,\epsilon)$. Then, we can simply replace the sequence of scales \eqref{eqn: intermediate scales} by the sequence $$R^{-N\tilde\epsilon}<R^{-(N-1)\tilde\epsilon}<\dots<R^{-\tilde\epsilon}<1,\quad N=1/\tilde\epsilon,$$ and proceed similarly. 
Now $\Gamma_p$ is an admissible curve over an interval $J_p$, and the associated intervals $I\vert_{\tau_{k_p}}:=[t_{k_p},v_{k_p}]$ (for $\tau_{k_p}\in\calt_{k_p}(\Gamma_p)$) form a covering of $J_p$.  Taking the pull-back under $\cal_p$, each $I\vert_{\tau_{k_p}}$ gives rise to an interval $[t,v]$ given by 
    \begin{equation}
        \label{eqn: pullback points}   \cal_p(\Gamma(t))=\Gamma_p(t_{k_p})\quad\text{and}\quad\cal_p(\Gamma(v))=\Gamma_p(v_{k_p}).
    \end{equation}
    We define the collection $\J_p$ as 
    \begin{equation}
        \label{eqn: J def}
        \J_p:=\big\{[t,v]:t,v\text{ satisfy }\eqref{eqn: pullback points}\text{ for some }\tau_{k_p}\in\calt_{k_p}(\Gamma_p)\big\},
    \end{equation}
which forms a covering of $I_p$.

\medskip\noindent\underline{Subcase 2B:} $R_p<R^\epsilon$. This is also an exceptional case, where we apply the single-scale algorithm \S\ref{sec: single-scale algorithm} to all such the admissible curve $\Gamma_p$ at the (different) scale $R_p^{-1}$. Thus, instead of applying the algorithm many times over small increments, we apply it only once, at the smallest level. Let us denote the resulting family of boxes as $\calt_1(\Gamma_p,R_p^{-1})$.\footnote{In this notation, the family $\calt_1(\Gamma)$ would be $\calt_1(\Gamma,R^{-\epsilon})$.} For each $I\vert_{\tau_{1,p}}=[t_{1,p},v_{1,p}]$ associated with $\tau_{1,p}\in\calt_1(\Gamma_p,R_p^{-1})$, define the pull-backs 
\begin{equation}
    \label{eqn: case 2b pullbacks}
    \cal_p(\Gamma(t))=\Gamma_p(t_{1,p})\quad\text{and}\quad \cal_p(\Gamma(v))=\Gamma_p(v_{1,p}),
\end{equation}
and the families $$ \J_p:=\big\{[t,v]:t,v\text{ satisfy }\eqref{eqn: case 2b pullbacks} \text{ for some }\tau_{1,p}\in\calt_{1}(\Gamma_p,R_p^{-1})\big\},$$ which gives a covering of $I_p$.

Having defined the families of intervals $\J_p$ for all possible indices $0\leq p\leq P-1$, we finally define $$\J:=\bigcup_{p=1}^P\J_p,$$ which we claim is an ideal partition in the sense of Definition \ref{def: ideal partition}.

    \medskip\noindent\textit{Verifying} (J1): Let $J\in\J$, so that $J\in\J_p$ for some $0\leq p\leq P-1$. We break into cases again.

     In Case 1, $J=I_p=[s_p,s_{p+1}]$, and by the definition of the collection $\mathcal{P}_1$, we know that $|I_p|=R^{-1}$. By \eqref{eqn: gamma normalisation}, it follows that 
    $$(s_{p+1}-s_p)(\gamma_L'(s_{p+1})-\gamma_R'(s_p))\leq |I_p|(\gamma_L'(1)-\gamma_R'(0))\leq R^{-1}.$$
   
   In Case 2A, $J=[t,v]$ for some $t,v$ satisfying \eqref{eqn: pullback points}. Since $\cal_p$ preserves angles, it follows that $$\A(\Gamma,J)=\A(\Gamma_p,I\vert_{\tau_{k_p}}).$$ On the other hand, since $\cal_p$ scales uniformly by $r_p$, we have $$\cal(\Gamma,J)=r_p^{-1}\cal(\Gamma_p,I\vert_{\tau_{k_p}}).$$ 
    Then $$\W(\Gamma,J)=r_p^{-1}\W(\Gamma_p,I\vert_{\tau_{k_p}})\leq r_p^{-1}R^{-k_p\epsilon},$$ where the last step follows from \eqref{eqn: wid over len tau k}. But $r_p^{-1}R^{-k_p\epsilon}=r_p^{-1}R_p^{-1}=R^{-1}$, by the definition of $R_p$. It follows that $$\W(\Gamma,J)\leq R^{-1}\quad\text{so that}\quad (v-t)(\gamma_L'(v)-\gamma_R'(t))\leq 2R^{-1},$$ where in the last step we used \eqref{eqn: gamma normalisation}.

    Finally, in Case 2B, $J=[t,v]$ for some $t,v$ satisfying \eqref{eqn: case 2b pullbacks}. Using similar arguments as above, we get $$\W(\Gamma,J)=r_p^{-1}\W(\Gamma_p,I\vert_{\tau_{1,p}})\leq r_p^{-1}R_p^{-1},$$ where we have used \eqref{eqn: wid/len formula}. But $r_p^{-1}R_p^{-1}=R^{-1}$, and so using \eqref{eqn: gamma normalisation} the desired conclusion follows again.
    
     \medskip\noindent\textit{Verifying} (J2): For all $p$ satisfying Case 1, we have $\#\J_p=1$. For all $p$ satisfying Case 2, we use \eqref{eqn: angle sep tau k} (in Subcase 2A) and \eqref{eqn: sigma angle sep} (in Subcase 2B) to the left typical boxes, and combine it with \cite[Lemma 2.3 (iii)]{SZ} to find that $$\#\J_p\leq C_1N(\Gamma_p,R_p^{-1})\log(2+R_p),$$ for some absolute constant $C_1$. But clearly $$N(\Gamma_p,R_p^{-1})=N(\Gamma\vert_{I_p},R^{-1}).$$ Summing over all $p$ we get $$\#\J\leq\sum_{p=1}^P\#\J_p\leq 3C_1R^{2\epsilon}\log(2+R)N(\Gamma,R^{-1}),$$ using the fact that $R_p\leq R$ and that $P\leq 3R^{2\epsilon}$. Then (J2) follows from the fact that $$\log(2+R)\leq R^\epsilon\quad\text{for all}\quad R\geq (1/\epsilon)^{1/\epsilon}.$$
     
    \medskip\noindent\textit{Verifying} (J3): Let $J\in\J_p$ for some $0\leq p\leq P-1$. If $p$ satisfies Case 1, then $J=I_p$. By \eqref{eqn: I_p length bound} we have $|I_p|\geq R^{-1}$. If $p$ satisfies Case 2A, then by \eqref{eqn: tau k are eccentric}, we have $$\cal(\Gamma,J)=r_p^{-1}\cal(\Gamma_p,I\vert_{\tau_{k_p}})=r_p^{-1}\len(\tau_{k_p})\geq r_p^{-1}R^{2\epsilon}R^{-k_p\epsilon}=R^{-1+2\epsilon}.$$ But by \eqref{eqn: gamma normalisation}, we have $\cal(\Gamma,J)\leq 2|J|$, and so (J3) follows for these intervals.
    Finally, if $p$ satisfies Case 2B, then following the proof of Lemma \ref{lemma: sigma eccentricity}, we find that
    $$\len(\tau_{1,p})\geq R^{2\epsilon}R_p^{-1}.$$
    Thus, $$\cal(\Gamma,J)=r_p^{-1}\len(\tau_{1,p})\geq r_p^{-1}R^{2\epsilon}R_p^{-1}=R^{-1+2\epsilon}.$$ Since $\cal(\Gamma,J)\leq 2|J|$, the conclusion follows here as well.    
\subsection{A reduction of the original problem}
Our goal is to prove Theorem \ref{thm: convex decoupling}. Since $P\leq 2R^{2\epsilon}$, in order to prove \eqref{eqn: main bound} (with $100\epsilon$ in place of $\epsilon$), it suffices to show the following for each $p$:
\begin{equation}
    \label{eqn: main bound p}
    \|\sum_{J\in\J_p}f_J\|_{L^6(\R^2)}\leq C_\epsilon R^{50\epsilon}\big(\sum_{J\in\J_p}\|f_J\|_{L^6(\R^2)}^2\big)^{1/2}.
\end{equation}
This is trivially established when $p$ satisfies Case 1 or Case 2B from the previous subsection. Indeed, if $p$ satisfies Case 1, then $\#\J_p=1$. If $p$ satisfies Case 2B, then by Lemma \ref{lemma: number of R1 boxes} applied to the collection $\calt_1(\Gamma_p,R_p^{-1})$, we have $\#\J_p\leq 32 R_p^{\epsilon/2}\leq 32 R^{\epsilon/2}$. Therefore, we focus on proving \eqref{eqn: main bound p} for only those $p$ that satisfy Case 2A. 

Applying the multi-scale algorithm to each such $\Gamma_p$, we obtain the families $\calt_{k_p}(\Gamma_p)$. Applying $\cal_p^{-1}$ to each $\tau_{k_p}\in\calt_{k_p}(\Gamma_p)$ produces the boxes $$\tau_J=\Gamma(t)+\frac{1}{\sqrt{1+\gamma_R'(t)^2}}\begin{pmatrix}
     1&\gamma_R'(t)\\\gamma_R'(t)&-1
 \end{pmatrix}\begin{pmatrix}
        \len(\tau_J)&0\\0&\wid(\tau_J)
    \end{pmatrix} \big([0,1]\times[-1,1]\big),$$ where $t$ satisfies \eqref{eqn: pullback points}, and $$\wid(\tau_J):=r_p^{-1}\wid(\tau_{k_p})=r_p^{-1}R^{-k_p\epsilon}=R^{-1},$$ and $$\len(\tau_J):=r_p^{-1}\len(\tau_{k_p})=r_p^{-1}\cal(\Gamma_p,I\vert_{\tau_{k_p}})=\cal(\Gamma,J).$$ Recall the Definition \ref{def: vertical nbhds}. It easily follows that $\mathcal{N}_{R^{-1}}(\Gamma,J)\subseteq 2\cdot\tau_J$ for all $J\in\J_p$. Thus, in order to prove \eqref{eqn: main bound p}, we may prove an estimate of the form $$\|\sum_{J\in\J_p}f_{\tau_J}\|_{L^6(\R^2)}\leq C_\epsilon R^{50\epsilon}\big(\sum_{J\in\J_p}\|f_{\tau_J}\|_{L^6(\R^2)}^2\big)^{1/2},$$ for all Schwarz functions $f_{\tau_J}$ with Fourier support in $2\cdot\tau_J$. But since the above estimate is invariant under affine transformations, it suffices to decouple at the level of the rescaled curves $\Gamma_p$. Since each $\Gamma_p$ is an admissible curve satisfying (C2) with an equality, we prove the following estimate, which is uniform across all such curves.
    \begin{prop}
        \label{prop: admissible decoupling}
        Let $\Gamma$ be an admissible curve that satisfies (C2) with equality, and let $\calt_k(\Gamma)$ be as constructed in \S\ref{sec: multi-scale algorithm} for $k=0,1,\dots,1/\epsilon$. There exists a uniform constant $C_\epsilon\geq 1$ such that 
        \begin{equation*}
        \|\sum_{\tau_k\in\calt_k(\Gamma)}f_{\tau_k}\|_{L^6(\R^2)}\leq C_\epsilon R^{50\epsilon}\big(\sum_{\tau_N\in\calt_k(\Gamma)}\|f_{\tau_k}\|_{L^6(\R^2)}^2\big)^{1/2},
        \end{equation*} for all Schwarz functions $f_{\tau_k}$ having Fourier support in $2\cdot\tau_k$.
    \end{prop}
 The remainder of the paper until \S\ref{sec: applications and examples} focuses on the proof of Proposition \ref{prop: admissible decoupling}. As such, let $k\in\{0,1,\dots,1/\epsilon\}$, which shall be kept fixed until \S\ref{sec: broad/narrow}. The high/low argument works by studying the problem at all scales between $R^{-k\epsilon}$ and $R^{-\epsilon}$, and so to simplify the notation, we shall denote this $k$ by $N$, which was previously used to denote $1/\epsilon$. With this new notation, we have $N\leq 1/\epsilon$. In order to prove Proposition \ref{prop: admissible decoupling}, we prove the following bound, which is ostensibly weaker:
\begin{equation}
    \label{eqn: convex decoupling ver 2}
     \|f\|_{L^6(B_R)}^6\leq C_\epsilon R^{20\epsilon}\big(\sum_{\tau_N\in\calt_N(\Gamma)}\|f_{\tau_N}\|_{L^\infty}^2\big)^2\big(\sum_{\tau_N\in\calt_N(\Gamma)}\|f_{\tau_N}\|_{L^2}^2\big),\quad\text{for all $R\times R$ square $B_R$.}
\end{equation}
Once \eqref{eqn: convex decoupling ver 2} is established, Proposition \ref{prop: admissible decoupling} will follow from a reverse H\"older's inequality achieved by pigeonholing in the wave-packets of the function $f$. The details of this step, \textit{mutatis mutandis}, can be found in \cite[\S 5]{gmw}. In order to simplify the notation further, for each $0\leq k\leq 1/\epsilon$, we replace each $\tau_k$ by $2\cdot \tau_k$. This change does not affect the geometry of these boxes, except for doubling their sizes, so that \eqref{eqn: tau k+1 containment} from Lemma \ref{lemma: tau k+1 containment} must be amended to the following
\begin{equation}
    \label{eqn: tau k+1 containment amendment}
    \tau_{k+1}\subseteq 3\cdot\tau_k,
\end{equation}
which is easy to verify from \eqref{eqn: tau k+1 containment}.
 \subsection{Pigeonholing in the lengths}\label{sec: pigeonholing}
In the case of the parabola (see Theorem A and Theorem B), the canonical boxes have the same size $R^{-1/2}\times R^{-1}$. In our general setting, however, we can have boxes of varying lengths. For instance, the canonical boxes for the examples considered in \S\ref{sec: examples} have every possible lengths in a sequence $$R^{-1/2}<\eta^\upkappa<\eta^{\upkappa-1}<\dots<\eta^2<\eta<1,$$ for some $\eta\in (0,1)$ and a large integer $\upkappa\in\mathbb{N}$. Nevertheless, the lengths all lie in the range $[R^{-1},R^{2\epsilon}]$ (see \eqref{eqn: lambda k absolute range}). This suggests that we can dyadically pigeonhole in the lengths. However, due to the multi-scale nature of the high/low argument, we must perform dyadic pigeonholing at scales $R_k^{-1}$ for every $k\in\{1,\dots,N\}$. Since we are working with $N\leq1/\epsilon$ many scales, this will incur a loss of $O(\log R)^{1/\epsilon}$, which is still negligible. We describe the pigeonholing process below.

Let $k\in\{0,1,\dots,N-1\}$. For each $\lambda_{k+1}\geq 0$, and each $\tau_{k}\in\calt_{k}$, we define the subcollections 
\begin{equation}
    \label{eqn: calt k+1 tau k lambda k+1 definition}
    \calt_{k+1}(\Gamma;\tau_k;\lambda_{k+1}):=\{\tau_{k+1}\in\calt_{k+1}(\Gamma;\tau_k):\lambda_{k+1}\leq\len(\tau_{k+1})\leq2\lambda_{k+1}\}.
\end{equation}
For a finite sequence of positive numbers $\Lambda:=(\lambda_1,\dots,\lambda_N)$, we define the `branch' $f^\Lambda$ of the function $f$ as follows. Define $f^\Lambda_{\tau_N}:=f_{\tau_N}$ for all $\tau_N\in\calt_N(\Gamma)$. For $\tau_k\in\calt_k(\Gamma)$, iteratively define $f^\Lambda_{\tau_k}$ by the relation 
\begin{equation}
    \label{eqn: f^Lambda_{tau_k} def}
    f^\Lambda_{\tau_k}:=\sum_{\tau_{k+1}\in\calt_{k+1
    }(\Gamma;\tau_k;\lambda_{k+1})}f^\Lambda_{\tau_{k+1}}.
\end{equation}
The iteration terminates at $k=0$, and we define $f^\Lambda:=f^\Lambda_{\tau_0}$. That is,
\begin{equation}
    \label{eqn: branch definition}
f^\Lambda=\sum_{\tau_1\in\calt_1(\Gamma;\tau_0;\lambda_1)}\sum_{\tau_2\in\calt_2(\Gamma;\tau_1;\lambda_2)}\dots\sum_{\tau_N\in\calt_N(\Gamma;\tau_{N-1};\lambda_N)}f_{\tau_N}.
\end{equation}
We introduce the following notation.
\begin{notation}
    \label{not: pigeonholed boxes}
    Let $0\leq k\leq N$, and $\tau_k\in\calt_k(\Gamma)$ satisfy $\lambda_k\leq \len(\tau_k)\leq 2\lambda_k$. For $0\leq j\leq N-k$, we define the collections $\calt^\Lambda_{k+j}(\Gamma;\tau_k)$ as follows. Define $\calt^\Lambda_{k}(\Gamma;\tau_k):=\{\tau_k\}$, and $\calt^\Lambda_{k+1}(\Gamma;\tau_k)=\calt_{k+1}(\Gamma;\tau_k;\lambda_{k+1})$. Then for $2\leq j\leq N-k$, iteratively define
    \begin{equation}
        \label{eqn: pigeonholed boxes}
        \calt^\Lambda_{k+j}(\Gamma;\tau_k):=\bigcup_{\tau_{k+1}\in\calt_{k+1}^\Lambda(\Gamma;\tau_k)}\calt_{k+j}^\Lambda(\Gamma;\tau_{k+1}),
    \end{equation}
     noting that $\calt_{k+j}^\Lambda(\Gamma;\tau_{k+1})=\calt_{(k+1)+(j-1)}^{\Lambda}(\Gamma;\tau_{k+1})$ is defined in the previous step of the iteration.
When $k=0$ in the above, we write $\calt^\Lambda_j(\Gamma;\tau_0)=:\calt^\Lambda_j(\Gamma)$. We also define the partial order $\prec$ by $$\tau_m\prec\tau_k \iff \tau_m\in\calt_m^\Lambda(\Gamma;\tau_k) \quad\text{for}\quad m\geq k.$$ This relation will be a surrogate for the frequently used ``$\theta\subset\tau$" notation in the context of parabolic decoupling.
\end{notation}
With this notation, we verify the following representation.
\begin{lemma}\label{lemma: pigeonholed boxes representation}
Let $k\in\{0,1,\dots,N\}$ and $\tau_k\in\calt^\Lambda(\Gamma)$. Then
\begin{equation}
    \label{eqn: pigeonholed boxes representation}
    \calt^\Lambda_{k+j}(\Gamma;\tau_k)=\bigcup_{\tau_{k+l}\in\calt^\Lambda_{k+l}(\Gamma;\tau_k)}\calt^\Lambda_{k+j}(\Gamma;\tau_{k+l}),\quad\text{for all}\quad 0\leq l\leq j\leq N-k.
\end{equation}
\end{lemma}
\begin{proof}
We prove this by induction on $l$. For $l=1$, this follows directly from the definition \eqref{eqn: pigeonholed boxes}. Assume that \eqref{eqn: pigeonholed boxes representation} holds for some $l<j$. Let $\tau_k\in\calt^\Lambda_k(\Gamma)$. We wish to prove
\begin{equation}
    \label{eqn: pigeonholed boxes rep induction step}
    \calt^\Lambda_{k+j}(\Gamma;\tau_k)=\bigcup_{\tau_{k+l+1}\in\calt^\Lambda_{k+l+1}(\Gamma;\tau_k)}\calt^\Lambda_{k+j}(\Gamma;\tau_{k+l+1}).
\end{equation}
Using \eqref{eqn: pigeonholed boxes} with $j=l+1$, we have $$\calt^\Lambda_{k+l+1}(\Gamma;\tau_k)=\bigcup_{\tau_{k+1}\in\calt^\Lambda_{k+1}(\Gamma;\tau_k)}\calt^\Lambda_{k+l+1}(\Gamma;\tau_{k+1}),$$ using which the right-hand side of \eqref{eqn: pigeonholed boxes rep induction step} can be written as the double union 
\begin{equation}
    \label{eqn: double union identity}
    \bigcup_{\tau_{k+l+1}\in\calt^\Lambda_{k+l+1}(\Gamma;\tau_k)}\calt^\Lambda_{k+j}(\Gamma;\tau_{k+l+1})=\bigcup_{\tau_{k+1}\in\calt^\Lambda_{k+1}(\Gamma;\tau_k)}\bigcup_{\tau_{k+l+1}\in\calt^\Lambda_{k+l+1}(\Gamma;\tau_{k+1})}\calt^\Lambda_{k+j}(\Gamma;\tau_{k+l+1}).
\end{equation}
Set $k':=k+1$, and $j':=j-1$. Then the inner union on the right-hand side of \eqref{eqn: double union identity} can be rewritten as $$\bigcup_{\tau_{k+l+1}\in\calt^\Lambda_{k+l+1}(\Gamma;\tau_{k+1})}\calt^\Lambda_{k+j}(\Gamma;\tau_{k+l+1})=\bigcup_{\tau_{k'+l}\in\calt^\Lambda_{k'+l}(\Gamma;\tau_{k'})}\calt^\Lambda_{k'+j'}(\Gamma;\tau_{k'+l})=\calt^\Lambda_{k'+j'}(\Gamma;\tau_{k'}),$$ where the last equality follows from the induction hypothesis. Substituting back the values of $k'$ and $j'$, and using the above in \eqref{eqn: double union identity}, we then get 
\begin{equation*}
    \label{eqn: inner union identity}
     \bigcup_{\tau_{k+l+1}\in\calt^\Lambda_{k+l+1}(\Gamma;\tau_k)}\calt^\Lambda_{k+j}(\Gamma;\tau_{k+l+1})=\bigcup_{\tau_{k+1}\in\calt^\Lambda_{k+1}(\Gamma;\tau_k)}\calt^\Lambda_{k+j}(\Gamma;\tau_{k+1})=\calt^\Lambda_{k+j}(\Gamma;\tau_k),
\end{equation*}
where the last equality follows from \eqref{eqn: pigeonholed boxes}. This is the desired identity \eqref{eqn: pigeonholed boxes rep induction step}.
\end{proof}
The following is an immediate consequence of \eqref{eqn: multiscale exceptional box properties}.
\begin{lemma}
    \label{lemma: chain of exceptional boxes}
    If there exists a chain of the form $$\tau_k\prec\tau_{k-1}\prec\dots\prec\tau_{m+1}\prec\tau_m,$$ where $\tau_j$ is exceptional for all $m<j\leq k$, then $\#\calt_k^\Lambda(\Gamma;\tau_m)=1$. Furthermore, $I\vert_{\tau_k}=I\vert_{\tau_m}$.
\end{lemma}
\begin{proof}
    The proof follows directly from the definition of exceptional boxes, and does not use the pigeonholing. We prove inductively that $\#\calt_j^\Lambda(\Gamma;\tau_m)=1$ and $I\vert_{\tau_j}=I\vert_{\tau_m}$ for all $m+1\leq j\leq k$.

    For the base case $j=m+1$, we have an exceptional box $\tau_{m+1}\in\calt_{m+1}(\Gamma;\tau_m)$. By \eqref{eqn: multiscale exceptional box properties} it follows immediately that $\#\calt^\Lambda_{m+1}(\Gamma;\tau_m)=1$ and $I\vert_{\tau_{m+1}}=I\vert_{\tau_m}$. Now suppose that $\#\calt_j^\Lambda(\Gamma;\tau_m)=1$ and $I\vert_{\tau_j}=I\vert_{\tau_m}$ for some $m+1\leq j<k$. By \eqref{eqn: pigeonholed boxes representation} we have $$\calt_{j+1}^\Lambda(\Gamma;\tau_m)=\bigcup_{\tau_j'\in\calt_j^\Lambda(\Gamma;\tau_m)}\calt_{j+1}^\Lambda(\Gamma;\tau_j').$$ 
    By the induction hypothesis, $\calt_j^\Lambda(\Gamma;\tau_m)=\{\tau_j\}$, and since $\tau_{j+1}\prec\tau_j$ is exceptional, by \eqref{eqn: multiscale exceptional box properties} we have $\#\calt_{j+1}^\Lambda(\Gamma;\tau_j)=1$. Combining these two observations, and using it in the identity above, we get $$\#\calt^\Lambda_{j+1}(\Gamma;\tau_m)=\#\calt^\Lambda_{j+1}(\Gamma;\tau_j)=1.$$ Again by \eqref{eqn: multiscale exceptional box properties}, we know that $I\vert_{\tau_{j+1}}=I\vert_{\tau_j}$; and by the induction hypothesis we have $I\vert_{\tau_j}=I\vert_{\tau_m}$. This completes the induction step.
\end{proof}
With the notation introduced above, we also record the following representations of the branches:
\begin{equation}
    \label{eqn: f^Lambda representation 2}
    f^\Lambda=\sum_{\tau_k\in\calt^\Lambda_k(\Gamma)}f^\Lambda_{\tau_k}\quad\text{for all}\quad 0\leq k\leq N,
\end{equation}
and \begin{equation}
    \label{eqn: f^Lambda_tau_k representation 2}
    f^\Lambda_{\tau_k}=\sum_{\tau_{k+j}\in\calt^\Lambda_{k+j}(\Gamma;\tau_k)}f^\Lambda_{\tau_{k+j}}\quad\text{for all}\quad \tau_k\in\calt^\Lambda_k(\Gamma)\quad\text{and all}\quad 0\leq j\leq N-k,\,0\leq k\leq N.
\end{equation}
Each branch $f^\Lambda$ takes into account precisely those $\tau_k\in\calt_k$ that satisfy $\lambda_k\leq\len(\tau_k)\leq2\lambda_k$. Thus, by summing over dyadic sequences, we have $f=\sum_{\Lambda\in(2^\mathbb{Z})^N}f^\Lambda$.  Let us define 
$$\Z_{\mathrm{dyad}}(f,R,\epsilon):=\{\Lambda\in (2^\mathbb{Z})^N:f^\Lambda\not\equiv 0\}.$$ We observe the following.
\begin{lemma}
    \label{lemma: number of branches}
    We have $$\#\Z_{\mathrm{dyad}}(f,R,\epsilon)\leq (1+\log R)^{1/\epsilon}.$$ 
\end{lemma}
\begin{proof}
    Let $\Lambda\in\Z(f,R,\epsilon)$. From \eqref{eqn: f^Lambda representation 2} it follows that $\calt_k^\Lambda(\Gamma)$ must be non-empty for all $0\leq k\leq N$. Each $\tau_k\in\calt_k^\Lambda(\Gamma)$ satisfies $\lambda_k\leq\len(\tau_k)\leq 2\lambda_{k}$. But from \eqref{eqn: tau k are eccentric} we have $$\len(\tau_k)\geq R^{-(k-2)\epsilon},\quad\text{and so}\quad \lambda_k\geq (1/2)R^{-(k-2)\epsilon}.$$ 
    On the other hand, $$\len(\tau_k)=\cal(\Gamma,I\vert_{\tau_k})\leq\cal(\Gamma,I)=R^{2\epsilon},\quad\text{so that}\quad \lambda_k\leq R^{2\epsilon}.$$
   Now the above bounds must hold for all $0\leq k\leq N$ in order for $f^\Lambda$ to be non-zero, so that 
    \begin{equation}
        \label{eqn: lambda k absolute range}
        (1/2)R^{-1+2\epsilon}\leq\lambda_k\leq R^{2\epsilon}\quad\text{for all}\quad 0\leq k\leq N.
    \end{equation} 
    Recall that each $\lambda_k$ is a dyadic number, and the number of dyadic reals in the interval $[(1/2)R^{-1+2\epsilon},R^{2\epsilon}]$ is $1+\log R$, so the result follows immediately.
\end{proof}
Since $f=\sum_{\Lambda\in \Z_{\mathrm{dyad}}(f,R,\epsilon)}f^\Lambda$, and $\Z_{\mathrm{dyad}}(f,R,\epsilon)\leq (1+\log R)^{1/\epsilon}$, in order to prove \eqref{eqn: convex decoupling ver 2}, it suffices to prove an analogous estimate for $f^\Lambda$ for all $\Lambda\in\Z_{\mathrm{dyad}}(f,R,\epsilon)$. For technical reasons, we actually prove it for all $\Lambda\in\Z(f,R,\epsilon)$, where $$\Z(f,R,\epsilon):=\{\Lambda\in\R_+^N:f^\Lambda\not\equiv 0\}.$$ In light of Lemma \ref{lemma: number of boxes}, Proposition \ref{prop: admissible decoupling} follows trivially by Cauchy--Schwarz when $R\lesssim_\epsilon 1$. Thus, we focus only on the case $R\gg_\epsilon 1$, which will be implicitly assumed from this point on.\footnote{$R\geq \epsilon^{-100/\epsilon}$ would certainly work.} 
\begin{restatable}{theorem}{pigeonholed}
    \label{thm: convex decoupling ver 2 pigeonholed}
    Let $\Gamma$ be an admissible curve that satisfies (C2) with an equality. There exists a uniform constant $C_\epsilon\geq 1$ such that for each $\Lambda\in\Z(f,R,\epsilon)$ and each $0\leq N\leq 1/\epsilon$, we have 
    \begin{equation}
    \label{eqn: convex decoupling ver 2 pigeonholed}
    \|f^\Lambda\|_{L^6(B_R)}^6\leq C_\epsilon R^{20\epsilon}\big(\sum_{\tau_N\in\calt^\Lambda_N(\Gamma)}\|f^\Lambda_{\tau_N}\|_{L^\infty(\R^2)}^2\big)^2\big(\sum_{\tau_N\in\calt^\Lambda_N(\Gamma)}\|f^\Lambda_{\tau_N}\|_{L^2(\R^2)}^2\big),
\end{equation}
for all $R\times R$ square $B_R$.
\end{restatable}
\subsection{Properties of the pigeonholed setup}
The pigeonholing process is a crucial step in the argument, as the high/low technique cannot be applied directly to the function $f$. Rather, we apply it to each branch $f^\Lambda$, where each $\tau_k$ involved has length about $\lambda_k$. Here, we discuss the properties of the pigeonholed setup that are instrumental for the high/low argument.

A key ingredient is the property \eqref{eqn: len tau k+1 relative}, which when combined with the pigeonholing above, gives the following relationship between successive scales 
\begin{equation}
    \label{eqn: lambda k relative range}
    \lambda_{k+1}\geq (1/2)R^{-\epsilon}\lambda_k.
\end{equation}
 The relation \eqref{eqn: lambda k relative range}, while very mild\footnote{By Lemma \ref{lemma: number of boxes}, it follows that the average length of a box $\tau_{k+1}\in\calt_{k+1}(\tau_k)$ is at least $R^{-\epsilon/2}\lambda_k$, which is quite larger than $R^{-\epsilon}\lambda_k$. When $\Gamma$ is the parabola, the average length coincides with $\len(\tau_{k+1})$ for each $\tau_{k+1}\in\calt_{k+1}(\tau_k)$.}, plays an important part in the argument.
Another important observation is that the boxes $\tau_k$ have separated directions. 
\begin{lemma}
    [Direction separation]
    \label{lemma: direction separation tau_k}
    Let $k\in\{1,\dots,N\}$, and let $\tau_k^1,\dots,\tau_k^4\in\calt^\Lambda_k(\Gamma)$ be four neighbouring boxes ordered left to right. If $I\vert_{\tau_k^j}=[t_k^j,v_k^j]$, then
    \begin{equation}
        \label{eqn: angle separation tau_k}
        \gamma_R'(t_k^4)-\gamma_R'(t_k^1)\geq (2\lambda_kR_k)^{-1}.
    \end{equation}
\end{lemma}
\begin{proof}
    The proof is divided into cases according to the nature of the box $\tau^2_k$.

    \medskip\noindent\underline{Case 1:} $\tau^2_k$ is not exceptional. The idea is to apply \eqref{eqn: angle sep tau k} to the box $\tau^2_k$, noting that 
  \begin{equation}
      \label{eqn: tau_k^2}
      t_k^2>t_k^1\quad\text{and}\quad v_k^2<t_k^4.
  \end{equation}
 Since \eqref{eqn: angle sep tau k} has two cases, we further break this situation into two subcases.
 
  \medskip\noindent\underline{Subcase 1A:} $\tau_k^2$ is a left box. By the first inequality in \eqref{eqn: angle sep tau k}, we have $$\W(\Gamma,[t_k^2,v_k^2+\eta])>R_k^{-1}\quad\text{for all}\quad \eta>0.$$  
  By \eqref{eqn: tau_k^2} we have $t_k^2\geq t_k^1$, and there exists $\eta>0$ such that $v_k^2+\eta\leq t_k^4$. By convexity it follows that $$\cal(\Gamma,[t_k^2,v_k^2+\eta])(\gamma_L'(t_k^4)-\gamma_R'(t_k^1))>R_k^{-1}\quad\text{for all $\eta>0$ sufficiently small}.$$ Recall that $\len(\tau_k^2)=\cal(\Gamma,[t_k^2,v_k^2])$. Taking limits as $\eta\to 0^+$, we get $$\len(\tau_k^2)(\gamma_L'(t_k^4)-\gamma_R'(t_k^1))\geq R_k^{-1}.$$ By the pigeonholing, we have $\len(\tau^2_k)\leq 2\lambda_k$, and so \eqref{eqn: angle separation tau_k} follows in this case.

  \medskip\noindent\underline{Subcase 1B:} $\tau_k^2$ is a right box. Here we apply the second inequality in \eqref{eqn: angle sep tau k} to get $$\W(\Gamma,[t_k^2-\eta,v_k^2])>R_k^{-1}\quad\text{for all}\quad \eta>0.$$ 
  By \eqref{eqn: tau_k^2} we have $v_k^2\leq t_k^4$, and there exists $\eta>0$ for which $t_k^2\geq t_k^1+\eta$. Again by convexity, we get $$\cal(\Gamma,[t_k^2-\eta,v_k^2])(\gamma_L'(t_k^4)-\gamma_R'(v_k^4))>R_k^{-1}\quad\text{for all $\eta>0$ sufficiently small}.$$ Taking limits as $\eta\to 0^+$ and using the pigeonholing we get the desired conclusion.

\medskip\noindent\underline{Case 2:} $\tau^2_k$ is exceptional. To deal with this case, we make the following observations. Suppose $\tau_k\in\calt^\Lambda_k(\Gamma)$ is an exceptional box. We record the following special case of \eqref{eqn: pigeonholed boxes representation}: $$\calt_k^\Lambda(\Gamma)=\bigcup_{\tau_j\in\calt_j^\Lambda(\Gamma)}\calt_k^\Lambda(\Gamma;\tau_j)\quad\text{for all}\quad j\leq k.$$ In other words, $\tau_k\prec\tau_j(\tau_k)$ for some $\tau_j(\tau_k)\in\calt_j^\Lambda(\Gamma)$ for all $j\leq k$. This yields a chain of boxes $$\tau_k\prec\tau_{k-1}(\tau_k)\prec\dots\prec\tau_1(\tau_k)\prec\tau_0.$$ If $\tau_j(\tau_k)$ is exceptional for all $1\leq j\leq k-1$, then by Lemma \ref{lemma: chain of exceptional boxes}, we get $\#\calt^\Lambda_k=\#\calt_k^\Lambda(\Gamma;\tau_0)=1$. By the hypothesis of this lemma, this cannot be the case. Therefore, there exists a maximal index $1\leq m\leq k-1$ for which $\tau_m(\tau_k)$ is typical. Another application of Lemma \ref{lemma: chain of exceptional boxes} then yields $I\vert_{\tau_m(\tau_k)}=I\vert_{\tau_k}$.

To prove \eqref{eqn: angle separation tau_k} in this case, we apply the above observation to the box $\tau_k=\tau_k^2$. Suppose that $\tau_m(\tau_k^2)$ is a left box, and follow the proof of Subcase 1A. By the first inequality in \eqref{eqn: angle sep tau k}, we have $$\W(\Gamma,[t_k^2,v_k^2+\eta])>R_m^{-1},$$ using the fact that $I\vert_{\tau_m(\tau_k^2)}=I\vert_{\tau_k^2}=[t_k^2,v_k^2]$. This leads to $$\len(\tau_k^2)(\gamma_L'(t^4_k)-\gamma_R'(t^1_k))\geq R_m^{-1},$$ and by the pigeonholing $$\gamma_L'(t^4_k)-\gamma_R'(t^1_k)\geq (2\lambda_kR_m)^{-1}.$$ But since $m\leq k$, we have $R_m\leq R_k$, which proves the lemma.
\end{proof}
Finally, we show that the pigeonholed structure is preserved under the similarity transformations $\cal_{\tau_k}$. This fact will be important for a rescaling argument in \S\ref{sec: broad/narrow}.
Recall the definitions of $\cal_{\tau_k}$ and $\Gamma_{\tau_k}$ from Step $k+1$ of the multi-scale algorithm \S\ref{sec: multi-scale algorithm}. Recall also the collections $\calt_{k+j}^\Lambda(\Gamma;\tau_k)$ defined in \eqref{eqn: pigeonholed boxes}.
\begin{lemma}[Rescaling]
\label{lemma: rescaling}
Let $\Gamma$ be any admissible curve, $\Lambda=(\lambda_i)_{i=1}^N$ be a positive sequence, and fix $\tau_k\in\calt^\Lambda_k(\Gamma)$. Let $\Lambda':=(R_k\lambda_{i+k})_{i=1}^N$.  Then for all $0\leq l\leq j\leq N-k$, the transformation $\cal_{\tau_k}$ induces a bijection between the collections $\{\calt_{k+j}^\Lambda(\Gamma;\tau_{k+l}):\tau_{k+l}\in\calt^\Lambda_{k+l}(\Gamma;\tau_k)\}$  and $\{\calt^{\Lambda'}_{j}(\Gamma_{\tau_k};\tilde\tau_l):\tilde\tau_l\in\calt_l^{\Lambda'}(\Gamma_{\tau_k})\}$ via $\cal_{\tau_k}:\tau_{k+l}\mapsto\tilde\tau_l$.
\end{lemma}

In the last line, we have implicitly made the assertion that $\cal_{\tau_k}(\tau_{k+l})\in\calt_l^{\Lambda'}(\Gamma_{\tau_k})$. This follows explicitly from the case $j=l$ of this lemma. Indeed, for $j=l$ the first collection reduces to $$\{\calt_{k+j}^\Lambda(\Gamma;\tau_{k+l}):\tau_{k+l}\in\calt^\Lambda_{k+l}(\Gamma;\tau_k)\}=\{\tau_{k+j}:\tau_{k+j}\in\calt^\Lambda_{k+j}(\Gamma;\tau_k)\}=\calt^\Lambda_{k+j}(\Gamma;\tau_k),$$ and the second collection reduces to $$\{\calt^{\Lambda'}_{j}(\Gamma_{\tau_k};\tilde\tau_l):\tilde\tau_l\in\calt_l^{\Lambda'}(\Gamma_{\tau_k})\}=\{\tilde\tau_l:\tilde\tau_l\in\calt^{\Lambda'}_l(\Gamma_{\tau_k})\}=\calt^{\Lambda'}_l(\Gamma_{\tau_k}).$$
The lemma then claims that (by an abuse of notation) $$\cal_{\tau_k}(\calt^\Lambda_{k+j}(\Gamma;\tau_k))=\{\cal_{\tau_k}(\tau_{k+j}):\tau_{k+j}\in\calt_{k+j}(\Gamma;\tau_k)\}=\calt^{\Lambda'}_l(\Gamma_{\tau_k}).$$
For notational convenience, we will make similar abuse throughout the proof.
\begin{proof}
    We prove this by double induction on $j$ and $l$. Let $P(l,j)$ denote the statement that the result holds for the pair of indices $(l,j)$ for all admissible curves $\Gamma$, all positive sequences $\Lambda$, and all $\tau_k\in\calt_k^\Lambda(\Gamma)$, for all $0\leq k\leq N$. First, we prove the diagonal cases $P(j,j)$ as follows: $$P(0,0)\implies P(1,1)\implies\dots\implies P(j,j).$$ Then we prove the sub-diagonal cases $P(l,j)$ as follows: $$P(j,j)\implies P(j-1,j)\implies\dots\implies P(l,j).$$
    
\noindent\underline{Diagonal case $P(j,j)$:}  
    By the remark appearing immediately before the proof, the two collections reduce to $\calt^\Lambda_{k+j}(\Gamma;\tau_k)$ and $\calt^{\Lambda'}_j(\Gamma_{\tau_k})$. Since $\cal_{\tau_k}$ is bijective, our goal is to show 
    \begin{equation}
        \label{eqn: diagonal case goal}
   \calt^\Lambda_{k+j}(\Gamma;\tau_k)= \cal_{\tau_k}^{-1}(\calt^{\Lambda'}_j(\Gamma_{\tau_k})).        
    \end{equation}
    We prove this by inducting on $j$. For $j=0$, the desired identity becomes $$\calt_k^\Lambda(\Gamma;\tau_k)=\mathcal{L}_{\tau_k}^{-1}(\calt_0^{\Lambda'}(\Gamma_{\tau_k})).$$ By Step 0 of the multi-scale algorithm, $\calt_0^{\Lambda'}(\Gamma_{\tau_k})$ consists of the single unit-scale box $[0,\cal(\Gamma_{\tau_k},I_{\tau_k})]\times[-1,1]$. By \eqref{eqn: length scaling} and the definitions of $\len(\tau_k),\wid(\tau_k)$ we have $$\cal(\Gamma_{\tau_k},I_{\tau_k})=R_k\cal(\Gamma,I\vert_{\tau_k})=\len(\tau_k)/\wid(\tau_k).$$ But by definition (see \eqref{eqn: cal tau k def}), $\cal_{\tau_k}$ maps $[0,\len(\tau_k)/\wid(\tau_k)]\times[-1,1]$ onto $\tau_k$. Since $\tau_k\in\calt^\Lambda(\Gamma)$, we have $\calt_k^{\Lambda}(\Gamma;\tau_k)=\{\tau_k\}$, and $P(0,0)$ follows. 
    
    Let us also prove $P(1,1)$, that is 
    \begin{equation}
        \label{eqn: P(1,1)}
        \calt^\Lambda_{k+1}(\Gamma;\tau_k)=\cal_{\tau_k}^{-1}(\calt_1^{\Lambda'}(\Gamma_{\tau_k})).
    \end{equation}
     By the definition given in \eqref{eqn: calt k+1 def}, we have $\calt_{k+1}(\Gamma;\tau_k)=\cal_{\tau_k}^{-1}(\calt_1(\Gamma_{\tau_k}))$. Rewriting \eqref{eqn: wid k+1 vs wid k} we get $\len(\cal_{\tau_k}(\tau_{k+1}))=R_k\len(\tau_{k+1})$. From here it follows that
     \begin{align*}
         \cal_{\tau_k}^{-1}(\calt_1^{\Lambda'}(\Gamma_{\tau_k}))&=\{\cal_{\tau_k}^{-1}(\tilde\tau_1):\tilde\tau_1\in\calt_1(\Gamma_{\tau_k}),\,R_k\lambda_{k+1}\leq \len(\tilde\tau_1)\leq 2R_k\lambda_{k+1}\}\\
         &=\{\tau_{k+1}:\tau_{k+1}\in\calt_{k+1}(\Gamma;\tau_k),\,\lambda_{k+1}\leq\len(\tau_{k+1})\leq 2\lambda_{k+1}\}\\
         &=\calt^\Lambda_{k+1}(\Gamma;\tau_k),
     \end{align*}
 where the first and last equality follow from \eqref{eqn: calt k+1 tau k lambda k+1 definition} and Notation \ref{not: pigeonholed boxes}.
 
    Supposing we have $P(j,j)$ for some $0\leq j\leq N-k-1$, we wish to show that $P(j+1,j+1)$ holds; that is,
    \begin{equation}
        \label{eqn: self-similarity induction step}
        \calt_{k+j+1}^\Lambda(\Gamma;\tau_k)=\mathcal{L}_{\tau_k}^{-1}(\calt_{j+1}^{\Lambda'}(\Gamma_{\tau_k})).
    \end{equation}
    By the definition given in \eqref{eqn: pigeonholed boxes}, we have  
    \begin{equation}
        \label{eqn: self-similarity step 1}
        \calt_{k+j+1}^\Lambda(\Gamma;\tau_k)=\bigcup_{\tau_{k+1}\in\calt_{k+1}^\Lambda(\Gamma;\tau_k)}\calt^\Lambda_{k+j+1}(\Gamma;\tau_{k+1}).
    \end{equation}
    For each $\tau_{k+1}\in\calt^\Lambda_{k+1}(\Gamma;\tau_k)$, we apply the induction hypothesis and obtain 
    \begin{equation}
        \label{eqn: self-similarity step 2}
        \calt_{k+j+1}^\Lambda(\Gamma;\tau_{k+1})=\calt^\Lambda_{(k+1)+j}(\Gamma;\tau_{k+1})=\cal_{\tau_{k+1}}^{-1}(\calt_j^{\Lambda''}(\Gamma_{\tau_{k+1}})),
    \end{equation}
    where $\Lambda'':=(R_{k+1}\lambda_{i+k+1})_{i=1}^N$.
    By definition \eqref{eqn: calt k+1 def}, $\calt_{k+1}(\Gamma;\tau_k)=\mathcal{L}_{\tau_k}^{-1}(\calt_1(\Gamma_{\tau_k}))$, and so each $\tau_{k+1}\in\calt^\Lambda_{k+1}(\Gamma;\tau_k)$ is of the form $$\tau_{k+1}=\mathcal{L}_{\tau_k}^{-1}(\sigma_{k+1}),\quad \lambda_{k+1}\leq\len(\tau_{k+1})\leq 2\lambda_{k+1},$$ for some $\sigma_{k+1}\in\calt_1(\Gamma_{\tau_k})$. Now $\cal_{\tau_{k}}$ scales isotropically by $R_k$, and so by the definition of $\Lambda'$, it follows that $\sigma_{k+1}\in\calt_1^{\Lambda'}(\Gamma_{\tau_k})$. 
   By Lemma \ref{lemma: cal composition law}, we know that the following diagram commutes \begin{figure}[H]
       \centering
       \includegraphics[width=0.35\linewidth]{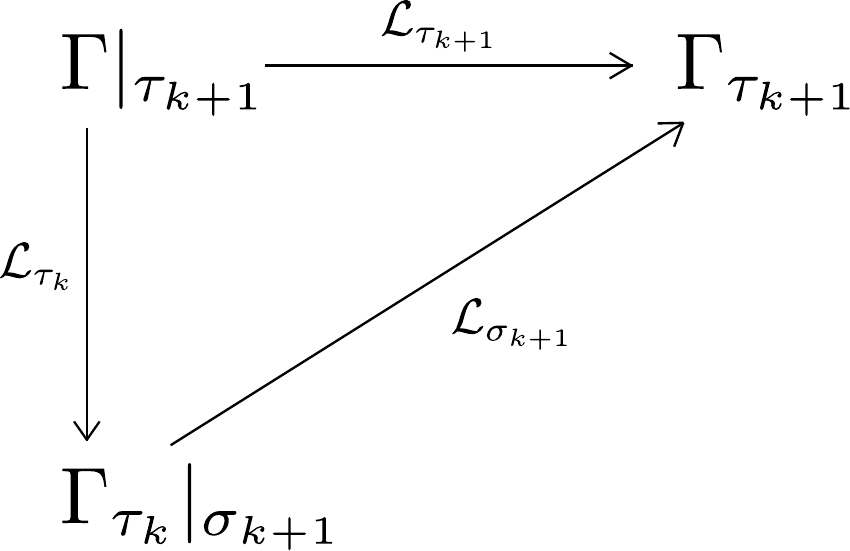}
\end{figure}
   \noindent  In particular, this means that 
   \begin{equation}
       \label{eqn: Gamma tau k+1 representation}
       \Gamma_{\tau_{k+1}}=\cal_{\sigma_{k+1}}(\Gamma_{\tau_k}\vert_{\sigma_{k+1}})=:(\Gamma_{\tau_k})_{\sigma_{k+1}}.
   \end{equation} 
    It also follows that 
    \begin{equation}
        \label{eqn: composition law application 1}
        \cal_{\tau_k}^{-1}\cal_{\sigma_{k+1}}^{-1}(\calt_j^{\Lambda''}(\Gamma_{\tau_{k+1}}))=\cal_{\tau_{k+1}}^{-1}(\calt^{\Lambda''}_j(\Gamma_{\tau_{k+1}})).
    \end{equation}
    By \eqref{eqn: wid k+1 vs wid k}, we get $$R_1\lambda'_{i+1}=R_{k+1}R_k^{-1}\lambda'_{i+1}=R_{k+1}\lambda_{i+k+1}=\lambda_i'',$$ using the definition of $\Lambda'$ as well as $\Lambda''$.  
   In light of this identity and \eqref{eqn: Gamma tau k+1 representation}, we apply our induction hypothesis to the curve $\Gamma_{\tau_k}$, and the transformation $\cal_{\sigma_{k+1}}$, which gives $$\calt_{j+1}^{\Lambda'}(\Gamma_{\tau_k};\sigma_{k+1})=\cal_{\sigma_{k+1}}^{-1}(\calt_j^{\Lambda''}((\Gamma_{\tau_{k}})_{\sigma_{k+1}}))=\cal_{\sigma_{k+1}}^{-1}(\calt_j^{\Lambda''}(\Gamma_{\tau_{k+1}})).$$ Recalling that $\sigma_{k+1}=\cal_{\tau_k}(\tau_{k+1})$, and combining the above observations we get
    \begin{equation}
        \label{eqn: self-similarity step 3}
        \cal_{\tau_k}^{-1}(\calt_{j+1}^{\Lambda'}(\Gamma_{\tau_k};\cal_{\tau_k}(\tau_{k+1})))=\cal_{\tau_k}^{-1}(\calt_{j+1}^{\Lambda'}(\Gamma_{\tau_k};\sigma_{k+1}))=\cal_{\tau_k}^{-1}\cal_{\sigma_{k+1}}^{-1}(\calt_j^{\Lambda''}(\Gamma_{\tau_{k+1}}))=\cal_{\tau_{k+1}}^{-1}(\calt_j^{\Lambda''}(\Gamma_{\tau_{k+1}})).
    \end{equation}
    Using \eqref{eqn: self-similarity step 3} and \eqref{eqn: self-similarity step 2} in \eqref{eqn: self-similarity step 1}, we get 
    \begin{equation}
        \label{eqn: self-similarity step 4}
        \calt_{k+j+1}^\Lambda(\Gamma;\tau_k)=\cal_{\tau_k}^{-1}\big(\bigcup_{\tau_{k+1}\in\calt^\Lambda_{k+1}(\Gamma;\tau_k)}\calt_{j+1}^{\Lambda'}(\Gamma_{\tau_k};\cal_{\tau_k}(\tau_{k+1})\big).
    \end{equation}
    Now we apply the case $P(1,1)$, so that by \eqref{eqn: P(1,1)}, we get
   $$\bigcup_{\tau_{k+1}\in\calt_{k+1}^\Lambda(\Gamma;\tau_k)}\calt_{j+1}^{\Lambda'}(\Gamma_{\tau_k};\cal_{\tau_k}(\tau_{k+1}))=\bigcup_{\tilde\tau_1\in\calt_1^{\Lambda'}(\Gamma_{\tau_k})}\calt_{j+1}^{\Lambda'}(\Gamma_{\tau_k};\tilde\tau_1)=\calt_{j+1}^{\Lambda'}(\Gamma_{\tau_k}),$$ where the last equality follows directly from \eqref{eqn: pigeonholed boxes}. Using the above in \eqref{eqn: self-similarity step 4} establishes \eqref{eqn: self-similarity induction step}, and therefore $P(j+1,j+1)$.
   
\medskip\noindent\underline{Sub-diagonal case $P(l,j)$:} Fix $0\leq j\leq N-k$. We prove $P(l,j)$ holds for all $0\leq l\leq j$, by induction on $l$. By the diagonal case, we have already established the correspondence between the collections $\calt_{k+l}^\Lambda(\Gamma;\tau_k)$ and $\calt_l^{\Lambda'}(\Gamma_{\tau_k})$ via $\cal_{\tau_k}$. Specifically, for each $\tau_{k+l}\in \calt_{k+l}^\Lambda(\Gamma;\tau_k)$, its image $\cal_{\tau_k}(\tau_{k+l})\in \calt_l^{\Lambda'}(\Gamma_{\tau_k})$ and, conversely, for each $\tilde\tau_l\in\calt_l^{\Lambda'}(\Gamma_{\tau_k})$, its pre-image $\cal_{\tau_k}^{-1}(\tilde\tau_l)\in \calt_{k+l}^\Lambda(\Gamma;\tau_k)$. Thus, the goal is to now show that 
\begin{equation}
    \label{eqn: sub-diagonal case goal}
        \cal_{\tau_k}(\calt_{k+j}^\Lambda(\Gamma;\tau_{k+l}))=\calt^{\Lambda'}_{j}(\Gamma_{\tau_k};\cal_{\tau_k}(\tau_{k+l}))\quad\text{for all}\quad\tau_{k+l}\in\calt^\Lambda_{k+l}(\Gamma;\tau_k).
\end{equation}

We take $l=j$ as our base case, which holds by the diagonal case $P(j,j)$ treated above. Suppose $P(l+1,j)$ holds for some $1\leq l+1\leq j$. We now show that $P(l,j)$ holds. By the definition \eqref{eqn: pigeonholed boxes}, we have $$\calt^\Lambda_{k+j}(\Gamma;\tau_{k+l})=\bigcup_{\tau_{k+l+1}\in\calt_{k+l+1}^\Lambda(\Gamma;\tau_{k+l})}\calt_{k+j}^\Lambda(\Gamma;\tau_{k+l+1}).$$
Applying $\cal_{\tau_k}$ to both sides yields the following 
\begin{equation}
    \label{eqn: sub-diagonal penultimate step}
     \cal_{\tau_k}(\calt_{k+j}^\Lambda(\Gamma;\tau_{k+l}))=\bigcup_{\tau_{k+l+1}\in\calt_{k+l+1}^\Lambda(\Gamma;\tau_{k+l})}\calt^{\Lambda'}_{j}(\Gamma_{\tau_k};\cal_{\tau_k}(\tau_{k+l+1})),
\end{equation}
 where we used the induction hypothesis $P(l+1,j)$ to infer $\cal_{\tau_k}(\calt_{k+j}^\Lambda(\Gamma;\tau_{k+l+1}))=\calt^{\Lambda'}_{j}(\Gamma_{\tau_k};\cal_{\tau_k}(\tau_{k+l+1}))$. Now consider $\tau_{k+l}\in\calt^\Lambda_{k+l}(\Gamma;\tau_k)$. By $P(l,l)$, we know that $\tau^k_l:=\cal_{\tau_k}(\tau_{k+l})\in\calt_l^{\Lambda'}(\Gamma_{\tau_k})$.
 
\medskip\noindent\textbf{Claim:} The following diagram commutes
\begin{figure}[H]
    \centering
    \includegraphics[width=0.35\linewidth]{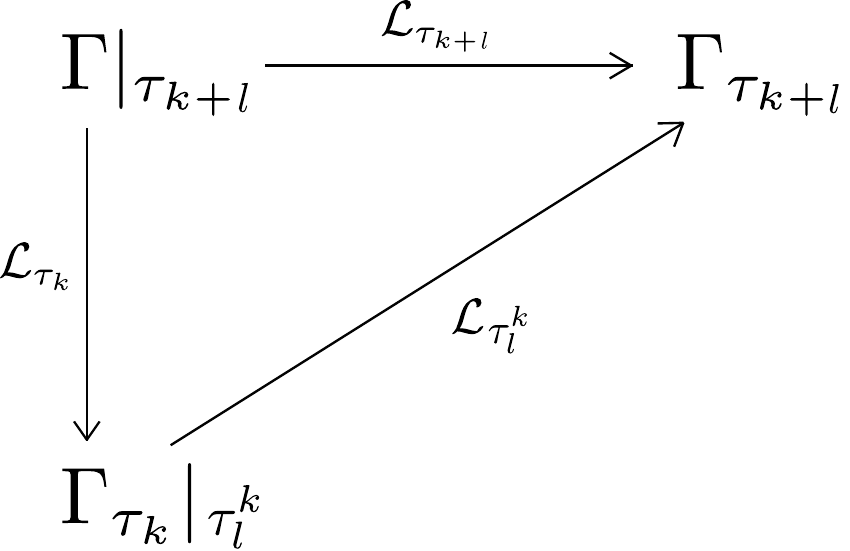}
\end{figure}
\noindent This can be shown by repeated applications of Lemma \ref{lemma: cal composition law}. A proof is included in the appendix.
In particular, this tells us 
\begin{equation}
    \label{eqn: Gamma tau k+l representation}
    \Gamma_{\tau_{k+l}}=\cal_{\tau^k_l}(\Gamma_{\tau_k}\vert_{\tau^k_l})=:(\Gamma_{\tau_k})_{\tau^k_l}.
\end{equation}
In light of this claim, by $P(1,1)$, we have 
\begin{equation}
    \label{eqn: diagonal application 1}
    \calt^\Lambda_{k+l+1}(\Gamma;\tau_{k+l})=\cal_{\tau_{k+l}}^{-1}(\calt^{\Lambda'''}_1(\Gamma_{\tau_{k+l}}))=\cal_{\tau_k}^{-1}\circ\cal_{\tau^k_l}^{-1}(\calt^{\Lambda'''}_1(\Gamma_{\tau_{k+l}})),
\end{equation}
 where $\Lambda''':=(R_{k+l}\lambda_{i+k+l})_{i=1}^N$. On the other hand, we have $$R_l\lambda_{i+l}'=R_{k+l}R_k^{-1}\lambda_{i+l}'=R_{k+l}\lambda_{i+k+l}=\lambda_i'''.$$ In light of \eqref{eqn: Gamma tau k+l representation} and $P(1,1)$, we have
 \begin{equation}
     \label{eqn: diagonal application 2}
     \cal_{\tau^k_l}(\calt^{\Lambda'}_{l+1}(\Gamma_{\tau_k};\tau^k_l))=\calt^{\Lambda'''}_1((\Gamma_{\tau_k})_{\tau^k_l}))=\calt^{\Lambda'''}_1(\Gamma_{\tau_{k+l}}).
 \end{equation}
 Using \eqref{eqn: diagonal application 2} in \eqref{eqn: diagonal application 1}, we get $$\calt^\Lambda_{k+l+1}(\Gamma;\tau_{k+l})=\cal_{\tau_k}^{-1}(\calt^{\Lambda'}_{l+1}(\Gamma_{\tau_k};\tau^k_l)),$$ where we recall that $\tau^k_l=\cal_{\tau_k}(\tau_{k+l})\in\calt^{\Lambda'}_l(\Gamma_{\tau_k})$. Using the above in \eqref{eqn: sub-diagonal penultimate step}, we get $$ \cal_{\tau_k}(\calt_{k+j}^\Lambda(\Gamma;\tau_{k+l}))=\bigcup_{\tilde\tau_{l+1}\in\calt^{\Lambda'}_{l+1}(\Gamma_{\tau_k};\tau^k_l)}\calt^{\Lambda'}_{j}(\Gamma_{\tau_k};\tilde\tau_{l+1})=\calt_j^{\Lambda'}(\Gamma_{\tau_k};\tau^k_l),$$ where the last equality follows from \eqref{eqn: pigeonholed boxes}. Since $\tau^k_l=\cal_{\tau_k}(\tau_{k+l})$, this is precisely \eqref{eqn: sub-diagonal case goal}, which establishes $P(l,j)$.
\end{proof}
\begin{notation}\label{not: lambda gamma}
    The remainder of the paper is dedicated to the proof of Theorem \ref{thm: convex decoupling ver 2 pigeonholed}. As such, we fix an admissible curve $\Gamma$ over an interval $I$ and a sequence $\Lambda\in\Z(f,R,\epsilon)$, and keep these fixed until \S\ref{sec: broad/narrow}. While $\Gamma$ and $\Lambda$ remain fixed, for notational convenience, we choose to suppress the dependence on $\Gamma$ and $\Lambda$ wherever they should appear until \S\ref{sec: broad/narrow}. For instance, we write $f^\Lambda$ as simply $f$, and $f^\Lambda_{\tau_k}$ as simply $f_{\tau_k}$. Similarly, we denote the collections $\calt_{k+1}(\Gamma;\tau_k;\lambda_{k+1})$ as simply $\calt_{k+1}(\tau_k)$, denote the collections $\calt^\Lambda_{k+j}(\Gamma;\tau_{k+l})$ as $\calt_{k+j}(\tau_{k+l})$, and denote the collections $\calt_k^\Lambda(\Gamma)$ as $\calt_k$. To also be consistent with the standard notation in the literature, we shall also denote $\tau_N\in\calt_N$ by $\theta$, and $\calt_N$ by $\Theta$.
\end{notation} With Notation \ref{not: lambda gamma}, \eqref{eqn: f^Lambda_{tau_k} def} can be rewritten as \begin{equation}
    \label{eqn: f_{tau_k} representation}
    f_{\tau_k}=\sum_{\tau_{k+1}\in\calt_{k+1}(\tau_k)}f_{\tau_{k+1}},\quad\text{for all}\quad \tau_k\in\calt_k;
\end{equation}
and \eqref{eqn: f^Lambda representation 2} can be rewritten as 
\begin{equation}
    \label{eqn: f representation}
    f=\sum_{\tau_k\in\calt_k}f_{\tau_k},\quad\text{for all}\quad 0\leq k\leq N.
\end{equation}
With this simplified notation, we record the following special case of Lemma \ref{lemma: pigeonholed boxes representation}, which we shall use frequently: 
\begin{equation}
    \label{eqn: calt k+1 representation}
    \calt_{k+1}=\bigcup_{\tau_{k}\in\calt_k}\calt_{k+1}(\tau_k).
\end{equation}
This follows by setting $k=0$, $j=k+1$ and $l=k$ in \eqref{eqn: pigeonholed boxes representation}. Finally, let us record the main estimate from Theorem \ref{thm: convex decoupling ver 2 pigeonholed} with this new notation: 
\begin{equation}
    \label{eqn: convex decoupling pigeonholed simplified notation}
     \|f\|_{L^6(B_R)}^6\leq C_\epsilon R^{20\epsilon}\big(\sum_{\theta\in\Theta}\|f_{\theta}\|_{L^\infty(\R^2)}^2\big)^2\big(\sum_{\theta\in\Theta}\|f_{\theta}\|_{L^2(\R^2)}^2\big).
\end{equation}
\subsection{Normalisation and error terms}\label{sec: normalising f}
Since the estimate \eqref{eqn: convex decoupling pigeonholed simplified notation} is homogeneous, we may replace $f$ by a (non-zero) constant multiple $cf$. We choose to normalise the $\ell^\infty L^\infty$ norm of $(f_\theta)_{\theta\in\Theta}$: \begin{equation}
    \label{eqn: normalisation in sup norm} 
    \max_{\theta\in\Theta}\|f_\theta\|_{L^\infty(\R^2)}=1.
\end{equation}
By the Bernstein inequality and the Fourier localisation of each piece $f_\theta$, we have the following.
\begin{lemma}[Bernstein]
\label{lemma: bernstein} For all $1\leq p\leq q\leq\infty$, and all $\theta$ we have $$\|f_\theta\|_{L^q(\R^2)}\lesssim|\theta|^{1/p-1/q}\|f_\theta\|_{L^p(\R^2)}.$$
\end{lemma}
If $\Gamma$ is an admissible curve satisfying (C2) with equality, then it follows from \eqref{eqn: lambda k absolute range}, that \begin{equation}
    \label{eqn: theta size}
    |\theta|\leq R^2\quad\text{for all}\quad \theta\in\Theta.
\footnote{We have the stronger conclusion $|\theta|\leq R^{2\epsilon-1}$, but choose to work with a weaker formulation to have some flexibility for the rescaling arguments appearing in \S\ref{sec: broad/narrow}.}\end{equation}
Using this in Lemma \ref{lemma: bernstein}, we get $$\max_\theta\|f_\theta\|_{L^\infty(\R^2)}\leq\big(\sum_{\theta\in\Theta}\|f_\theta\|_{L^\infty(\R^2)}^2\big)^{1/2}\lesssim R^{1/3}\big(\sum_{\theta\in\Theta}\|f_\theta\|_{L^6(\R^2)}^2\big)^{1/2}.$$ Combining the above with \eqref{eqn: normalisation in sup norm}, followed by H\"older's inequality we get 
\begin{equation}
    \label{eqn: error term bound}
R^{-2}\leq\big(\sum_{\theta\in\Theta}\|f_\theta\|_{L^6(\R^2)}^2\big)^{3}\leq\big(\sum_{\theta\in\Theta}\|f_{\theta}\|_{L^\infty(\R^2)}^2\big)^2\big(\sum_{\theta\in\Theta}\|f_{\theta}\|_{L^2(\R^2)}^2\big),
\end{equation}
which is the right-hand side of \eqref{eqn: convex decoupling pigeonholed simplified notation}.
The above will be used to control various error terms in our estimates by the right-hand expression above. The purpose of the normalisation is to express error terms in the standard form of $R^{-O(1)}$.

\section{Pruning, high/low decomposition, and a partition of the spatial domain}\label{sec: pruning, hi/low, omega sets}
In this section, we carry out the main argument of Guth--Maldague--Wang \cite{gmw} in our general setting. 
Let $K\geq 1$ be a large integer to be specified later.\footnote{The parameter $K$ will be determined by the proof of Lemma \ref{lemma: local bilinear square function}.} 
\begin{definition}[Transverse boxes]
    \label{def: transversality}
    For any $k\in\{1,\dots,N\}$, we write $\tau_{k}'\nsim\tau_{k}''$ if the boxes are separated by at least $\lceil K/4\rceil$ many neighbouring boxes $\tau_k\in\calt_k$. We write $\tau_{k}'\sim\tau_{k}''$ otherwise.
\end{definition}
 Let $B_R$ be an $R\times R$ square, and let $\alpha,\beta>0$. We define the set \begin{align}\label{eqn: U set definition}
    U_{\alpha,\beta}^{\text{br}}:=\bigg\{x\in B_R:&\alpha\leq|f(x)|\leq 2\alpha,\; \beta\leq\sum_{\theta\in\Theta}|f_\theta(x)|^2\leq 2\beta,|f(x)|\leq 2KR^{2\epsilon}\max_{\tau_1\nsim\tau_1'}|f_{\tau_1}(x)|^{1/2}|f_{\tau_1'}(x)|^{1/2},\nonumber\\&\big(\sum_{\tau_1\in\calt_1}|f_{\tau_1}(x)|^6\big)^{1/6}\leq 2KR^{2\epsilon}\max_{\substack{\tau_1,\tau_1'\in\calt_1\\\tau_1\nsim\tau_1'}}|f_{\tau_1}(x)|^{1/2}|f_{\tau_1'}(x)|^{1/2}\bigg\}.
\end{align}
In light of the normalisation \eqref{eqn: normalisation in sup norm}, by a pigeonholing argument and a broad/narrow analysis we can reduce the proof of \eqref{eqn: convex decoupling pigeonholed simplified notation} to proving the following weak type estimate. 
\begin{theorem}
    \label{thm: weak type}
    For each $R^{-1/2}\leq\alpha\leq R^{1/2}$, and each $R^{-1/2}\alpha^2\leq\beta\leq R^{1/2}$, we have 
    \begin{equation}
    \label{eqn: main broad}
    \alpha^6|U_{\alpha,\beta}^{\mathrm{br}}|\leq C_\epsilon R^{20\epsilon}\big(\sum_{\theta\in\Theta}\|f_\theta\|_{L^\infty(\R^2)}^2\big)^2\big(\sum_{\theta\in\Theta}\|f_\theta\|_{L^2(\R^2)}^2\big).
\end{equation}
\end{theorem}
The details of this reduction are discussed in \S\ref{sec: broad/narrow}. The proof of Theorem \ref{thm: weak type} is the gist of the paper, and the rest of this section sets up the main components of the proof. 
\begin{notation}
\label{not: G}
    Fix two quantities $\alpha$ and $\beta$ satisfying $R^{-1/2}\leq\alpha\leq R^{1/2}$, and  $R^{-1/2}\alpha^2\leq\beta\leq R^{1/2}$, and
    define $G:=A_\epsilon R^{5\epsilon}(\beta/\alpha)$, for some constant $A_\epsilon\geq 1$ to be specified later.\footnote{The constant $A_\epsilon$ will be determined by the proof of Lemma \ref{lemma: replacing f}.} Note that since $\beta\geq R^{-1/2}\alpha^2$ and $\alpha\geq R^{-1/2}$, we have $G\geq R^{-1}$.
\end{notation}
\subsection{Auxiliary functions} We define a number of auxiliary functions that appear frequently in our proof. 
\begin{definition}
    [Polar of a Convex body]
    We shall call a compact convex set $D\subset\R^2$ a \textit{convex body} if it has non-empty interior. We define its \textit{dual} or \textit{polar body} as $$D^*:=\{x\in\R^2:|x\cdot\xi|\leq 1\quad\text{for all}\quad \xi\in D\}.$$
\end{definition}
\begin{theorem}[John Ellipsoid Theorem \cite{john}]
For every convex body $D\subset\R^2$, there exists an ellipsoid $J(D)$ (called the John Ellipsoid of $D$) satisfying $$J(D)\subseteq D\subseteq 2\cdot J(D).$$ 
\end{theorem}

\begin{notation}\label{not: delta}
    Let $\delta:=\epsilon^{10}$. 
\end{notation}
\begin{definition}[Definition of $\eta_{k}$]
\label{def: eta}
Let $\eta\in\mathcal{S}(\R^2)$ be chosen to satisfy the following 
\begin{itemize}
    \item $\widehat{\eta}(\xi)=0$ for all $\xi\notin 2\B$;
    \item $\widehat{\eta}(\xi)=1$ for all $\xi\in\B$;
    \item $\widehat{\eta}$ is radially decreasing.
\end{itemize}
For each $k\in\{1,\dots,N-1\}$ and $\lambda_k$ as in \S\ref{sec: pigeonholing}, define $\eta_{k}$ as $$\widehat{\eta}_{k}(\xi):=\widehat{\eta}(\lambda_{k}^{-1}\xi).$$
\end{definition}
We record the following properties of $\eta_{k}$ which are immediate from the choice of $\eta$.
\begin{lemma}
\label{lemma: eta}
    For each $k\in\{1,\dots,N-1\}$, we have \begin{itemize}
        \item $\widehat{\eta}_{k}(\xi)=0$ for all $\xi\notin B(0,2\lambda_{k})$;
        \item $\widehat{\eta}_{k}(\xi)=1$ for all $\xi\in B(0,\lambda_{k})$;
        \item $\|\eta_{k}\|_{L^1(\R^2)}\lesssim 1$.
    \end{itemize}
\end{lemma}
\begin{definition}[Definition of $\rho$]
\label{def: rho}
     Let $\rho\in\mathcal{S}(\R^2)$ be defined as $\widehat{\rho}(\xi):=\widehat{\eta}(R^{-\delta}\xi)$. For each $k\in\{1,\dots,N\}$ and $\tau_k\in\calt_k$, let $A_{\tau_k}$ be an affine transformation mapping $2J(\tau_k)$ onto $32^{k-N}\B$ (the particular choice of the map is irrelevant). We define $\rho_{\tau_k}$ by $$\widehat{\rho}_{\tau_k}(\xi):=\widehat{\rho}(A_{\tau_k}\xi).$$ For each $k\in\{1,\dots,N\}$, we also define $A_k$ as the dilation mapping $\lambda_k\B$ onto $32^{k-N}\B$. Define $\rho_k$ by $$\widehat{\rho}_k(\xi):=\widehat{\rho}(A_k\xi).$$  
\end{definition}
 We list out the relevant properties of these functions which follow immediately from their definition. 
\begin{lemma}
\label{lemma: rho}
   For each $\tau_k\in\calt_k$, we have
    \begin{itemize}
        \item $\widehat{\rho}_{\tau_k}(\xi)=0$ for all $\xi\notin 2R^\delta32^{N-k}\cdot\tau_k$;
        \item $\widehat{\rho}_{\tau_k}(\xi)=1$ for all $\xi\in R^\delta32^{N-k}\cdot\tau_k$;
        \item $\|\rho_{\tau_k}\|_{L^1(\R^2)}\lesssim 1$.
    \end{itemize}
    Also, 
    \begin{itemize}
        \item $\widehat{\rho}_k(\xi)=0$ for all $\xi\notin 2R^\delta\lambda_k32^{N-k}\B$;
        \item $\widehat{\rho}_k(\xi)=1$ for all $\xi\in R^\delta\lambda_k32^{N-k}\B$;
        \item $\|\rho_k\|_{L^1(\R^2)}\lesssim 1$.
    \end{itemize}
\end{lemma}
\begin{definition}
    [Definition of $w,\omega$]
    \label{def: w,omega}
    For each $\tau_k\in\calt_k$, we define $$w_{\tau_k}(x):=\sup_{y\in x+4\tau_k^*}|\rho_{\tau_k}(y)|.$$ 
   Define $\omega_{\tau_1}:=w_{\tau_1}$, and for each $k\geq 1$ we inductively define 
$$\omega_{\tau_{k+1}}:=\max\big(w_{\tau_{k+1}},\omega_{\tau_k},\,\omega_{\tau_k}*|\rho_{\tau_k}|*|\eta_{k+1}|\big)\quad\text{for all}\quad \tau_{k+1}\in\calt_{k+1}(\tau_{k}).$$ 
For each $k$, we also define $$w_k(x):=\sup_{y\in x+4R^\delta\lambda_k^{-1}\B}|\rho_k(y)|.$$
\end{definition}
\begin{lemma}
    \label{lemma: L^1 norm of omega} 
    For all $k$ and each $\tau_k\in\calt_k$, we have 
    \begin{equation}
        \label{eqn: w norm bound}
        \|w_{\tau_k}\|_{L^1(\R^2)}, \|\omega_{\tau_k}\|_{L^1(\R^2)}\lesssim_\epsilon R^{2\delta}.
    \end{equation}
    \end{lemma}
\begin{proof}
    We first prove the $L^1$ bound for the function $w_{\tau_k}$, by establishing the following pointwise bound
     \begin{equation}
        \label{eqn: w rapid decay}
        w_{\tau_k}(x)\lesssim_M R^{2\delta}|\det A_{\tau_k}|^{-1}(1+R^\delta(|A_{\tau_k}^{-t}x|-8\cdot 32^{N-k}))^{-M}\quad\text{for all}\quad x\notin 4J(\tau_k)^*.
    \end{equation}
    We have $$w_{\tau_k}(x)=\sup_{y\in x+4\tau_k^*}|\rho_{\tau_k}(y)|\leq \sup_{y\in x+4J(\tau_k)^*}|\rho_{\tau_k}(y)|,$$ since $\tau_k^*\subseteq J(\tau_k)^*$. Now $4J(\tau_k)^*=8\cdot 32^{N-k}A_{\tau_k}^t(\B)$ and $\rho_{\tau_k}(y)=|\det A_{\tau_k}|^{-1}\rho(A_{\tau_k}^{-t}y)$, so 
    \begin{equation}
        \label{eqn: w pointwise bound}
        w_{\tau_k}(A_{\tau_k}^tz)\leq|\det A_{\tau_k}|^{-1}\sup_{u\in z+8\cdot 32^{N-k}\B}|\rho(u)|,\quad\text{where}\quad z:=A_{\tau_k}^{-t}x.
    \end{equation}
    Now $\rho(u)=R^{2\delta}\eta(R^\delta u)$ and since $\eta\in\mathcal{S}(\R^2)$, we have $$|\eta(R^\delta u)|\lesssim_M (1+R^\delta|u|)^{-M}.$$ For all $u\in z+8\cdot 32^{N-k}\B$ we have $|u|\geq |z|-8\cdot 32^{N-k}$, so that by \eqref{eqn: w pointwise bound} we get 
    \begin{equation}
        \label{eqn: w rapid decay z version}
         w_{\tau_k}(A_{\tau_k}^tz)\lesssim_M R^{2\delta}|\det A_{\tau_k}|^{-1}\big(1+R^\delta(|z|-8\cdot 32^{N-k})\big)^{-M},\quad\text{for all}\quad z\notin 8\cdot 32^{N-k}\B,
    \end{equation} 
    which is just \eqref{eqn: w rapid decay} given in terms of the variable $z$. In order to obtain the $L^1$ bound for $\|w_{\tau_k}(x)\|_{L^1_x}$ from this, we use the change of variables $x=A_{\tau_k}^t z$ in the integral, and split the integration into the regions $\{|z|\leq 8\cdot 32^{N-k}\}$ and $\{|z|\geq 8\cdot 32^{N-k}\}$. Using \eqref{eqn: w rapid decay z version} in the latter region with $M=10$ gives
    \begin{equation}
        \label{eqn: rho norm estimate}
        \|w_{\tau_k}\|_{L^1}\lesssim_\epsilon R^{2\delta}\bigg(1+\int_{|z|\geq 8\cdot 32^{N-k}}\big(1+R^\delta(|z|-8\cdot 32^{N-k})\big)^{-10}\bigg)\lesssim R^{2\delta},
    \end{equation}   
    as required.
    
    Now we prove the $L^1$ bound for the functions $\omega_{\tau_k}$. By inducting on $k$, we will show that 
    \begin{equation}
        \label{eqn: omega k induction}
        \|\omega_{\tau_k}\|_{L^1(\R^2)}\leq (d_\epsilon+d_\epsilon^2+\dots+d_\epsilon^k)R^{2\delta},
    \end{equation} for some constant $d_\epsilon\geq 1$. For $k=1$, we have $\omega_{\tau_1}=w_{\tau_1}$, so the base case of \eqref{eqn: omega k induction} follows from above if $d_\epsilon$ is chosen sufficiently large. Now suppose we have the bound \eqref{eqn: omega k induction} for some $k\geq 1$. Observe that
    \begin{align*}
         \|\omega_{\tau_{k+1}}\|_{L^1}&\leq \|w_{\tau_{k+1}}\|_{L^1}+\|\omega_{\tau_k}\|_{L^1}+\|\omega_{\tau_k}*|\rho_{\tau_k}|*|\eta_{k+1}|\|_{L^1}.
    \end{align*}
    The first part of the proof established that $\|w_{\tau_{k+1}}\|_{L^1}\lesssim_\epsilon R^{2\delta}$. By Young's convolution inequality, Lemma \ref{lemma: rho}, and Lemma \ref{lemma: eta}, we have $$\|\omega_{\tau_k}*|\rho_{\tau_k}|*|\eta_{k+1}|\|_{L^1}\lesssim\|\omega_{\tau_k}\|_{L^1(\R^2)}.$$
    Hence, if we choose $d_\epsilon$ large enough, we get $$\|\omega_{\tau_{k+1}}\|_{L^1}\leq d_\epsilon R^{2\delta}+d_\epsilon\|\omega_{\tau_k}\|_{L^1}\leq (d_\epsilon+d_\epsilon^2+\dots+d_\epsilon^{k+1})R^{2\delta},$$ by the induction hypothesis.   
    \end{proof}
    \begin{lemma}
     For each $k$ and all $M\in\mathbb{N}$, we have  \begin{equation}
        \label{eqn: w_k rapid decay}
        w_{k}(x)\lesssim_M R^{2\delta}|\det A_{k}|^{-1}(1+R^\delta(|A_{k}^{-t}x|-4R^\delta\cdot 32^{N-k}))^{-M}\quad\text{for all}\quad x\notin 4R^\delta\lambda_k^{-1}\B.
    \end{equation}
\end{lemma}
\begin{proof}
   The bound \eqref{eqn: w_k rapid decay} is analogous to \eqref{eqn: w rapid decay} obtained in the proof above. Using the definitions of $w_k,\rho_k$ and arguing as above, we find that
   $$w_k(x)=|\det A_k|^{-1}\sup_{u\in z+4R^\delta 32^{N-k}\B}|\rho(u)|,\quad\text{where}\quad z:=A_k^{-t}x.$$ For all $u\in z+4R^\delta 32^{N-k}\B$, we have $|u|\geq |z|-4R^\delta32^{N-k}$. Using the fact $\rho(u)=R^{2\delta}\eta(R^\delta u)$ and the rapid decay of $\eta$, we then get $$\sup_{u\in z+4R^\delta 32^{N-k}\B}|\rho(u)|\lesssim_M R^{2\delta}\big(1+R^\delta(|z|-4R^\delta 32^{N-k})\big)^{-M},\quad\text{for all}\quad z\notin 4R^\delta 32^{N-k}\B.$$ From here, \eqref{eqn: w_k rapid decay} follows by substituting back $x=A_k^tz$.
\end{proof}
    \begin{definition}
        [Definition of $\phi$]
        \label{def: phi}
        Let $\phi\in\mathcal{S}(\R^2)$ be chosen to satisfy the following conditions 
        \begin{itemize}
            \item $\widehat{\phi}$ is supported in $\B$;
            \item $0\leq\phi(x)\leq 1$ for all $x\in\R^2$;
            \item $\phi(x)\gtrsim 1$ for all $x\in\B$.
        \end{itemize}
        Now for a convex body $D\subset\R^2$, we define $\widehat{\phi}_D(\xi):=\widehat{\phi}(A_D^{-t}\xi)$, where $A_D$ is the affine transformation mapping $2J(D)$ onto $\B$.
    \end{definition}
    We record the following list of properties satisfied by $\phi_D$, that follow easily from the above definition.
    \begin{lemma}\label{lemma: phi}
        For each convex body $D\subset\R^2$, we have 
        \begin{itemize}
            \item $\widehat{\phi}_D$ is supported in $(1/2)D^*$;
            \item $0\leq\phi_D(x)\leq 1$ for all $x\in\R^2$;
            \item $\phi_D(x)\gtrsim 1$ for all $x\in 2J(D)$.
        \end{itemize}
    \end{lemma}
    \begin{definition}
        [A partition of unity]
        \label{def: zeta} Let $\zeta\in\mathcal{S}(\R^2)$ satisfy the following conditions 
        \begin{itemize}
           \item $\zeta(x)=0$ for all $x\notin [-1,1]^2$;
            \item $\zeta(x)=1$ for all $x\in [-1/2,1/2]^2$;
            \item $0\leq\zeta(x)\leq 1$ for all $x\in\R^2$.
        \end{itemize}
        It follows that $$1\leq\sum_{n\in\mathbb{Z}^2}\zeta(x-n)\leq 9 \quad\text{for all}\quad x\in\R^2.$$
        For the square $Q_n:=n+[-1/2,1/2]^2$ centred at $n\in\mathbb{Z}^2$, define $$\psi_{Q_n}(x):=\frac{\zeta(x-n)}{\sum_{n\in\mathbb{Z}^2}\zeta(x-n)}.$$ It follows that \begin{itemize}
           \item $\psi_{Q_n}(x)=0$ for $x\notin 2Q_n$;
           \item $\sum_{n\in\mathbb{Z}^2}\psi_{Q_n}(x)=1$ for all $x\in\R^2$;
            \item $0\leq \psi_{Q_n}(x)\leq 1$ for all $x\in\R^2$.
        \end{itemize}
        Let $\T$ denote a disjoint family of congruent parallelograms tiling $\R^2$. For each $T\in\T$, let $A_T$ denote some affine transformation mapping $T$ onto $Q_0$, and define $$\psi_T(x):=\psi_{Q_0}(A_Tx).$$ Then the family $\{\psi_T\}_{T\in\T}$ forms a smooth partition of unity sub-ordinate to the cover $\T$, satisfying
        \begin{itemize}
          \item $\psi_{T}(x)=0$ for $x\notin 2T$;
            \item $\sum_{T\in\T}\psi_{T}(x)=1$ for all $x\in\R^2$;
           \item $0\leq \psi_{T}(x)\leq 1$ for all $x\in\R^2$.
        \end{itemize}
    \end{definition}
    We use the partition of unity described above to perform wave-packet decompositions of the function $f$. This is described below.
\subsection{A pruning process for wave-packets}\label{sec: pruning} In this part of the argument, we modify the original function $f$ by considering its wave-packet decomposition, and pruning out the \textit{bad wave-packets}, where we do not have good control over $\|f_{\tau_k}\|_{L^\infty}$ (recall the definition of $f_{\tau_k}$ from \S \ref{sec: pigeonholing}).

Let us first describe the wave-packet decomposition. For each $\tau_k$, let $\T_{\tau_k}$ denote a collection of essentially disjoint rectangles congruent to $\tau_k^*$ that tiles $\R^2$. We form the partition of unity $\{\psi_T\}_{T\in\T_{\tau_k}}$ sub-ordinate to $\T_{\tau_k}$ as described in Definition \ref{def: zeta}. In light of \eqref{eqn: f representation}, we have the following wave-packet decomposition $$f=\sum_{\tau_k\in\calt_k}\sum_{T\in\T_{\tau_k}}\psi_Tf_{\tau_k}.$$ 
Since our estimates are localised to the $R\times R$ square $B_R$, we 
may pass to the subcollection of \textit{close wave-packets} $$\T^{\mathrm{cl}}_{\tau_k}:=\{T\in\T_{\tau_k}:T\cap 10B_R\neq\emptyset\},$$ so that $$f(x)=\sum_{\tau_k\in\calt_k}\sum_{T\in\T^{\mathrm{cl}}_{\tau_k}}\psi_T(x)f_{\tau_k}(x)\quad\text{for all}\quad x\in B_R.$$ 

We categorize the wave-packets as \textit{good} or \textit{bad}, according to their amplitudes. The quantity $G$ introduced in Notation \ref{not: G} will set the threshold. We do this process inductively, starting at the finest level $k=N$, and progressing down to $k=1$. For each $\tau_N\in\calt_N$ define the subcollection of good wave-packets $$\T_{\tau_N}^{\mathrm{g}}:=\{T\in\T^{\mathrm{cl}}_{\tau_N}:\|\psi_Tf_{\tau_N}\|_{L^\infty}\leq G\}.$$ After removing the bad wave-packets $\T^{\mathrm{b}}_{\tau_N}:=\T^{\mathrm{cl}}_{\tau_N}\setminus\T^{\mathrm{g}}_{\tau_N}$, we define the \textit{pruned} function $$f_{N,\tau_N}:=\sum_{T\in\T_{\tau_N}^{\mathrm{g}}}\psi_Tf_{\tau_N}.$$ For each $\tau_{N-1}\in\calt_{N-1}$, we also define the pruned function $$f_{N,\tau_{N-1}}:=\sum_{\tau_N\in\calt_N(\tau_{N-1})}f_{N,\tau_N},$$ as well as the pruned function $$f_N:=\sum_{\tau_N\in\calt_N}f_{N,\tau_N}.$$
At level $k$, we identify the good wave-packets $\T^{\mathrm{g}}_{\tau_k}\subseteq\T^{\mathrm{cl}}_{\tau_k}$, and subsequently define the pruned function $f_{k,\tau_k}:=\sum_{T\in\T^{\mathrm{g}}_{\tau_k}}\psi_Tf_{k+1,\tau_k}$ for each $\tau_k\in\calt_k$. We also define the pruned function $f_{k,\tau_{k-1}}:=\sum_{\tau_k\in\calt_k(\tau_{k-1})}f_{k,\tau_k}$ for each $\tau_{k-1}\in\calt_{k-1}$, and $f_k:=\sum_{\tau_k\in\calt_k}f_{k,\tau_k}$. At level $k-1$, we define the good wave-packets $$\T_{\tau_{k-1}}^{\mathrm{g}}:=\{T\in\T^{\mathrm{cl}}_{\tau_{k-1}}:\|\psi_Tf_{k,\tau_{k-1}}\|_{L^\infty}\leq G\}\quad\text{for each}\quad \tau_{k-1}\in\calt_{k-1}.$$ By pruning out the bad wave-packets $\T^{\mathrm{b}}_{\tau_{k-1}}:=\T^{\mathrm{cl}}_{\tau_{k-1}}\setminus\T^{\mathrm{g}}_{\tau_{k-1}},$ we define the functions $$f_{k-1,\tau_{k-1}}:=\sum_{T\in\T_{\tau_{k-1}}^{\mathrm{g}}}\psi_Tf_{k,\tau_{k-1}}\quad\text{for each}\quad \tau_{k-1}\in\calt_{k-1}.$$ We also define the functions $$f_{k-1,\tau_{k-2}}:=\sum_{\tau_{k-1}\in\calt_{k-1}(\tau_{k-2})}f_{k-1,\tau_{k-1}},\quad\text{for all}\quad\tau_{k-2}\in\calt_{k-2},$$ as well as the function $$f_{k-1}:=\sum_{\tau_{k-1}\in\calt_{k-1}}f_{k-1,\tau_{k-1}}.$$ 
The process terminates at the level $k=1$, producing the pruned functions $f_{1,\tau_1}$ for each $\tau_1\in\calt_1$, and the pruned function $f_{1,\tau_0}=f_1$.
\begin{remark}
    The pruned functions $f_{k,\tau_k}$ and $f_{k+1,\tau_k}$ are both modified versions of the function $f_{\tau_k}$, and the first subscripts indicate the respective level at which they have been defined. Similarly, the function $f_k$ is a modified version of $f$, and the subscript $k$ indicates that the function has been produced at the level $k$.
\end{remark}
\begin{lemma} [Properties of $f_k$]\label{lemma: properties of f_k} For each $\tau_k\in\calt_k$, the pruned functions satisfy the following properties:
    \begin{enumerate}[label=(\roman*)]
        \item $|f_{k,\tau_k}|\leq |f_{k+1,\tau_k}|$;
        \item $\|f_{k,\tau_k}\|_{L^\infty}\leq 9G$;
        \item $f_{k+1,\tau_k}$ and $f_{k,\tau_k}$ are essentially Fourier supported in $\tau_k$ in the following sense $$\|f_{k+1,\tau_{k}}-f_{k+1,\tau_{k}}*\rho_{\tau_{k}}\|_{L^\infty(\R^2)}\lesssim R^{-100},\quad\text{and}\quad \|f_{k,\tau_{k}}-f_{k,\tau_{k}}*\rho_{\tau_{k}}\|_{L^\infty(\R^2)}\lesssim R^{-100};$$
        \item $f_{k,\tau_k}$ is supported in $10B_R$.
    \end{enumerate}
\end{lemma}
\begin{proof}
    (i) This follows from the choice of the partition of unity $\{\psi_T\}_{T\in\T_{\tau_k}}$. 
    
    \noindent (ii) Again by the choice of the partition of unity, for each $x\in\R^2$, we have $$f_{k,\tau_k}(x)=\sum_{\substack{T\in\T_{\tau_k}^{\mathrm{g}}\\ x\in 2T}}\psi_T(x)f_{k+1,\tau_k}(x).$$ 
    Now $\#\{T\in\T_{\tau_k}:x\in 2T\}\leq 9$, and by definition of the collection $\T_{\tau_k}^{\mathrm{g}}$, we have  $$|\psi_Tf_{k+1,\tau_k}(x)|\leq G\quad\text{for all}\quad T\in\T_{\tau_k}^{\mathrm{g}}.$$ Thus, $$|f_{k,\tau_k}(x)|\leq\sum_{\substack{T\in\T_{\tau_k}^{\mathrm{g}}\\ x\in 2T}}|\psi_T(x)f_{k+1,\tau_k}(x)|\leq 9G.$$ 

    \noindent (iii) From Lemma \ref{lemma: rho}, we observe that $f_{k+1,\tau_{k}}*\rho_{\tau_{k}}$ is compactly Fourier supported around $\tau_k$, and its Fourier transform is equal to $\widehat{f}_{k+1,\tau_k}$ on $\tau_k$. Each unmodified $f_{\tau_k}$ is compactly Fourier supported on $\tau_k$, and the modified versions $f_{k+1,\tau_k}$ are obtained by successively multiplying the original version by spatial cut-offs $\psi_T$, introducing uncertainty on the Fourier side. However, the rectangles involved in the construction of $f_{k+1,\tau_k}$ all have scale larger or equal to that of $\tau_k^*$. Hence, the modified functions $\widehat{f}_{k+1,\tau_k}$ are essentially supported on $\tau_k$, since the scale of uncertainty is smaller than or equal to that of $\tau_k$. See the appendix for a rigorous proof of (iii).

    \noindent (iv) This follows from the fact that the good wave-packets are chosen to be in $\T^{\mathrm{cl}}_{\tau_k}$, which are the rectnagles that intersect $10B_R$.
\end{proof}
\begin{definition} [Definition of $g_k$]
\label{def: g_k}
    For $1\leq k\leq N$, we define the square functions $g_k$ as $$g_k:=\sum_{\tau_k\in\calt_k}|f_{k+1,\tau_k}|^2*\omega_{\tau_k},$$
    where $f_{N+1,\tau_N}:=f_{\tau_N}$.
\end{definition}
For $k=N$, this yields an averaged version of the original square function $\sum_\theta|f_\theta|^2$. Convolving with $\omega_{\tau_k}$ aids us in the application of certain uncertainty heuristics. Note that by Young's convolution inequality and Lemma \ref{lemma: L^1 norm of omega}, we have $$\||f_{k+1,\tau_k}|^2*\omega_{\tau_k}\|_{L^p(\R^2)}\lesssim_\epsilon R^{2\delta}\|f_{k+1,\tau_k}^2\|_{L^p(\R^2)},$$ which can be used to remove the convolutions from our final decoupling estimate.
\begin{lemma}
    [Controlling the bad wave-packets]\label{lemma: bad wave packets} For each $\tau_1\in\calt_1$, we have $$|\sum_{\tau_k\in\calt_k(\tau_1)}\sum_{T\in\T_{\tau_k}^{\mathrm{b}}}\psi_Tf_{k+1,\tau_k}|\leq G^{-1}g_k+R^{-100}.$$
\end{lemma}
\begin{proof}
    By relaxing the outer summation, for each $x\in\R^2$ we have 
    \begin{equation}
        \label{eqn: bad wave-packets}
|\sum_{\tau_k\in\calt_k(\tau_1)}\sum_{T\in\T_{\tau_k}^{\mathrm{b}}}\psi_T(x)f_{k+1,\tau_k}(x)|\leq G^{-1}\sum_{\tau_k\in\calt_k}\sum_{\substack{T\in\T_{\tau_k}^{\mathrm{b}}\\x\in 2T}}\|\psi_Tf_{k+1,\tau_k}\|^2_{L^\infty},
    \end{equation}
     since for each $T\in\T_{\tau_k}^{\mathrm{b}}$ we have $\|\psi_Tf_{k+1,\tau_k}\|_{L^\infty}\geq G.$ Now fix $T\in\T^{\mathrm{b}}_{\tau_k}$ satisfying $x\in 2T$, and let $\|\psi_Tf_{k+1,\tau_k}\|^2_{L^\infty}=|\psi_T(y)f_{k+1,\tau_k}(y)|$ for some $y\in 2T$. By Lemma \ref{lemma: properties of f_k} (iii), we know that
    \begin{equation}
        \label{eqn: local constancy use 1}
        |f_{k+1,\tau_k}(y)-f_{k+1,\tau_k}*\rho_{\tau_k}(y)|\lesssim R^{-100}.
    \end{equation}
    By Cauchy--Schwarz and Lemma \ref{lemma: rho}, we find $$|f_{k+1,\tau_k}*\rho_{\tau_k}(y)|^2\lesssim|f_{k+1,\tau_k}|^2*|\rho_{\tau_k}|(y).$$
    Next, observe that $$|f_{k+1,\tau_k}|^2*\omega_{\tau_k}(x)=\int|f_{k+1,\tau_k}(u)|^2*\omega_{\tau_k}(x-u)\dd{u}\geq \int|f_{k+1,\tau_k}(u)|^2*w_{\tau_k}(x-u)\dd{u},$$ since $\omega_{\tau_k}\geq w_{\tau_k}$ by Definition \ref{def: w,omega}. From the definition we also have $$w_{\tau_k}(x-u)=\sup_{z\in x+4\tau_k^*}|\rho_{\tau_k}(z-u)|\geq |\rho_{\tau_k}(y-u)|,$$ where the last step follows from the fact that $x,y\in 2T$, so that $y-x\in 2T-2T\subseteq 4\tau_k^*$. Combining the above, we find \begin{equation}
        \label{eqn: rho to omega}
        |f_{k+1,\tau_k}*\rho_{\tau_k}(y)|^2\lesssim|f_{k+1,\tau_k}|^2*\omega_{\tau_k}(x).
    \end{equation}
    Thus, for the $T$ fixed above, we have $$\|\psi_Tf_{k+1,\tau_k}\|^2_{L^\infty}=|\psi_T(y)f_{k+1,\tau_k}(y)|^2\leq|f_{k+1,\tau_k}(y)|^2\lesssim |f_{k+1,\tau_k}(y)-f_{k+1,\tau_k}*\rho_{\tau_k}(y)|^2+|f_{k+1,\tau_k}*\rho_{\tau_k}(y)|^2,$$ where we have used the fact that $\psi_T(y)\leq 1$. Then using \eqref{eqn: local constancy use 1} with \eqref{eqn: rho to omega}, we get $$\|\psi_Tf_{k+1,\tau_k}\|^2_{L^\infty}\lesssim |f_{k+1,\tau_k}|^2*\omega_{\tau_k}(x)+R^{-200}.$$ This holds for all $T\in\T^{\mathrm{b}}_{\tau_k}$ satisfying $x\in 2T$. Using this in \eqref{eqn: bad wave-packets}, and observing that $\#\{T\in\T^{\mathrm{b}}_{\tau_k}:x\in 2T\}\lesssim 1$, we get $$|\sum_{\tau_k\in\calt_k(\tau_1)}\sum_{T\in\T_{\tau_k}^{\mathrm{b}}}\psi_T(x)f_{k+1,\tau_k}(x)|\lesssim G^{-1}\sum_{\tau_k\in\calt_k}|f_{k+1,\tau_k}|^2*\omega_{\tau_k}(x)+G^{-1}(\#\calt_k)R^{-200}.$$
    Since $G\geq R^{-1}$ and $\#\calt_k\leq R^{k\epsilon/2}$, we have $G^{-1}(\#\calt_k)R^{-200}\leq R^{-100}$, and the desired conclusion follows. 
\end{proof}
\subsection{High/low decomposition}
In this section we describe the \textit{high/low decomposition}, which is central to the main argument.
\begin{definition}
    [High and low parts]
    \label{def: hi lo}
    Recall the function $\eta_{k+1}$ introduced in Definition \ref{def: eta}. We define the \textit{low part} of $g_k$ as $$g_k^{\mathrm{lo}}:=g_k*\eta_{k+1},$$ and the \textit{high part} as $$g_k^{\mathrm{hi}}:=g_k-g_k^{\mathrm{lo}}.$$ 
\end{definition}
     Then we have \begin{equation}
    \label{eqn: hi+lo} g_k\leq |g_k^{\mathrm{hi}}|+|g_k^{\mathrm{lo}}|.
\end{equation}
The following lemmas describe the desirable features of the high and the low parts of $g_k$, respectively. We state these lemmas presently, briefly explaining their roles in the proof of Theorem \ref{thm: weak type}, and provide the proofs afterwards. 
\begin{lemma} 
    [High Lemma]\label{lemma: high lemma} There exists an absolute constant $C_{\mathrm{hi\,}}\geq 1$ such that 
    $$\|g_k^{\mathrm{hi}}\|_{L^2(Q_R)}^2\leq C_{\mathrm{hi\,}}  R^{2\epsilon}\sum_{\tau_k\in\calt_k}\|f_{k+1,\tau_k}\|_{L^4}^4+R^{-90},$$ for all $R$-ball $Q_R$.
\end{lemma}
The \textit{High Lemma} yields an $\ell^4L^4$ decoupling estimate by exploiting the geometry of the high-frequency part of $g_k$, which is essentially Fourier supported in a union of tubes having small overlap (see Figure \ref{fig: high/low}).
\begin{figure}
    \centering
    \includegraphics[width=12cm]{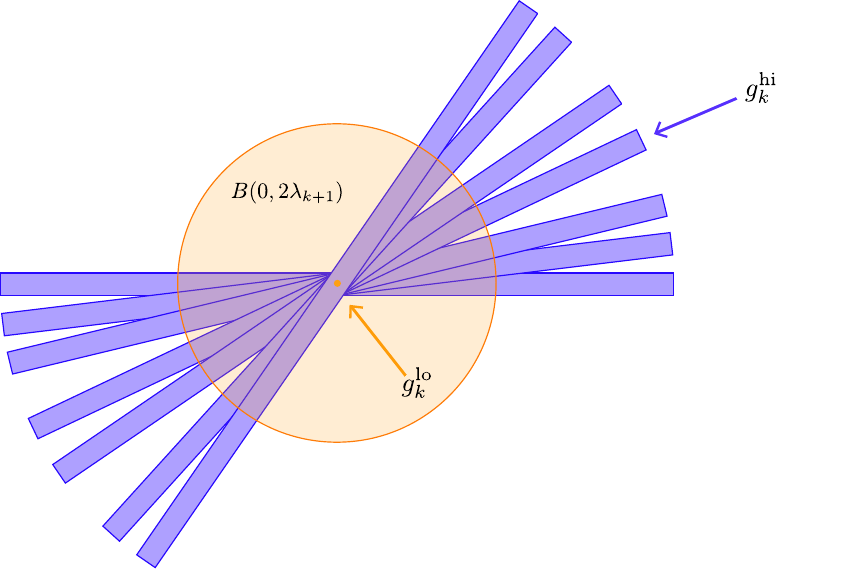}
    \caption{The high/low decomposition of $g_k$. The square function $g_k$ is essentially Fourier supported in a union of tubes that forms a bush centred at the origin. The low part $g_k^{\mathrm{lo}}$ is Fourier supported in $B(0,2\lambda_{k+1})$. The high part $g_k^{\mathrm{hi}}$ is essentially Fourier supported outside this ball, where these tubes have small overlap, as shown in this diagram. As a result, the terms $|f_{k+1,\tau_k}|^2*\omega_k$ are essentially orthogonal in this region.}
    \label{fig: high/low}
\end{figure}
\begin{lemma}
    [Low Lemma]\label{lemma: low lemma} There exists an absolute constant $C_{\mathrm{lo\,}}\geq 1$ such that
    $$|g_k^{\mathrm{lo}}|\leq C_{\mathrm{lo\,}}R^{2\delta} g_{k+1}+R^{-99}.$$
\end{lemma}
The \textit{Low Lemma} uses local $L^2$-orthogonality (Lemma \ref{lemma: local orthogonality}) to pointwise dominate the low-frequency part $g_k^{\mathrm{lo}}$ by the square function $g_{k+1}$.  

The idea is to combine the High and Low Lemmas in an iterative scheme to prove Theorem \ref{thm: weak type}. The square function at each scale is either high-dominated or low-dominated, in the sense of \eqref{eqn: hi+lo}. We apply the High Lemma to the high-dominated case leading to an $\ell^4L^4$ decoupled expression, which will be upgraded to the desired $\ell^2L^6$ decoupling estimate. In the low-dominated case, we apply the Low Lemma, which allows us to pass onto the next scale, at which point we repeat the whole process. The iteration scheme will continue until we hit the terminal scale. If the square functions are low-dominated at every scale, we will be able to prove a pointwise square function estimate, which trivially implies decoupling.

We now provide the proofs of these two lemmas.
\begin{proof}
    [Proof of the High Lemma]
    We start by smoothing out the sharp cut-off on the left-hand integral using Lemma \ref{lemma: phi} as follows $$\|g_k^{\mathrm{hi}}\|_{L^2(Q_R)}\lesssim\|g_k^{\mathrm{hi}}\cdot\phi_{Q_R}\|_{L^2(\R^2)}.$$ Next, we write $$g_k^{\mathrm{hi}}\cdot\phi_{Q_R}=\sum_{\tau_k\in\calt_k}h_{\tau_k},\qquad\text{where}\qquad h_{\tau_k}:=(|f_{k+1,\tau_k}|^2*\omega_{\tau_k})\cdot\phi_{Q_R}-(|f_{k+1,\tau_k}|^2*\omega_{\tau_k}*\eta_{k+1})\cdot\phi_{Q_R}.$$ 
    For $\mathcal{E}:=|f_{k+1,\tau_k}|^2-|f_{k+1,\tau_k}*\rho_{\tau_k}|^2$, we have $h_{\tau_k}=\tilde h_{\tau_k}+(\mathcal{E}*\omega_{\tau_k}-\mathcal{E}*\omega_{\tau_k}*\eta_{k+1})\cdot\phi_{Q_R}$, where $$\tilde h_{\tau_k}:=(|f_{k+1,\tau_k}*\rho_{\tau_k}|^2*\omega_{\tau_k})\cdot\phi_{Q_R}-(|f_{k+1,\tau_k}*\rho_{\tau_k}|^2*\omega_{\tau_k}*\eta_{k+1})\cdot\phi_{Q_R}.$$
    \medskip\noindent\underline{Fourier support:}  By Lemma \ref{lemma: rho}, we know that the function $$(|f_{k+1,\tau_k}*\rho_{\tau_k}|^2)\;\widehat{}\;=\big(\widehat{f}_{k+1,\tau_k}\cdot\widehat{\rho}_{\tau_k}\big)*\big(\widehat{f}_{k+1,\tau_k}\cdot\widehat{\rho}_{\tau_k}(\,-\,\cdot\,)\big)$$ is supported on $R^{2\delta}(\tau_k-\tau_k)$ for $R\gg_\epsilon 1$; and $\widehat\phi_{Q_R}$ is supported on $B(0,R^{-1})$. Hence, $\tilde h_{\tau_k}$ is Fourier supported in the box $R^{2\delta}(\tau_k-\tau_k)+B(0,R^{-1})$ centred at the origin. On the other hand by Lemma \ref{lemma: eta}, we know that $\widehat{\eta}_{k+1}(\xi)=1$ for all $\xi\in B(0,\lambda_{k+1})$. From here it follows that $$(\tilde h_{\tau_k})\;\widehat{}\;(\xi)=0\quad\text{for all}\quad\xi\in B(0,\lambda_{k+1}).$$ 
     By \eqref{eqn: lambda k relative range} we have $\lambda_{k+1}\geq (1/2)R^{-\epsilon}\lambda_k$.
    As such, we get $B(0,\lambda_{k+1})\supseteq B(0,\lambda_kR^{-\epsilon}/2)$. Combining the above we find that $\tilde h_{\tau_k}$ is Fourier supported in the set $\big[R^{2\delta}(\tau_k-\tau_k)+B(0,R^{-1})\big]\,\setminus\, B(0,\lambda_kR^{-\epsilon}/2)$. 
    
    \medskip\noindent\underline{Orthogonality:} 
    The set $R^{2\delta}(\tau_k-\tau_k)+B(0,R^{-1})$ is contained in a $16R^{2\delta}\lambda_k\times 16R^{{2\delta}-k\epsilon}$ tube $\bar\tau_k$, centred at the origin. By \eqref{eqn: tau k are eccentric}, we have $$\len(\tau_k)\geq R^{2\epsilon}\wid(\tau_k),\quad\text{so that}\quad \lambda_k\geq (1/2)R^{-(k-2)\epsilon},$$ where the last step follows from the pigeonholing. It follows that $$\lambda_k\geq 32R^{2\delta} R^{-(k-1)\epsilon},\quad\text{whenever}\quad R\geq 2^{\frac{5}{\epsilon-{2\delta}}}.$$ Thus for all $R\gg_\epsilon 1$, we get $\lambda_kR^{-\epsilon}/2\geq 16R^{{2\delta}-k\epsilon},$ which is the width of the tube $\bar\tau_k$. Hence, the set $\bar\tau_k^\circ:=\bar\tau_k\,\setminus\, B(0,\lambda_kR^{-\epsilon}/2)$ is roughly a $16R^{2\delta}\lambda_k\times 16R^{{2\delta}-k\epsilon}$ tube centred at the origin, with $(1/32)R^{-\epsilon-{2\delta}}$ proportion of its length removed from the middle. 
    It is clear that the long side of $\bar\tau_k^\circ$ is parallel to $\tau_k$. By an application of Lemma \ref{lemma: direction separation tau_k}, and simple trigonometry, it follows (see Figure \ref{fig: high/low}) that $$\#\{\tau_{k}^\circ\ni\xi\}\lesssim R^{\epsilon+{2\delta}}\quad\text{for all}\quad\xi\in\R^2.$$

    \medskip\noindent\underline{Error terms:} By triangle inequality and Lemma \ref{lemma: number of boxes}, we have $$\|\sum_{\tau_k\in\calt_k}(\mathcal{E}*\omega_{\tau_k})\cdot\phi_{Q_R}\|_{L^2(\R^2)}\lesssim R^{1/2}\max_{\tau_k}\|(\mathcal{E}*\omega_{\tau_k})\cdot\phi_{Q_R}\|_{L^2(\R^2)}.$$
    For each $\tau_k$ we have $$\|(\mathcal{E}*\omega_{\tau_k})\cdot\phi_{Q_R}\|_{L^2(\R^2)}\leq\|\mathcal{E}*\omega_{\tau_k}\|_{L^\infty(\R^2)}\cdot\|\phi_{Q_R}\|_{L^2(\R^2)}\lesssim\|\mathcal{E}\|_{L^\infty(\R^2)}\|\phi_{Q_R}\|_{L^2(\R^2)}.$$ Now $$|\mathcal{E}(x)|\lesssim \#\calt_N(\tau_k)\max_{\theta\in\calt_N(\tau_k)}\|f_\theta\|_{L^\infty{(\R^2)}}\cdot|f_{k+1,\tau_k}(x)-f_{k+1,\tau_k}*\rho_{\tau_k}(x)|\quad\text{for all}\quad x\in\R^2.$$ By \eqref{eqn: local constancy use 1}, it follows that $$\|\mathcal{E}\|_{L^\infty(\R^2)}\lesssim R^{-100+1/2}\max_\theta\|f_\theta\|_{L^\infty(\R^2)}.$$ Due to the normalisation \eqref{eqn: normalisation in sup norm}, we find that the error is of the order $O(R^{-99})$.
    
    \medskip\noindent In light of the Fourier support, orthogonality relations, and the bound on the error obtained above, we find that $$\|\sum_{\tau_{k}\in\calt_k} h_{\tau_{k}}\|_{L^2}^2\lesssim R^{\epsilon+2\delta}\sum_{\tau_{k}\in\calt_k}\|\tilde h_{\tau_{k}}\|_{L^2}^2+R^{-198}.$$
    By the properties of the functions $
\omega_{\tau_k},\eta_{\lambda_k+1}, \phi_{Q_R}$ and by Young's convolution inequality, we have $$\|\tilde h_{\tau_{k}}\|^2_{L^2}\lesssim  R^{8\delta}\|f_{k+1,\tau_k}\|^4_{L^4},$$  from which the desired conclusion follows. 
\end{proof}
\begin{proof}
    [Proof of the Low Lemma]
       By Lemma \ref{lemma: properties of f_k} (iii) and Lemma \ref{lemma: number of boxes}, we have $$f_{k+1,\tau_k}=\sum_{\tau_{k+1}\in\calt_{k+1}(\tau_k)}f_{k+1,\tau_{k+1}}*\rho_{\tau_{k+1}}+O(R^{-100+\epsilon}).$$ By Lemma \ref{lemma: eta} and Lemma \ref{lemma: L^1 norm of omega}, we then find
    $$|g_k^{\mathrm{lo}}(x)|\lesssim\big|\sum_{\tau_k\in\calt_k}|\sum_{\tau_{k+1}\in\calt_{k+1}(\tau_k)}f_{k+1,\tau_{k+1}}*\rho_{\tau_{k+1}}|^2*\tilde\eta_{k+1}(x)\big|+O(R^{-99}),$$ where $\tilde\eta_{k+1}:=\omega_{\tau_k}*\eta_{k+1}$. Now $$\bigg||\sum_{\tau_{k+1}\in\calt_{k+1}(\tau_k)}f_{k+1,\tau_k+1}*\rho_{\tau_{k+1}}|^2*\tilde\eta_{k+1}(x)\bigg|=\bigg|\int_{\R^2}|\sum_{\tau_{k+1}\in\calt_{k+1}(\tau_k)}f_{k+1,\tau_k+1}*\rho_{\tau_k}(y)|^2\tilde\eta_{k+1}(x-y)\dd{y}\bigg|,$$ where the function $\tilde\eta_{k+1}(x-\cdot)$ is Fourier supported in $B(0,2\lambda_{k+1})$, and each $f_{k+1,\tau_k+1}*\rho_{\tau_{k+1}}$ is Fourier supported in $2R^{2\delta}\cdot\tau_{k+1}$ for $R\gg_\epsilon 1$. By the pigeonholing in \S \ref{sec: pigeonholing}, each $\tau_{k+1}$ is approximately a box of size $\lambda_{k+1}\times R_{k+1}^{-1}$. Geometrically, it is evident that the sets $2R^{2\delta}\cdot\tau_{k+1}+B(0,2\lambda_{k+1})$ are only $O(R^{2\delta})$-overlapping. Thus, by local $L^2$-orthogonality (Lemma \ref{lemma: local orthogonality}),
    $$\bigg|\int_{\R^2}|\sum_{\tau_{k+1}\in\calt_{k+1}(\tau_k)}f_{k+1,\tau_k+1}*\rho_{\tau_k}(y)|^2\tilde\eta_{k+1}(x-y)\dd{y}\bigg|\lesssim R^{2\delta}\sum_{\tau_{k+1}\in\calt_{k+1}(\tau_k)}\int_{\R^2}|f_{k+1,\tau_{k+1}}*\rho_{\tau_k}(y)|^2|\tilde\eta_{k+1}(x-y)|\dd{y}.$$
    By Cauchy--Schwarz, triangle inequality, and the definition of $\tilde\eta_{k+1}$, we have $$\int_{\R^2}|f_{k+1,\tau_k+1}*\rho_{\tau_k}(y)|^2|\tilde\eta_{k+1}(x-y)|\dd{y}\lesssim|f_{k+1,\tau_{k+1}}|^2*\omega_{\tau_k}*|\rho_{\tau_k}|*|\eta_{k+1}|(x).$$ By Definition \ref{def: w,omega}, we have $$\omega_{\tau_{k+1}}\geq \omega_{\tau_k}*|\rho_{\tau_k}|*|\eta_{k+1}|,$$ and as such $$\int_{\R^2}|f_{k+1,\tau_k+1}*\rho_{\tau_k}(y)|^2|\tilde\eta_{k+1}(x-y)|\dd{y}\lesssim |f_{k+1,\tau_{k+1}}|^2*\omega_{\tau_{k+1}}(x).$$ Combining everything obtained so far, we have $$|g_k^{\mathrm{lo}}(x)|\lesssim R^{2\delta}\sum_{\tau_k\in\calt_k}\sum_{\tau_{k+1}\in\calt_{k+1}(\tau_k)}|f_{k+1,\tau_{k+1}}|^2*\omega_{\tau_{k+1}}(x)+O(R^{-99}).$$
    By Lemma \ref{lemma: properties of f_k} (i), we have $|f_{k+1,\tau_{k+1}}|\leq |f_{k+2,\tau_{k+1}}|$. In light of \eqref{eqn: calt k+1 representation}, the desired conclusion follows upon choosing the constant $C_{\mathrm{lo}}$ sufficiently large.
\end{proof}
\begin{remark}
    In the work of Guth--Maldague--Wang \cite{gmw}, the constant appearing on the  right-hand side of the Low Lemma is $1$. That is, their version of the Low Lemma is the following: $$|g_k^{\mathrm{lo}}|\leq g_{k+1}+R^{-99}.$$
    They achieve this by reducing the problem to a `well-spaced' case, where the Low Lemma holds with the improved constant. This is critical in obtaining the decoupling estimate for the parabola with $(\log R)^c$-loss. For the purposes of this paper, we shall work with the weaker version of the lemma.
\end{remark}
\subsection{A partition of the spatial domain}
In order to prove the weak-type estimate \eqref{eqn: main broad}, we partition the spatial domain $B_R$ into certain subsets $\Omega_k$ similar to \cite[\S3.5]{gmw}. On the set $\Omega_k$, the square function $g_k$ will be dominated by its high part $g_k^{\mathrm{hi}}$ (Lemma \ref{lemma: high dom}), and the bilinearized versions of $f$ and $f_{k}$ will be essentially same (Lemma \ref{lemma: replacing f}).
\begin{definition}
    [Definition of $\Omega_k$]
    \label{def: Omega sets}
    Let $\Q_{N}$ be a covering of $B_R$ by finitely-overlapping $\lambda_{N}^{-1}\times\lambda_{N}^{-1}$ squares $Q_{N}$, each having non-empty intersections with $B_R$. Define $$\Omega_N:=\bigsqcup_{Q_N\in\Q_N}Q_N.$$
     The side-length of the squares $Q_{N}$ have been chosen to respect the scale of $g_{N}$. Next, we obtain a covering $\widetilde\Q_{{N-1}}$ of $B_R$, by partitioning each $Q_{N}\in\Q_{N}$ into finitely-overlapping $\lambda_{N-1}^{-1}\times\lambda_{N-1}^{-1}$ squares $Q_{{N-1}}$, noting that $\lambda_{N-1}^{-1}\leq \lambda_{N}^{-1}$. Recall the parameter $\beta$ from Notation \ref{not: G}. Define the subcollection of squares $$\Q_{{N-1}}:=\big\{Q_{{N-1}}\in\widetilde\Q_{{N-1}}: \|g_{n_{N-1}}\|_{L^\infty(R^\delta\cdot Q_{{N-1}})}\geq (2R^{2\delta} C_{\mathrm{lo}\,})^2 \beta\big\},\quad\text{and define}\quad \Omega_{N-1}:=\bigsqcup_{Q_{N-1}\in\Q_{N-1}}Q_{N-1}.$$
    Having defined the collection of squares $\Q_N,\widetilde\Q_N,\Q_{N-1},\widetilde\Q_{N-1},\dots,\Q_{k+1},\widetilde\Q_{k+1}$, and the partitioning sets $\Omega_N,\Omega_{N-1},\dots,\Omega_{k+1}$, we define $\Omega_k$ as follows. We obtain a covering $\widetilde\Q_{k}$ of the set $B_R\setminus(\Omega_{N-1}\sqcup\dots\sqcup\Omega_{k+1})$ by partitioning each $Q_{k+1}\in\widetilde\Q_{k+1}\setminus\Q_{k+1}$ into finitely-overlapping $\lambda_{k}^{-1}\times\lambda_{k}^{-1}$ squares $Q_{k}$, and define $$\Q_{k}:=\big\{Q_{k}\in\widetilde\Q_{k}: \|g_k\|_{L^\infty(R^\delta\cdot Q_{k})}\geq (2R^{2\delta} C_{\mathrm{lo}\,})^{N-k+1}\beta\big\},\quad\text{and}\quad \Omega_k:=\bigsqcup_{Q_{k}\in\Q_{k}}Q_{k}.$$ This process terminates at $k=1$. Finally, we define the remainder set 
    \begin{equation}
        \label{eqn: omega 0 def}
        \Omega_0:=B_R\setminus(\Omega_{N-1}\sqcup\dots\sqcup\Omega_1).
    \end{equation}
\end{definition}
    The following containment is an immediate consequence of the definition of the sets $\Omega_k$:
    \begin{equation}
        \label{eqn: omega kappa and omega mu}
        \Omega_k\subseteq\bigsqcup_{Q_j\in\widetilde\Q_j\setminus\Q_j}Q_j\quad\text{for all}\quad k+1\leq j\leq N-1.
    \end{equation}
 Indeed, for all $k\geq 1$ we have $$B_R\setminus(\Omega_{N-1}\sqcup\dots\sqcup\Omega_{k+1})=\bigsqcup_{Q_k\in\widetilde\Q_{k}}Q_k=\big(\bigsqcup_{Q_k\in\widetilde\Q_k\setminus\Q_k}Q_k\big)\sqcup\Omega_k,$$ which implies 
 \begin{equation}
     \label{eqn: discarded cubes identity}
     B_R\setminus(\Omega_{N-1}\sqcup\dots\sqcup\Omega_k)=\bigsqcup_{Q_k\in\widetilde\Q_k\setminus\Q_k}Q_k,
 \end{equation}
 and 
 \begin{equation}
     \label{eqn: omega kappa containment}
     \Omega_k\subseteq B_R\setminus(\Omega_{N-1}\sqcup\dots\sqcup\Omega_j),\quad\text{for all}\quad k+1\leq j\leq N-1.
 \end{equation}
 But \eqref{eqn: discarded cubes identity} holds for all $k\geq 1$, and so for $k=j\geq 1$ we get $$B_R\setminus(\Omega_{N-1}\sqcup\dots\sqcup\Omega_{j})=\bigsqcup_{Q_j\in\widetilde\Q_j\setminus\Q_j}Q_j.$$ 
 On the other hand, when $k=0$, \eqref{eqn: omega kappa containment} holds with equality by \eqref{eqn: omega 0 def}.
 Combining the above with \eqref{eqn: omega kappa containment} (for $k\geq 1$) and \eqref{eqn: omega 0 def} (for $k=0$), yields \eqref{eqn: omega kappa and omega mu}.
\begin{lemma}
\label{lemma: replacing f}
   Let $0\leq k\leq N$. For each $\tau_1\in\calt_1$, we have $$|f_{\tau_1}(x)-f_{k+1,\tau_1}(x)|\leq (10K^2R^{4\epsilon})^{-1}\alpha\quad\text{for all}\quad x\in \Omega_k\cap U^{\mathrm{br}}_{\alpha,\beta},$$ where we adopt the notation $f_{j,\tau_1}:=\sum_{\tau_{j}\in\calt_{j}(\tau_1)}f_{j,\tau_j}.$
\end{lemma}
\begin{proof}
Following the notation from Definition \ref{def: g_k}, we denote $f_{N+1,\tau_1}=f_{\tau_1}$, and $f_{N+1,\tau_N}=f_{\tau_N}$. We start by writing the difference $f_{\tau_1}(x)-f_{k+1,\tau_1}(x)$ as a telescoping sum $$f_{\tau_1}(x)-f_{k+1,\tau_1}(x)=\sum_{j=k+1}^N[f_{j+1,\tau_1}(x)-f_{j,\tau_1}(x)].$$ 
By the notation introduced in the statement and Lemma \ref{lemma: pigeonholed boxes representation}, we have $$f_{j+1,\tau_1}=\sum_{\tau_{j+1}\in\calt_{j+1}(\tau_1)}f_{j+1,\tau_{j+1}}=\sum_{\tau_j\in\calt_j(\tau_1)}\sum_{\tau_{j+1}\in\calt_{j+1}(\tau_j)}f_{j+1,\tau_{j+1}}.$$ But the inner sum on the right-hand side above is defined precisely to be the pruned function $f_{j+1,\tau_j}$ in \S\ref{sec: pruning}, and so
$$f_{j+1,\tau_1}=\sum_{\tau_j\in\calt_j(\tau_1)}f_{j+1,\tau_j}.$$ On the other hand, using the same notation and the definition of the pruned functions, we have $$f_{j,\tau_1}=\sum_{\tau_j\in\calt_j(\tau_1)}f_{j,\tau_j}=\sum_{\tau_j\in\calt_j(\tau_1)}\sum_{T\in\T_{\tau_j}^{\mathrm{g}}}\psi_Tf_{j+1,\tau_j}.$$ Combining the last two identities with Lemma \ref{lemma: bad wave packets}, we have $$|f_{j+1,\tau_1}(x)-f_{j,\tau_1}(x)|=|\sum_{\tau_j\in\calt_j(\tau_1)}\sum_{T\in\T_{\tau_j}^{\mathrm{b}}}\psi_Tf_{j+1,\tau_j}(x)|\leq G^{-1}g_j(x)+R^{-100}.$$ Summing over $j$ we get 
\begin{equation}
    \label{eqn: telescope difference bound}
    |f_{\tau_1}(x)-f_{k+1,\tau_1}(x)|\leq G^{-1}\sum_{j=k+1}^Ng_j(x)+(N-k)R^{-100}.
\end{equation} 
Since $x\in\Omega_k$, by \eqref{eqn: omega kappa and omega mu} we get $x\in Q_j$ for some $Q_j\in\widetilde Q_j\setminus\Q_j$ for all $k+1\leq j\leq N-1$. By definition of the family $\Q_j$, we then find 
\begin{equation}
    \label{eqn: g_n_mu control}
    g_{j}(x)\leq\|g_{j}\|_{L^\infty(R^\delta\cdot Q_j)}< (2R^{2\delta} C_{\mathrm{lo}})^{N-j+1}\beta.
\end{equation}
On the other hand when $j=N$, 
\begin{equation}
    \label{eqn: g_n_upnu control}
   g_j(x)=g_N(x)\leq 2\beta,\quad\text{for all}\quad x\in U^{\mathrm{br}}_{\alpha,\beta},
\end{equation}
by the definition of the set $U^{\mathrm{br}}_{\alpha,\beta}$. 
Combining \eqref{eqn: g_n_mu control} and \eqref{eqn: g_n_upnu control}, we get 
\begin{equation}
    \label{eqn: g_j control}
    g_j(x)\leq (2R^{2\delta} C_{\mathrm{lo}})^{N-k}\beta,\quad\text{for all}\quad x\in\Omega_k\cap U^{\mathrm{br}}_{\alpha,\beta},\quad\text{and for all}\quad k+1\leq j\leq N.
\end{equation}
Using \eqref{eqn: g_j control} in \eqref{eqn: telescope difference bound}, we get $$|f_{\tau_1}(x)-f_{k+1,\tau_1}(x)|\leq G^{-1}(N-k)(2R^{2\delta} C_{\mathrm{lo}})^{N-k}\beta+(N-k)R^{-100}.$$
By the definition of $\delta$ from Notation \ref{not: delta}, we have $R^{2\delta/\epsilon}=R^{2\epsilon^9}\leq R^{\epsilon},$ since $\epsilon<1/2$. Using the definition of $G$ from Notation \ref{not: G}, the right-hand side above can then be bounded as follows:
$$G^{-1}(N-k)(2R^{2\delta} C_{\mathrm{lo}})^{N-k}\beta+(N-k)R^{-100}\leq (\epsilon A_\epsilon R^{4\epsilon})^{-1}(2C_{\mathrm{lo}})^{1/\epsilon}\alpha+\epsilon^{-1}R^{-100}\leq 2(\epsilon A_\epsilon R^{4\epsilon})^{-1}(2C_{\mathrm{lo}})^{1/\epsilon}\alpha,$$ where in the last step we recall from Notation \ref{not: G} that $\alpha\geq R^{-1/2}\gg R^{-100}$.
If we choose the constant $A_\epsilon$ large enough to satisfy $$A_\epsilon\geq 20K^2\epsilon^{-1}(2C_{\mathrm{lo\,}})^{1/\epsilon},$$ then the conclusion of this lemma follows immediately.
\end{proof}
By definition of the set $\U$ defined in \eqref{eqn: U set definition}, we have $$|f(x)|\lesssim R^{2\epsilon}\max_{\substack{\tau_1,\tau_1'\in\calt_1\\\tau_1\nsim\tau_1'}}|f_{\tau_1}(x)|^{1/2}|f_{\tau_1'}(x)|^{1/2}\quad\text{for all}\quad x\in\U.$$ 
In light of Lemma \ref{lemma: replacing f}, when $x\in\Omega_k\cap\U$, we can say the following.
\begin{corollary}
    \label{cor: replacing f}
    Let $k\in\{0,1,\dots,N-1\}$. For each $x\in\Omega_k\cap\U$, we have $$|f(x)|\lesssim R^{2\epsilon}\max_{\substack{\tau_1,\tau_1'\in\calt_1\\\tau_1\nsim\tau_1'}}|f_{k+1,\tau_1}(x)|^{1/2}|f_{k+1,\tau_1'}(x)|^{1/2}.$$
\end{corollary}
\begin{proof}
    Let $x\in \Omega_k\cap U_{\alpha,\beta}^{\text{br}}$. Since $x\in U_{\alpha,\beta}^{\text{br}}$, we have $$|f(x)|\leq 2KR^{2\epsilon}\max_{\substack{\tau_1,\tau_1'\in\calt_1\\\tau_1\nsim\tau_1'}}|f_{\tau_1}(x)|^{1/2}|f_{\tau_1'}(x)|^{1/2}=2KR^{2\epsilon}|f_{\tau_1}(x)|^{1/2}|f_{\tau_1'}(x)|^{1/2},$$ for some non-adjacent pair $(\tau_1,\tau_1')\in\calt_1\times\calt_1$ attaining the maximum.
Since $x\in \Omega_k\cap U_{\alpha,\beta}^{\text{br}}$, by Lemma \ref{lemma: replacing f} we also have $$|f_{\tau_1}(x)||f_{\tau_1'}(x)|\leq |f_{k+1,\tau_1}(x)||f_{k+1,\tau_1'}(x)|+(10K^2R^{4\epsilon})^{-1}\alpha|f_{\tau_1}(x)|+(10K^2R^{4\epsilon})^{-1}\alpha|f_{\tau_1'}(x)|+(10K^2R^{4\epsilon})^{-2}\alpha^2.$$ 
By definition of the set $U_{\alpha,\beta}^{\text{br}}$, we have $$|f_{\tau}(x)|\leq \big(\sum_{\tau_1''\in\calt_1}|f_{\tau_1''}(x)|^6\big)^{1/6}\leq 2KR^{2\epsilon}|f_{\tau_1}(x)|^{1/2}|f_{\tau_1'}(x)|^{1/2}\quad\text{for}\quad \tau=\tau_1,\tau_1',$$ as well as $$\alpha\leq|f(x)|\leq 2KR^{2\epsilon}|f_{\tau_1}(x)|^{1/2}|f_{\tau_1'}(x)|^{1/2}.$$
Combining everything so far, we get $$|f(x)|\lesssim R^{2\epsilon}\max_{\substack{\tau_1,\tau_1'\in\calt_1\\\tau_1\nsim\tau_1'}}|f_{k+1,\tau_1}(x)|^{1/2}|f_{k+1,\tau_1'}(x)|^{1/2}\quad\text{for all}\quad x\in\Omega_k\cap\U,$$ as required.
\end{proof}
\begin{lemma}
    [High-domination on $\Omega_k$]
    \label{lemma: high dom}
    For all $k\in\{1,\dots,N\}$, and each $Q_k\in\Q_k$, we have $$\|g_k\|_{L^\infty(R^\delta\cdot Q_{k})}\leq 2\|g_k^{\mathrm{hi}}\|_{L^\infty(R^\delta\cdot Q_{k})}+R^{-98}.$$
\end{lemma}
\begin{proof}
    Let $Q_k\in\Q_k$. By \eqref{eqn: hi+lo}, we clearly have 
    \begin{equation}
        \label{eqn: hi+lo 2}
        \|g_k\|_{L^\infty(R^\delta\cdot Q_{k})}\leq \|g_k^{\mathrm{lo}}\|_{L^\infty(R^\delta\cdot Q_{k})}+\|g_k^{\mathrm{hi}}\|_{L^\infty(R^\delta\cdot Q_{k})},
    \end{equation}
     and by the Low Lemma, we have 
     \begin{equation}
         \label{eqn: low lemma application 1}
         \|g_k^{\mathrm{lo}}\|_{L^\infty(R^\delta\cdot Q_{k})}\leq R^{2\delta} C_{\mathrm{lo}\,}\|g_{k+1}\|_{L^\infty(R^\delta\cdot Q_{k})}+R^{-99}.
     \end{equation}
    Recall from Definition \ref{def: Omega sets} that $\Q_k\subseteq\widetilde\Q_k$, and $\widetilde\Q_k$ was obtained by partitioning each $Q_{k+1}\in\widetilde\Q_{k+1}\setminus\Q_{k+1}$ into finitely-overlapping $\lambda_{k}^{-1}\times\lambda_{k}^{-1}$-squares $Q_k$. Thus, we have $Q_k\subseteq Q_{k+1}$ for some $Q_{k+1}\in\widetilde\Q_{k+1}\setminus\Q_{k+1}$. 
    By the definition of the collection $\Q_{k+1}$, it then follows that  
    \begin{equation}
        \label{eqn: g_n_k+1 bound}
        \|g_{k+1}\|_{L^\infty(R^\delta\cdot Q_k)}\leq (2R^{2\delta} C_{\mathrm{lo}\,})^{N-k}\beta.
    \end{equation}
    Using \eqref{eqn: g_n_k+1 bound} in \eqref{eqn: low lemma application 1}, we obtain $$\|g_k^{\mathrm{lo}}\|_{L^\infty(R^\delta\cdot Q_{k})}\leq (1/2)(2R^{2\delta} C_{\mathrm{lo}\,})^{N-k+1}\beta+R^{-99}.$$ Since $Q_k\in\Q_{k}$, Definition \ref{def: Omega sets} also tells us that $\|g_k\|_{L^\infty(R^\delta\cdot Q_{k})}\geq (2R^{2\delta} C_{\mathrm{lo}\,})^{N-k+1}\beta$, which when used in the display above yields 
    \begin{equation}
        \label{eqn: g_n_k lo part bound}
        \|g_k^{\mathrm{lo}}\|_{L^\infty(R^\delta\cdot Q_{k})}\leq (1/2)\|g_k\|_{L^\infty(R^\delta\cdot Q_{k})}+R^{-99}.
    \end{equation}
     Combining \eqref{eqn: g_n_k lo part bound} with \eqref{eqn: hi+lo 2} yields the desired result.
\end{proof}
\section{Main estimates}
In the first half of this section, we establish the weak-type estimate of Theorem \ref{thm: weak type}. The setup introduced in \S\ref{sec: pruning, hi/low, omega sets} plays a key role in this proof. In the second half, we leverage the weak-type estimate to derive Theorem \ref{thm: convex decoupling ver 2 pigeonholed}, employing standard techniques such as pigeonholing and broad/narrow analysis.
\subsection{Proof of Theorem \ref{thm: weak type}}
 Recall the definition of the set $\U$ from \eqref{eqn: U set definition}. 
 In view of \eqref{eqn: omega 0 def}, the sets $\{\Omega_k\}_{k=0}^{N-1}$ form a partition of the spatial domain $B_R$. Hence, it is sufficient to prove the following:
 \begin{equation}
    \label{eqn: omega k estimate}
    \alpha^6|\Omega_k\cap U^{\mathrm{br}}_{\alpha,\beta}|\lesssim_\epsilon R^{20\epsilon}\big(\sum_{\theta\in\Theta}\|f_\theta\|_{L^\infty}^2\big)^2\big(\sum_{\theta\in\Theta}\|f_\theta\|_{L^2}^2\big)\quad\text{for each}\quad k\in\{0,1,\dots,N-1\}.
\end{equation}
Since the number of partitioning sets is $N\leq1/\epsilon$, this will imply Theorem \ref{thm: weak type}. We first treat the set $\Omega_0$ separately.
\subsection*{\texorpdfstring{Contribution from $\Omega_0$}{Contribution from Omega 0}}
Here we prove \eqref{eqn: omega k estimate} for the special case $k=0$. For that, we will need the following lemma, which allows the passage from coarser to finer scales in $L^2$. 
  \begin{lemma}\label{lemma: unwinding}
    For all $k\geq 1$ we have $$\sum_{\tau_{k}\in\calt_k}\|f_{k+1,\tau_{k}}\|_{L^2(\R^2)}^2\lesssim_\epsilon R^\epsilon\sum_{\theta\in\Theta}\|f_\theta\|_{L^2(\R^2)}^2+R^{-99}.$$ 
\end{lemma}
If we replace $f_{k+1,\tau_k}$ in the above by $f_{\tau_k}$, then Lemma \ref{lemma: unwinding} is a simple consequence of Plancherel's theorem. However, the pruning process alters the Fourier supports of these functions, thus requiring some additional consideration.
\begin{proof}
By Lemma \ref{lemma: properties of f_k} (iii), we can write 
\begin{equation}
    \label{eqn: unwinding uncertainty use}
    f_{k+1,\tau_{k+1}}=f_{k+1,\tau_{k+1}}*\rho_{\tau_{k+1}}+O(R^{-100}).
\end{equation}
 Each $f_{k+1,\tau_{k+1}}*\rho_{\tau_{k+1}}$ is Fourier supported in the box $2R^{2\delta}\cdot\tau_{k+1}$ for $R\gg_\epsilon 1$, and these boxes are $O(R^{2\delta})$-overlapping. In light of this observation, Lemma \ref{lemma: number of boxes}, and \eqref{eqn: unwinding uncertainty use}, we get
$$\|f_{k+1,\tau_k}\|^2_{L^2(\R^2)}=\big\|\sum_{\tau_{k+1}\in\calt_{k+1}(\tau_k)}f_{k+1,\tau_{k+1}}\big\|^2_{L^2(\R^2)}\lesssim_\epsilon R^{4\delta}\sum_{\tau_{k+1}\in\calt_{k+1}(\tau_k)}\|f_{k+1,\tau_{k+1}}*\rho_{\tau_{k+1}}\|_{L^2(\R^2)}^2+O(R^{-100+\epsilon}).$$ By Young's convolution inequality, Lemma \ref{lemma: rho}, and Lemma \ref{lemma: properties of f_k} (i), we have $$\|f_{k+1,\tau_{k+1}}*\rho_{\tau_{k+1}}\|_{L^2(\R^2)}^2\lesssim\|f_{k+1,\tau_{k+1}}\|_{L^2(\R^2)}^2\leq\|f_{k+2,\tau_{k+1}}\|_{L^2(\R^2)}^2\quad\text{for all}\quad \tau_{k+1}\in\calt_{k+1}.$$
 Combining everything so far, we get 
 \begin{align*}
     \sum_{\tau_k\in\calt_k}\|f_{k+1,\tau_k}\|^2_{L^2(\R^2)}&\lesssim_\epsilon R^{4\delta}\sum_{\tau_k\in\calt_k}\sum_{\tau_{k+1}\in\calt_{k+1}(\tau_k)}\|f_{k+2,\tau_{k+1}}\|_{L^2(\R^2)}^2+O(R^{-100+\epsilon})\\&=R^{4\delta}\sum_{\tau_{k+1}\in\calt_{k+1}}\|f_{k+2,\tau_{k+1}}\|_{L^2(\R^2)}^2+O(R^{-100+\epsilon}),
 \end{align*}
 where in the last step we have used \eqref{eqn: calt k+1 representation}.
    Next, we work with the functions $f_{k+2,\tau_{k+1}}$ and repeat the argument above. After $N-k$ iterations of the argument, we reach the following estimate $$\sum_{\tau_k\in\calt_k}\|f_{k+1,\tau_k}\|^2_{L^2(\R^2)}\lesssim_\epsilon R^{4(N-k)\delta}\sum_{\tau_N\in\calt_N}\|f_{N+1,\tau_N}\|_{L^2(\R^2)}^2+(N-k)O(R^{-100+\epsilon}),$$ where $f_{N+1,\tau_N}$ denotes $f_{\tau_N}$. Since $N-k\leq1/\epsilon$, we have $(N-k)O(R^{-100+\epsilon})\lesssim_\epsilon R^{-99}$, and by the definition of $\delta$ from Notation \ref{not: delta}, we have $R^{4(N-k)\delta}\leq R^{\epsilon}.$ This completes the proof.
\end{proof}
\begin{proof}
    [Proof of \eqref{eqn: omega k estimate} for $k=0$]
By Corollary \ref{cor: replacing f} applied to the case $k=0$, we have 
\begin{equation*}
    |f(x)|\lesssim R^{2\epsilon}\max_{\substack{\tau_1,\tau_1'\in\calt_1\\\tau_1\nsim\tau_1'}}|f_{1,\tau_1}(x)|^{1/2}|f_{1,\tau_1'}(x)|^{1/2}\quad\text{for all}\quad x\in\Omega_0\cap\U.
\end{equation*}
Taking sixth powers and integrating, we get 
\begin{equation}
    \label{eqn: omega_0 bound first step}
    \alpha^6|\Omega_0\cap\U|\lesssim R^{12\epsilon}\int_{\Omega_0} \max_{\substack{\tau_1,\tau_1'\in\calt_1\\\tau_1\nsim\tau_1'}}|f_{1,\tau_1}|^3|f_{1,\tau_1'}|^3\leq R^{12\epsilon}\int_{\Omega_0}\big(\sum_{\tau_1\in\calt_1}|f_{1,\tau_1}|\big)^6,
\end{equation} 
where we have used that fact that $|f(x)|\geq\alpha$ for all $x\in\U$.
By Lemma \ref{lemma: number of boxes}, and an application of Cauchy--Schwarz, we have $$\big(\sum_{\tau_1\in\calt_1}|f_{1,\tau_1}|\big)^6\lesssim R^{3\epsilon/2}\big(\sum_{\tau_1\in\calt_1}|f_{1,\tau_1}|^2\big)^3.$$
By Lemma \ref{lemma: properties of f_k} (iii), the definition of $\omega_{\tau_1}$ from Definition \ref{def: w,omega}, and the Cauchy--Schwarz inequality, we get $$|f_{1,\tau_1}|^2=|f_{1,\tau_1}*\rho_{\tau_1}|^2+O(R^{-100})\lesssim|f_{1,\tau_1}|^2*|\rho_{\tau_1}|+R^{-100}\lesssim|f_{1,\tau_1}|^2*\omega_{\tau_1}+R^{-100},$$ which when used in the display above it gives us 
\begin{equation}
    \label{eqn: sum to square function}
    \big(\sum_{\tau_1\in\calt_1}|f_{1,\tau_1}|\big)^6\lesssim R^{3\epsilon/2}g_1^2\big(\sum_{\tau_1\in\calt_1}|f_{1,\tau_1}|^2\big)+R^{-100}\leq R^{3\epsilon/2}g_1^2\big(\sum_{\tau_1\in\calt_1}|f_{2,\tau_1}|^2\big)+R^{-100},
\end{equation}
where the final step follows from Lemma \ref{lemma: properties of f_k} (i). Using \eqref{eqn: sum to square function} in \eqref{eqn: omega_0 bound first step} gives us 
\begin{equation}
    \label{eqn: omega_0 bound second step}
    \alpha^6|\Omega_0\cap\U|\lesssim R^{14\epsilon} \int_{\Omega_0}g_1^2\big(\sum_{\tau_1\in\calt_1}|f_{2,\tau_1}|^2\big)+R^{-50}.
\end{equation}
From Definition \ref{def: Omega sets}, we recall that $\Omega_0=\bigsqcup_{Q_1\in\widetilde\Q_1\setminus\Q_1}Q_1$, and also that  
\begin{equation}
    \label{eqn: g_1 sup norm bound}
    \|g_{1}\|_{L^\infty(R^\delta\cdot Q_1)}\leq (2R^{2\delta} C_{\mathrm{lo}\,})^{N}\beta\quad\text{for all}\quad Q_1\in\widetilde\Q_1\setminus\Q_1.
\end{equation}
Decomposing $\Omega_0$ into the finitely-overlapping squares $Q_1\in\widetilde\Q_1\setminus\Q_1$ in \eqref{eqn: omega_0 bound second step}, and applying the above gives us 
 \begin{align}\label{eqn: omega_0 bound third step}
    \nonumber \alpha^6|\Omega_0\cap\U|&\lesssim (2R^{2\delta} C_{\mathrm{lo}\,})^{2N}\beta^2 R^{14\epsilon} \sum_{Q_1\in\widetilde\Q_1\setminus\Q_1}\int_{Q_1}\sum_{\tau_1\in\calt_1}|f_{2,\tau_1}|^2+R^{-50}\\&\lesssim (2R^{2\delta} C_{\mathrm{lo}\,})^{2N}\beta^2 R^{14\epsilon} \sum_{\tau_1\in\calt_1}\|f_{2,\tau_1}\|^2_{L^2(\R^2)}+R^{-50}.
 \end{align} 
By Lemma \ref{lemma: unwinding} we have $$\sum_{\tau_{1}\in\calt_1}\|f_{2,\tau_{1}}\|_{L^2(\R^2)}^2\lesssim_\epsilon R^\epsilon\sum_{\theta\in\Theta}\|f_\theta\|_{L^2(\R^2)}^2+R^{-99},$$ and by the definition of the set $\U$ from \eqref{eqn: U set definition}, we have $$\beta\leq \sum_{\theta\in\Theta}\|f_\theta\|_{L^\infty(\U)}^2\leq\sum_{\theta\in\Theta}\|f_\theta\|_{L^\infty(\R^2)}^2.$$  Using the above in \eqref{eqn: omega_0 bound third step}, we get $$\alpha^6|\Omega_0\cap\U|\lesssim _\epsilon R^{16\epsilon} \big(\sum_{\theta\in\Theta}\|f_\theta\|_{L^\infty(\R^2)}^2\big)^2\big(\sum_{\theta\in\Theta}\|f_\theta\|_{L^2(\R^2)}^2\big)+R^{-50},$$ where we also used $R^{4N\delta}\leq R^\epsilon$.
In light of the normalisation in \S\ref{sec: normalising f}, the error term $R^{-50}$ above is negligible.
\end{proof}

\subsection*{\texorpdfstring{Contribution from $\Omega_k$}{Contribution from Omega kappa}}
Now we prove \eqref{eqn: omega k estimate} for all $1\leq k\leq N-1$. In order to do so, we first need the following lemma. Recall the definition of transversality from Definition \ref{def: transversality}.
\begin{lemma}
    [Local bilinear square function]\label{lemma: local bilinear square function} Let $k\in\{1,\dots,N\}$, and let $\tau_1,\tau_1'\in\calt_1$ satisfy $\tau_1\nsim\tau_1'$.
     For each $\lambda_k^{-1}\times\lambda_k^{-1}$-square $Q_k$, we have 
    \begin{align*}
        \big\||f_{k+1,\tau_1}|^{1/2}|f_{k+1,\tau_1'}|^{1/2}\big\|_{L^4(Q_{k})}&\lesssim R^{\epsilon}\big\|\big(\sum_{\tau_k\prec\tau_1}|f_{k+1,\tau_k}|^2*\omega_{\tau_k}\big)^{1/2}\big\|_{L^4(R^\delta\cdot Q_{k})}^{1/2}\big\|\big(\sum_{\tau_k'\prec\tau_1'}|f_{k+1,\tau_k'}|^2*\omega_{\tau_k'}\big)^{1/2}\big\|_{L^4(R^\delta\cdot Q_{k})}^{1/2}\\&+R^{-98}.
    \end{align*}
\end{lemma}
The global version of this estimate was obtained by Seeger--Ziesler \cite{SZ}, who used a variant of the classical biorthogonality argument of C\'ordoba--Fefferman for the parabolic square function estimate in the plane. The local version presented here follows directly from their biorthogonality result \cite[Lemma 2.4]{SZ}.
\begin{proof}
    By Lemma \ref{lemma: properties of f_k} (iii) we have $$\int_{Q_{k}}|f_{k+1,\tau_1}|^2|f_{k+1,\tau_1'}|^2\lesssim\int_{Q_{k}}|\sum_{\tau_k\prec\tau_1,\tau_k'\prec\tau_1'}(f_{k+1,\tau_k}*\rho_{\tau_k})\cdot (f_{k+1,\tau_k'}*\rho_{\tau_k'})|^2+O(R^{-98}).$$ By Lemma \ref{lemma: phi}, we may bound the right-hand integral above (up to a constant multiple) by the following weighted version $$\int_{\R^2}|\sum_{\tau_k\prec\tau_1,\tau_k'\prec\tau_1'}(f_{k+1,\tau_k}*\rho_{\tau_k})\cdot (f_{k+1,\tau_k'}*\rho_{\tau_k'})|^2\phi_{Q_{k}}.$$
   
    Each function $(f_{k+1,\tau_k}*\rho_{\tau_k})\cdot (f_{k+1,\tau_k'}*\rho_{\tau_k'})$ is Fourier supported on the sumset $R^{2\delta}\cdot\tau_k+R^{2\delta}\cdot\tau_k'$, while the weight $\phi_{Q_{k}}$ is Fourier supported on $B(0,\lambda_k)$.
    A minor modification of the argument presented in \cite[Lemma 2.4]{SZ} shows that if the constant $K$ is chosen sufficiently large, then the sumsets $\{\tau_k+\tau_k'\}_{\tau_k,\tau_k'}$ are finitely-overlapping. In particular, we will choose $K\geq 4$. Since $\tau_1\nsim\tau_1'$, Lemma \ref{lemma: direction separation tau_k}, dictates that the angle between the long directions of $\tau_1,\tau_1'$ is  $\gtrsim(R_1\lambda_1)^{-1}\gtrsim R^{-3\epsilon}$, where the last step follows from \eqref{eqn: lambda k absolute range}. Hence, each sumset $\tau_k+\tau_k'$ is roughly a $\lambda_k\times R^{-3\epsilon}\lambda_k$ box (see Figure \ref{fig:transverse sumset}),
    \begin{figure}[hh]
        \centering
        \includegraphics[width=9cm]{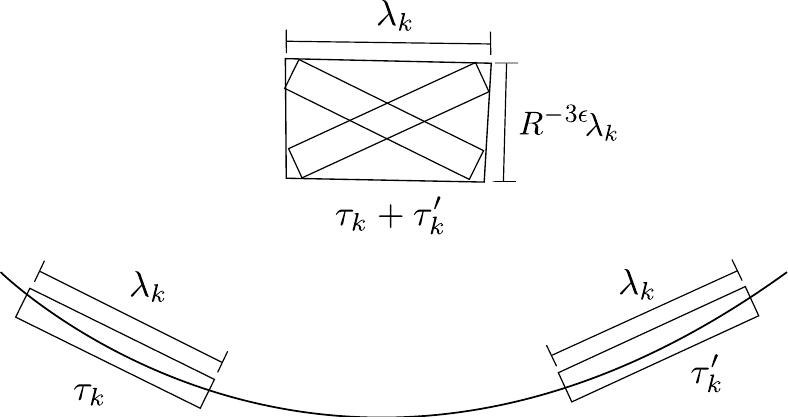}
        \caption{Sumsets of transversal boxs $\tau_k\in\calt_k(\tau_1)$ and $\tau_k'\in\calt_k(\tau_1')$.}
        \label{fig:transverse sumset}
    \end{figure}
    and so the sets $R^{2\delta}\cdot\tau_k+R^{2\delta}\cdot\tau_k'+B(0,\lambda_k)$ are at most $O(R^{3\epsilon+4\delta})$-overlapping. By local $L^2$-orthogonality (Lemma \ref{lemma: local orthogonality}), we get $$\int_{\R^2}|\sum_{\tau_k\prec\tau_1,\tau_k'\prec\tau_1'}(f_{k+1,\tau_k}*\rho_{\tau_k})\cdot (f_{k+1,\tau_k'}*\rho_{\tau_k'})|^2\phi_{Q_{k}}\lesssim R^{4\epsilon}\int_{\R^2}\sum_{\tau_k\prec\tau_1,\tau_k'\prec\tau_1'}|(f_{k+1,\tau_k}*\rho_{\tau_k})\cdot (f_{k+1,\tau_k'}*\rho_{\tau_k'})|^2|\phi_{Q_{k}}|,$$ where we used that $R^{4\delta}\leq R^\epsilon$ since $\epsilon<1/2$.
    Now $$\sum_{\tau_k\prec\tau_1,\tau_k'\prec\tau_1'}|(f_{k+1,\tau_k}*\rho_{\tau_k})\cdot (f_{k+1,\tau_k'}*\rho_{\tau_k'})|^2=\big(\sum_{\tau_k\in\calt_k(\tau_1)}|f_{k+1,\tau_k}*\rho_{\tau_k}|^2\big)\cdot\big(\sum_{\tau_k'\in\calt_k(\tau_1')} |f_{k+1,\tau_k'}*\rho_{\tau_k'}|^2\big),$$
     so, by the rapid decay of $\phi_{Q_{k}}$ away from $Q_{k}$ and Cauchy--Schwarz, we get 
     \begin{align*}
         \int\big(\sum_{\tau_k\in\calt_k}|f_{k+1,\tau_k}*\rho_{\tau_k}|^2\big)\big(\sum_{\tau_k'\in\calt_k} |f_{k+1,\tau_k'}*\rho_{\tau_k'}|^2\big)\phi_{Q_{k}}&\leq \|\sum_{\tau_k\in\calt_k}|f_{k+1,\tau_k}*\rho_{\tau_k}|^2\|_{L^2(R^\delta\cdot Q_{k})}\|\sum_{\tau_k'\in\calt_k}|f_{k+1,\tau_k'}*\rho_{\tau_k'}|^2\|_{L^2(R^\delta\cdot Q_{k})}\\&+R^{-100}.
     \end{align*}
    By another application of Cauchy--Schwarz, Lemma \ref{lemma: rho}, and the construction of $\omega_{\tau_k}$ from Definition \ref{def: w,omega}, we have $$\sum_{\tau_k\in\calt_k(\tau_1)}|f_{k+1,\tau_k}*\rho_{\tau_k}|^2\lesssim\sum_{\tau_k\in\calt_k(\tau_1)}|f_{k+1,\tau_k}|^2*|\rho_{\tau_k}|\leq\sum_{\tau_k\in\calt_k(\tau_1)}|f_{k+1,\tau_k}|^2*\omega_{\tau_k}.$$ Arguing similarly for the other summand concludes the proof.
\end{proof}
\begin{proof}
    [Proof of \eqref{eqn: omega k estimate} for $k\geq 1$]
By Corollary \ref{cor: replacing f}, we have
 $$|f(x)|\lesssim R^{2\epsilon}\max_{\substack{\tau_1,\tau_1'\in\calt_1\\\tau_1\nsim\tau_1'}}|f_{k+1,\tau_1}(x)|^{1/2}|f_{k+1,\tau_1'}(x)|^{1/2}\quad\text{for all}\quad x\in\Omega_k\cap\U.$$
Taking fourth powers and integrating, we get $$\alpha^4|\Omega_k\cap\U|\lesssim R^{8\epsilon}\int_{\Omega_k}\max_{\substack{\tau_1,\tau_1'\in\calt_1\\\tau_1\nsim\tau_1'}}|f_{k+1,\tau_1}(x)|^2|f_{k+1,\tau_1'}(x)|^2\dd{x}\lesssim R^{8\epsilon}\sum_{\substack{\tau_1,\tau_1'\in\calt_1\\\tau_1\nsim\tau_1'}}\int_{\Omega_k}|f_{k+1,\tau_1}(x)|^2|f_{k+1,\tau_1'}(x)|^2\dd{x},$$ where we have used the fact $|f(x)|\geq \alpha$ for all $x\in\U$.
We bound each of the above integrals as $$\int_{\Omega_k}|f_{k+1,\tau_1}(x)|^2|f_{k+1,\tau_1'}(x)|^2\dd{x}\leq\sum_{Q_{k}\in\Q_{k}}\int_{Q_{k}}|f_{k+1,\tau_1}(x)|^2|f_{k+1,\tau_1'}(x)|^2\dd{x},$$
and estimate the right-hand integrals using Lemma \ref{lemma: local bilinear square function} applied to the squares $Q_k\in\Q_k$.
This leads to $$\alpha^4|\Omega_k\cap\U|\lesssim R^{12\epsilon}\sum_{\tau_1\nsim\tau_1'}\sum_{Q_{k}\in\Q_{k}}\|\sum_{\tau_{k}\in\calt_{k}(\tau_1)}|f_{k+1,\tau_{k}}|^2*\omega_{\tau_{k}}\|_{L^2(R^\delta\cdot Q_{k})}\|\sum_{\tau_{k}'\in\calt_{k}(\tau_1')}|f_{k+1,\tau_{k}'}|^2*\omega_{\tau_{k}'}\|_{L^2(R^\delta\cdot Q_{k})}+R^{-10},$$
    and upon relaxing the ranges of summation above, and using Lemma \ref{lemma: number of boxes}, and the definition of $g_k$ from Definition \ref{def: g_k}, we get 
    \begin{equation}
        \label{eqn: post-bilinear square function}
        \alpha^4|\Omega_k\cap\U|\lesssim R^{14\epsilon}\sum_{Q_{k}\in\Q_{k}}\|g_k\|_{L^2(R^\delta\cdot Q_{k})}^2+R^{-10}.
    \end{equation}
    By Lemma \ref{lemma: high dom} we have 
    \begin{equation}
        \label{eqn: g_k to g_k hi}
        g_k^2(x)\leq\|g_k\|_{L^\infty(R^\delta\cdot Q_{k})}^2\leq 2\|g_k^{\mathrm{hi}}\|_{L^\infty(R^\delta\cdot Q_{k})}^2+R^{-100},\quad\text{for all}\quad x\in R^\delta\cdot Q_{k}.
    \end{equation}
    Similar to the proof of Lemma \ref{lemma: unwinding}, we would like to invoke the essential Fourier support of the square functions $g_k^{\mathrm{hi}}$. For this, we would like to use the following square-function analogue of \eqref{eqn: unwinding uncertainty use} 
    \begin{equation}
        \label{eqn: square function uncertainty}
        g_{k}^{\mathrm{hi}}=g_{k}^{\mathrm{hi}}*\rho_{k}+O(R^{-100}),
    \end{equation}
    which can be proved similarly. Using \eqref{eqn: square function uncertainty} followed by an application of
     Cauchy--Schwarz and Lemma \ref{lemma: rho} gives us $$|g_k^{\mathrm{hi}}(y)|^2\lesssim|g_k^{\mathrm{hi}}*\rho_{{k}}(y)|^2+R^{-100
}\lesssim|g_k^{\mathrm{hi}}|^2*|\rho_{k}|(y)+R^{-100}\quad\text{for all}\quad y\in R^\delta\cdot Q_{k}.$$
Now for all $x,y\in R^\delta\cdot Q_k$, we have $y-x\in 2R^\delta\lambda_{k}^{-1}\B$. Using the above, and the definition of $w_{k}$ from Definition \ref{def: w,omega} in \eqref{eqn: g_k to g_k hi}, we find $$g_k(x)^2\lesssim|g_k^{\mathrm{hi}}|^2*w_{k}(x)+R^{-100}\quad\text{for all}\quad x\in R^\delta\cdot Q_{k},$$
so that
\begin{equation}
    \label{eqn: g_k L^2 norm bound}
    \|g_k\|_{L^2(R^\delta\cdot Q_{k})}^2\lesssim \int_{R^\delta\cdot Q_{k}} |g_k^{\mathrm{hi}}|^2*w_{k}(x)\dd{x}+R^{-90}.
\end{equation}
Let us split the right-hand integral above as follows. 
\begin{equation}
    \label{eqn: g_k hi integral decomp}
    \int_{R^\delta\cdot Q_{k}} |g_k^{\mathrm{hi}}|^2*w_{k}(x)\dd{x}=\int_{R^\delta\cdot Q_{k}} \int_{\R^2\setminus  8R^\delta \cdot Q_{k}}|g_k^{\mathrm{hi}}(y)|^2w_{k}(x-y)\dd{y}\dd{x}+\int_{R^\delta\cdot Q_{k}} \int_{8R^\delta\cdot Q_{k}}|g_k^{\mathrm{hi}}(y)|^2w_{k}(x-y)\dd{y}\dd{x}.
\end{equation}
In the first region, we have $x\in R^\delta\cdot Q_{k}$, and $y\notin 8R^\delta\cdot Q_{k}$, so that $x-y\notin 5R^\delta\lambda_{k}^{-1}\B$, and $|A_{k}^{-t}(x-y)|\geq 5R^\delta32^{N-k}$. Thus, by \eqref{eqn: w_k rapid decay}, we have $$|w_{k}(x-y)|\lesssim_M R^{2\delta}|\det A_{k}|^{-1}(1+R^\delta(|A_{k}^{-t}(x-y)|-4R^\delta\cdot 32^{N-k}))^{-M}\quad\text{for all}\quad M\in\mathbb{N}.$$
Choosing $M=2+50/\delta$, we then get 
\begin{equation}
    \label{eqn: 1st integral first bound}
    \int_{R^\delta\cdot Q_{k}} \int_{\R^2\setminus 8R^\delta\cdot Q_{k}}|g_k^{\mathrm{hi}}(y)|^2w_{k}(x-y)\dd{y}\dd{x}\lesssim R^{-100}\|g_k^{\mathrm{hi}}\|_{L^2(\R^2)}^2.
\end{equation}
By Definition \ref{def: hi lo} and Lemma \ref{lemma: eta}, we have 
\begin{equation}
    \label{eqn: g_k hi norm less than g_k norm}
    \|g_k^{\mathrm{hi}}\|_{L^2(\R^2)}^2\lesssim\|g_k\|_{L^2(\R^2)}^2.
\end{equation}
By Lemma \ref{lemma: number of boxes} and Cauchy--Schwarz, we have the following crude estimate 
\begin{align*}
    \|g_k\|_{L^2(\R^2)}^2=\|\sum_{\tau_{k}\in\calt_{k}}|f_{k+1,\tau_{k}}|^2*\omega_{\tau_{k}}\|_{L^2(\R^2)}^2&\lesssim R^{1/2}\sum_{\tau_{k}\in\calt_{k}}\|f_{k+1,\tau_{k}}*\omega_{\tau_{k}}\|_{L^2(\R^2)}^2\\&\lesssim R^{1/2+4\delta}\sum_{\tau_{k}\in\calt_{k}}\|f_{k+1,\tau_{k}}\|^2_{L^2(\R^2)},
\end{align*}
 where we used Young's convolution inequality and Lemma \ref{lemma: L^1 norm of omega} in the last step. Now for each $\tau_{k}\in\calt_{k}$, we have $$\|f_{k+1,\tau_{k}}\|^2_{L^2(\R^2)}=\|\sum_{\tau_{k+1}\in\calt_{k+1}(\tau_{k})}f_{k+1,\tau_{k+1}}\|_{L^2(\R^2)}^2\lesssim R^{\epsilon}\sup_{\tau_{k+1}\in\calt_{k+1}(\tau_{k})}\|f_{k+1,\tau_{k+1}}\|_{L^2(\R^2)}^2,$$ where we have again used Lemma \ref{lemma: number of boxes}. By Lemma \ref{lemma: properties of f_k} (iv) and \eqref{eqn: normalisation in sup norm}, we have $$\|f_{k+1,\tau_{k+1}}\|_{L^2(\R^2)}^2\lesssim R^2\quad\text{for all}\quad \tau_{k+1}\in\calt_{k+1}(\tau_{k}).$$ Using the above estimates in \eqref{eqn: g_k hi norm less than g_k norm}, we get $$\|g_k^{\mathrm{hi}}\|_{L^2(\R^2)}\lesssim R^3,$$ and using this in \eqref{eqn: 1st integral first bound}, we obtain 
\begin{equation}
    \label{eqn: 1st integral final bound}
     \int_{R^\delta\cdot Q_{k}} \int_{\R^2\setminus 8R^\delta\cdot Q_{k}}|g_k^{\mathrm{hi}}(y)|^2w_{k}(x-y)\dd{y}\dd{x}\lesssim R^{-97},
\end{equation}
 so that the contribution from the first integral in \eqref{eqn: g_k hi integral decomp} is negligible.
 The second integral in \eqref{eqn: g_k hi integral decomp} can be treated as follows $$\int_{R^\delta\cdot Q_{k}} \int_{8R^\delta\cdot Q_{k}}|g_k^{\mathrm{hi}}(y)|^2w_{k}(x-y)\dd{y}\dd{x}\leq \|g_k^{\mathrm{hi}}\|_{L^2(8R^\delta\cdot Q_{k})}^2\cdot\|w_{k}\|_{L^1(\R^2)}\lesssim _\epsilon R^{2\delta}\|g_k^{\mathrm{hi}}\|_{L^2(8R^\delta\cdot Q_{k})}^2,$$ upon using Lemma \ref{lemma: L^1 norm of omega}. Using this and \eqref{eqn: 1st integral final bound} in \eqref{eqn: g_k hi integral decomp}, we get $$ \int_{R^\delta\cdot Q_{k}} |g_k^{\mathrm{hi}}|^2*w_{k}(x)\dd{x}\lesssim_\epsilon R^{2\delta}\|g_k^{\mathrm{hi}}\|_{L^2(8R^\delta\cdot Q_k)}^2+R^{-90},$$ which when used in \eqref{eqn: g_k L^2 norm bound}, gives $$\|g_k\|_{L^2(R^\delta\cdot Q_{k})}^2\lesssim_\epsilon R^{2\delta}\|g_k^{\mathrm{hi}}\|_{L^2(8R^\delta\cdot Q_k)}+R^{-90}.$$
Using the above in \eqref{eqn: post-bilinear square function} yields
\begin{equation}
    \label{eqn: pre high lemma est}
    \alpha^4|\Omega_k\cap\U|\lesssim_\epsilon R^{14\epsilon+2\delta}\sum_{Q_{k}\in\Q_{k}}\|g_k^{\mathrm{hi}}\|_{L^2(8R^\delta\cdot Q_{k})}^2+R^{-10}\lesssim_\epsilon R^{14\epsilon+2\delta}\|g_k^{\mathrm{hi}}\|_{L^2(8R^\delta B_R)}+R^{-10},
\end{equation} since $\bigsqcup_{Q_{k}\in\Q_{k}}Q_{k}=\Omega_k\subseteq B_R$, and the squares $R^\delta\cdot Q_{k}$ are at most $O(R^{2\delta})$-overlapping. Now $8R^\delta \cdot B_R$ can be covered by $O(R^{2\delta})$ many balls $Q_R$ of radius $R$. Applying the High Lemma to each $Q_R$ and adding the estimates together we get
$$\|g_k^{\mathrm{hi}}\|_{L^2(8R^{\delta}\cdot B_R)}^2\lesssim_\epsilon  R^{2\epsilon+2\delta}\sum_{\tau_{k}\in\calt_{k}}\|f_{k+1,\tau_{k}}\|_{L^4(\R^2)}^4+R^{-10},$$ which when combined with \eqref{eqn: pre high lemma est} gives \begin{equation}
    \label{eqn: omega_k penultimate bound}
    \alpha^4|\Omega_k\cap\U|\lesssim R^{16\epsilon+4\delta}\sum_{\tau_{k}\in\calt_{k}}\|f_{k+1,\tau_{k}}\|_{L^4(\R^2)}^4+R^{-10}.
\end{equation}
 By Lemma \ref{lemma: number of boxes} and  Lemma \ref{lemma: properties of f_k} (ii), we have  $$\|f_{k+1,\tau_{k}}\|_{L^\infty}=\|\sum_{\tau_{k+1}\in\calt_{k+1}(\tau_{k})}f_{k+1,\tau_{k+1}}\|_{L^\infty}\lesssim R^{\epsilon}\sup_{\tau_{k+1}\in\calt_{k+1}(\tau_{k})}\|f_{k+1,\tau_{k+1}}\|_{L^\infty}\lesssim R^\epsilon G.$$ Using the above in \eqref{eqn: omega_k penultimate bound} we find
 $$\alpha^4|\Omega_k\cap\U|\lesssim G^2R^{18\epsilon+4\delta}\sum_{\tau_{k}\in\calt_{k}}\|f_{k+1,\tau_{k}}\|_{L^2}^2+R^{-10}.$$
 Since we have an $L^2$ expression on the right-hand side, we may apply Lemma \ref{lemma: unwinding}, which leads to the following 
 \begin{equation}
     \label{eqn: omega_k ultimate bound}
     \alpha^4|\Omega_k\cap\U|\lesssim G^2R^{19\epsilon+4\delta}\sum_{\theta\in\Theta}\|f_\theta\|_{L^2(\R^2)}^2+R^{-10}\lesssim_\epsilon R^{19\epsilon+4\delta}(\beta^2/\alpha^2)\sum_{\theta\in\Theta}\|f_\theta\|_{L^2(\R^2)}^2+R^{-10},
 \end{equation} by using the definition of $G$ from Notation \ref{not: G}. By the definition of the set $\U$ from \eqref{eqn: U set definition}, we have $$\beta\leq \sum_{\theta\in\Theta}\|f_\theta\|_{L^\infty(\U)}^2\leq\sum_{\theta\in\Theta}\|f_\theta\|_{L^\infty(\R^2)}^2.$$ Rearranging the terms in \eqref{eqn: omega_k ultimate bound} and using the above bound on $\beta$, we obtain 
 \begin{equation}
     \label{eqn: omega_k contribution}
     \alpha^6|\Omega_k\cap\U|\lesssim_\epsilon R^{20\epsilon}\big(\sum_{\theta\in\Theta}\|f_\theta\|_{L^\infty(\R^2)}^2\big)^2\big(\sum_{\theta\in\Theta}\|f_\theta\|_{L^2(\R^2)}^2\big)+R^{-10}.
 \end{equation}
 In light of the normalisation in \S\ref{sec: normalising f}, the error term $R^{-10}$ above is negligible.
  \end{proof}
\subsection{Proof of Theorem \ref{thm: convex decoupling ver 2 pigeonholed}}\label{sec: broad/narrow} From here on, we revert back to the notation being used before Notation \ref{not: lambda gamma}. Let us recall the theorem statement. \pigeonholed*
This will follow as a consequence of Theorem \ref{thm: weak type} after pigeonholing and a broad/narrow argument.
\subsubsection{Removing $\alpha$ and $\beta$}\label{subsec: alpha and beta} Recall the parameter $K$ appearing in Definition \ref{def: transversality}. Define the \textit{broad set}
\begin{align*}
    B_{\tau_0}:=\bigg\{x\in B_R: |{f^\Lambda}(x)|\leq 2KR^{2\epsilon}&\max_{\substack{\tau_1,\tau_1'\in\calt^\Lambda_1(\Gamma)\\\tau_1\nsim\tau_1'}}|f^\Lambda_{\tau_1}(x)|^{1/2}|f^\Lambda_{\tau_1'}(x)|^{1/2},\\&\big(\sum_{\tau_1\in\calt^\Lambda_1(\Gamma)}|f^\Lambda_{\tau_1}(x)|^6\big)^{1/6}\leq 2KR^{2\epsilon}\max_{\substack{\tau_1,\tau_1'\in\calt^\Lambda_1(\Gamma)\\\tau_1\nsim\tau_1'}}|f^\Lambda_{\tau_1}(x)|^{1/2}|f^\Lambda_{\tau_1'}(x)|^{1/2}\bigg\}.
\end{align*}
 Here we prove the following 
 \begin{prop}\label{prop: broad estimate}
     There exists a constant $C^{\mathrm{br}}_\epsilon\geq 1$ such that for each $\Lambda\in\Z(f,R,\epsilon)$ and each $0\leq N\leq 1/\epsilon$, the following estimate holds
     \begin{equation}
    \label{eqn: broad set bound}
    \|{f^\Lambda}\|_{L^6(B_{\tau_0})}^6\leq C_\epsilon^{\mathrm{br}}R^{20\epsilon}\big(\sum_{\tau_N\in\calt^\Lambda_N(\Gamma)}\|f^\Lambda_{\tau_N}\|_{L^\infty}^2\big)^2\big(\sum_{\tau_k\in\calt^\Lambda_N(\Gamma)}\|f^\Lambda_{\tau_N}\|_{L^2}^2\big).
\end{equation} 
 \end{prop}
\begin{proof}For $\alpha>0$, define $B_{\tau_0,\alpha}:=B_{\tau_0}\cap\{x\in B_R:\alpha\leq |{f^\Lambda}(x)|\leq 2\alpha\}$. In light of the normalisation \eqref{eqn: normalisation in sup norm}, it follows that $$|{f^\Lambda}(x)|=|\sum_{\tau_N} f^\Lambda_{\tau_N}(x)|\leq\sum_{\tau_N}\|f^\Lambda_{\tau_N}\|_{L^\infty(\R^2)}\lesssim R^{1/2}\quad\text{for all}\quad x\in\R^2.$$  On the other hand, $$\int_{B_R\cap\{|{f^\Lambda}|\leq R^{-1/2}\}}|{f^\Lambda}|^6\leq R^{-3}|B_R|\lesssim R^{-1}\lesssim R^{-1}.$$ 
By dyadic pigeonholing, it follows that  
\begin{equation}
    \label{eqn: alpha removal}
   \|{f^\Lambda}\|^6_{L^6(B_{\tau_0})}\lesssim(\log R)\|{f^\Lambda}\|_{L^6(B_{\tau_0,\alpha})}^6+R^{-1},\quad\text{for some dyadic}\quad\alpha\in[R^{-1/2},R^{1/2}].
\end{equation}
For this value of $\alpha$, and $\beta>0$ define $B_{\tau_0,\alpha,\beta}:=B_{\tau_0,\alpha}\cap\{x\in B_R:\beta\leq \sum_{{\tau_N}\in\Theta}|f^\Lambda_{\tau_N}(x)|^2\leq 2\beta\}$. By Cauchy--Schwarz, $$\alpha\leq\sum_{\tau_N}|f^\Lambda_{\tau_N}(x)|\leq R^{1/4}\big(\sum_{\tau_N}|f^\Lambda_{\tau_N}(x)|^2\big)^{1/2}\quad\text{for all}\quad x\in B_{\tau_0,\alpha,\beta}.$$
On the other hand, $$\sum_{\tau_N}|f^\Lambda_{\tau_N}(x)|^2\lesssim R^{1/2}\max_{\tau_N}\|f^\Lambda_{\tau_N}\|_{L^\infty(\R^2)}^2=R^{1/2}\quad\text{for all}\quad x\in\R^2.$$ Again, by dyadic pigeonholing, we get that \begin{equation}
    \label{eqn: beta removal}
    \|{f^\Lambda}\|^6_{L^6(B_{\tau_0,\alpha})}\lesssim (\log R)\|{f^\Lambda}\|_{L^6(B_{\tau_0,\alpha,\beta})}^6,\quad\text{for some dyadic}\quad \beta\in[R^{-1/2}\alpha^2,R^{1/2}].
\end{equation}
But the set $B_{\tau_0,\alpha,\beta}$ is equal to the set $U^{\text{br}}_{\alpha,\beta}$. Thus, by Theorem \ref{thm: weak type} we get $$\|{f^\Lambda}\|^6_{L^6(B_{\tau_0,\alpha,\beta})}\lesssim\alpha^6|B_{\tau_0,\alpha,\beta}|\leq C_\epsilon R^{20\epsilon}\big(\sum_{\tau_N}\|f^\Lambda_{\tau_N}\|_{L^\infty}^2\big)^2\big(\sum_{\tau_N}\|f^\Lambda_{\tau_N}\|_{L^2}^2\big).$$
Using this with \eqref{eqn: beta removal} and \eqref{eqn: alpha removal}, we get the desired bound \eqref{eqn: broad set bound} upon choosing $C_{\epsilon}^{\mathrm{br}}$ sufficiently large.
\end{proof}
\subsubsection{Removing the broad hypothesis}\label{sec: broad hypothesis}
Here we obtain the bound \eqref{eqn: convex decoupling ver 2 pigeonholed}.
Set $B'_{\tau_0}:=B_{\tau_0}$, and for each $0\leq k\leq N$ and each $\tau_k\in\calt^\Lambda_k(\Gamma)$, define the corresponding broad sets \begin{align*}
    B'_{\tau_k}:=\bigg\{x\in B_R: |f^\Lambda_{\tau_k}(x)|&\leq 2KR^{2\epsilon}\max_{\substack{\tau_{k+1},\tau_{k+1}'\in\calt_{k+1}(\Gamma;\tau_k)\\\tau_{k+1}\nsim\tau_{k+1}'}}|f^\Lambda_{\tau_{k+1}}(x)|^{1/2}|f^\Lambda_{\tau_{k+1}'}(x)|^{1/2},\\&\big(\sum_{\tau_{k+1}\in\calt_{k+1}(\Gamma;\tau_k)}|f^\Lambda_{\tau_{k+1}}(x)|^6\big)^{1/6}\leq 2KR^{2\epsilon}\max_{\substack{\tau_{k+1},\tau_{k+1}'\in\calt_{k+1}(\Gamma;\tau_k)\\\tau_{k+1}\nsim\tau_{k+1}'}}|f^\Lambda_{\tau_{k+1}}(x)|^{1/2}|f^\Lambda_{\tau_{k+1}'}(x)|^{1/2}\bigg\}.
\end{align*}
By applying Proposition \ref{prop: broad estimate} to each rescaled curve $\Gamma_{\tau_k}$, we obtain the following broad set estimates.
\begin{prop}\label{prop: k-broad bound}
    Let $0\leq N\leq 1/\epsilon$ and $0\leq k\leq N$. For each $\tau_k\in\calt_k^\Lambda(\Gamma)$, we have the broad-set bound
    \begin{equation}
        \label{eqn: k-broad bound}
        \|f^\Lambda_{\tau_k}\|_{L^6(B'_{\tau_k})}^6\leq C_\epsilon^{\mathrm{br}}R^{20\epsilon}\big(\sum_{\tau_N\in\calt_N^\Lambda(\Gamma;\tau_k)}\|f^\Lambda_{\tau_N}\|_{L^\infty}^2\big)^2\big(\sum_{\tau_N\in\calt_N^\Lambda(\Gamma;\tau_k)}\|f^\Lambda_{\tau_N}\|_{L^2}^2\big).
    \end{equation}
\end{prop}
In order to prove the above, we need to relate the estimate \eqref{eqn: k-broad bound} to a scaled version of \eqref{eqn: broad set bound} for the admissible curve $\Gamma_{\tau_k}$ (recall Lemma \ref{lemma: rescaled curves}). Lemma \ref{lemma: rescaling} will be used for this. Since Proposition \ref{prop: broad estimate} is proved using Theorem \ref{thm: weak type}, which relied on the bound \eqref{eqn: theta size} to achieve the normalisation in \S\ref{sec: normalising f}, we need the following analogue for the rescaled curves:
\begin{equation}
    \label{eqn: rescaled curve normalisation}
    |\cal_{\tau_k}(\tau_N)|\leq R^2\quad\text{for all}\quad\tau_k\in\calt_k^\Lambda(\Gamma)\quad\text{and all}\quad\tau_N\in\calt_N^\Lambda(\Gamma;\tau_k).
\end{equation}
To see why this is true, recall each $\tau_N$ is contained in a $2\lambda_N\times R_N^{-1}$ rectangle. By \eqref{eqn: lambda k absolute range} we have $\lambda_N\leq R^{2\epsilon}$, and by \eqref{eqn: cal tau k def}, we have that $|\det\cal_{\tau_k}|=R_k^2$. It follows that $$|\cal_{\tau_k}(\tau_N)|=R_k^2|\tau_N|\leq 2R^{(2+N)\epsilon}.$$ Then \eqref{eqn: rescaled curve normalisation} follows from the fact that $N\leq 1/\epsilon$. 
\begin{proof}[Proof of Proposition \ref{prop: k-broad bound}]
    The idea is to relate \eqref{eqn: k-broad bound} to a scaled version of \eqref{eqn: broad set bound}. We do this by applying the affine transformation $\cal_{\tau_k}^t$  (defined in \S\ref{sec: box definition}) to the spatial domain, and observing its effect on the quantities in \eqref{eqn: k-broad bound}. Let the image of $f$ under this transformation be denoted by $f\circ\cal_{\tau_k}^t=:g$. Since $f_{\tau_N}=f*\widecheck{\chi}_{\tau_N}$, it follows that $f_{\tau_N}\circ\cal_{\tau_k}^t=g_{\cal_{\tau_k}(\tau_N)}$. Then by \eqref{eqn: f^Lambda_tau_k representation 2} we have
    \begin{equation*}
f^\Lambda_{\tau_{k}}\circ\cal_{\tau_k}^t=\sum_{\tau_N\in\calt_N^\Lambda(\Gamma;\tau_{k})}f_{\tau_N}\circ\cal_{\tau_k}^t=\sum_{\tau_N\in\calt_N^\Lambda(\Gamma;\tau_{k})}g_{\cal_{\tau_k}(\tau_N)}.
    \end{equation*}
    By Lemma \ref{lemma: rescaling}, $\cal_{\tau_k}$ maps the collection $\calt_N^\Lambda(\Gamma;\tau_k)$ onto $\calt_{N-k}^{\Lambda'}(\Gamma_{\tau_k})$. Thus,
    \begin{equation}
        \label{eqn: f tau j image}
        f^\Lambda_{\tau_{k}}\circ\cal_{\tau_k}^t=\sum_{\tilde\tau_{N-k}\in\calt^{\Lambda'}_{N-k}(\Gamma_{\tau_k})}g_{\tilde\tau_{N-k}}=g^{\Lambda'},
    \end{equation}
    where the last equality follows from \eqref{eqn: f^Lambda representation 2}.
    The change of variables $y=\cal_{\tau_k}^{-t}x$ then reduces \eqref{eqn: k-broad bound} to proving the following bound
    \begin{equation}
        \label{eqn: k-broad bound transformed}
          \|g^{\Lambda'}\|_{L^6(D_{\tau_k})}^6\leq C_\epsilon^{\mathrm{br}}R^{20\epsilon}\big(\sum_{\tilde\tau_{N-k}\in\calt_{N-k}^{\Lambda'}(\Gamma_{\tau_k})}\|g_{\tilde\tau_{N-k}}\|_{L^\infty}^2\big)^2\big(\sum_{\tilde\tau_{N-k}\in\calt_{N-k}^{\Lambda'}(\Gamma_{\tau_k})}\|g_{\tilde\tau_{N-k}}\|_{L^2}^2\big),
    \end{equation}
    where $D_{\tau_k}:=\cal_{\tau_k}^t(B_{\tau_k}').$ Let $\tau_{k+1}\in\calt^\Lambda_{k+1}(\Gamma;\tau_k)$. By another application of Lemma \ref{lemma: rescaling} and \eqref{eqn: f^Lambda_tau_k representation 2}, we have $$f^\Lambda_{\tau_{k+1}}\circ\cal_{\tau_k}^t=\sum_{\tilde\tau_{N-k}\in\calt^{\Lambda'}_{N-k}(\Gamma_{\tau_k};\tilde\tau_1)}g_{\tilde\tau_{N-k}}=g_{\tilde\tau_1}^{\Lambda'},$$ where $\tilde\tau_1:=\cal_{\tau_k}(\tau_{k+1})\in\calt_1^{\Lambda'}(\Gamma_{\tau_k})$. Note that the transversality relation $\nsim$ is preserved under $\cal_{\tau_k}$. We can then express the set $D_{\tau_k}$ as follows
    \begin{align*}
        D_{\tau_k}=\bigg\{x\in \cal_{\tau_k}^{-t}(B_R): &|g^{\Lambda'}(x)|\leq 2KR^{2\epsilon}\max_{\substack{\tilde\tau_{1},\tilde\tau_{1}'\in\calt^{\Lambda'}_{1}(\Gamma_{\tau_k})\\\tilde\tau_{1}\nsim\tilde\tau_{1}'}}|g^{\Lambda'}_{\tilde\tau_{1}}(x)|^{1/2}|g^{\Lambda'}_{\tilde\tau_{1}'}(x)|^{1/2},\\&\big(\sum_{\tilde\tau_{1}\in\calt^{\Lambda'}_{1}(\Gamma_{\tau_k})}|g^{\Lambda'}_{\tilde\tau_{1}}(x)|^6\big)^{1/6}\leq 2KR^{2\epsilon}\max_{\substack{\tilde\tau_{1},\tilde\tau_{1}'\in\calt^{\Lambda'}_{1}(\Gamma_{\tau_k})\\\tilde\tau_{1}\nsim\tilde\tau_{1}'}}|g^{\Lambda'}_{\tilde\tau_{1}}(x)|^{1/2}|g^{\Lambda'}_{\tilde\tau_{1}'}(x)|^{1/2}\bigg\}.
    \end{align*}
    Observe that the set $\cal_{\tau_k}^{-t}(B_R)$ is contained in some $R$-ball, say, $Q_R$.
    Then 
    \begin{align*}
        D_{\tau_k}\subseteq \tilde B_{\tilde\tau_0}:=\bigg\{x\in Q_R: |g^{\Lambda'}(x)|&\leq 2KR^{2\epsilon}\max_{\substack{\tilde\tau_{1},\tilde\tau_{1}'\in\calt^{\Lambda'}_{1}(\Gamma_{\tau_k})\\\tilde\tau_{1}\nsim\tilde\tau_{1}'}}|g^{\Lambda'}_{\tilde\tau_{1}}(x)|^{1/2}|g^{\Lambda'}_{\tilde\tau_{1}'}(x)|^{1/2},\\&\big(\sum_{\tilde\tau_{1}\in\calt^{\Lambda'}_{1}(\Gamma_{\tau_k})}|g^{\Lambda'}_{\tilde\tau_{1}}(x)|^6\big)^{1/6}\leq 2KR^{2\epsilon}\max_{\substack{\tilde\tau_{1},\tilde\tau_{1}'\in\calt^{\Lambda'}_{1}(\Gamma_{\tau_k})\\\tilde\tau_{1}\nsim\tilde\tau_{1}'}}|g^{\Lambda'}_{\tilde\tau_{1}}(x)|^{1/2}|g^{\Lambda'}_{\tilde\tau_{1}'}(x)|^{1/2}\bigg\}.
    \end{align*}
    Now $\|g^{\Lambda'}\|_{L^6(D_{\tau_k})}\leq\|g^{\Lambda'}\|_{L^6(\tilde B_{\tilde\tau_0})}$, and from Lemma \ref{lemma: rescaled curves} we know that $\Gamma_{\tau_k}$ is an admissible curve. Additionally, we have $$|\tilde\tau_{N-k}|\leq R^2\quad\text{for all}\quad \tilde\tau_{N-k}\in\calt^{\Lambda'}_{N-k}(\Gamma_{\tau_k}),$$ which is just a reformulation of \eqref{eqn: rescaled curve normalisation}. Then \eqref{eqn: k-broad bound transformed} is a direct consequence of Proposition \ref{prop: broad estimate} applied to the function $g^{\Lambda'}$.
\end{proof}
The key ingredient to removing the broad hypothesis is the broad/narrow inequality, which is free from any Fourier analytic considerations.
\begin{lemma}
    [Broad/narrow inequality]
    \label{lemma: broad narrow} Recall the parameter $K$ appearing in Definition \ref{def: transversality}, and let $f^\Lambda_{\tau_0}:=f^\Lambda$. Then for all $k\in\{0,\dots,N\}$ and all $\tau_k\in\calt_k^\Lambda(\Gamma)$ we have 
    \begin{equation}
        \label{eqn: broad narrow}
        |f^\Lambda_{\tau_k}(x)|\leq K\big(\sum_{\tau_{k+1}\in\calt^\Lambda_{k+1}(\tau_k)}|f^\Lambda_{\tau_{k+1}}(x)|^6\big)^{1/6}+\sum_{\substack{\tau_{k+1},\tau_{k+1}'\in\calt^\Lambda_{k+1}(\tau_k)\\\tau_{k+1}\nsim\tau_{k+1}'}}|f^\Lambda_{\tau_{k+1}}(x)|^{1/2}|f^\Lambda_{\tau_{k+1}'}(x)|^{1/2},
    \end{equation} for all $x\in\R^2$.
\end{lemma}
The idea is to decompose the function $f^\Lambda_{\tau_k}$ into its \textit{broad} and \textit{narrow} parts. The broad part of the function dominates at the points $x\in\R^2$ where the sequence $\{|f^\Lambda_{\tau_{k+1}}(x)|\}_{\tau_{k+1}\in\calt_{k+1}(\tau_k)}$ has a flat (or broad) distribution. The narrow part dominates at the points $x\in\R^2$ where $\{|f^\Lambda_{\tau_{k+1}}(x)|\}_{\tau_{k+1}\in\calt_{k+1}(\tau_k)}$ has a peaked (or narrow) distribution. The idea of broad/narrow decomposition goes back to the following paper of Bourgain--Guth \cite{BG}, while the exact form of the inequality \eqref{eqn: broad narrow} has appeared in works such as \cite[Lemma 2.8]{lee_ham} and \cite[Lemma 4.1]{KLO}.
\begin{proof}
    Fix $k\in\{0,\dots,N\}$ and $\tau_k\in\calt_k$. Recall from \eqref{eqn: f^Lambda_tau_k representation 2} that $$f^\Lambda_{\tau_k}(x)=\sum_{\tau_{k+1}\in\calt^\Lambda_{k+1}(\Gamma;\tau_k)}f^\Lambda_{\tau_{k+1}}(x)\quad\text{for all}\quad x\in\R^2.$$ We split the above sum into $$f^\Lambda_{\tau_k}(x)=\sum_{\substack{\tau_{k+1}\in\calt^\Lambda_{k+1}(\tau_k)\\\tau_{k+1}\sim\tau_{k+1}(x)}}f^\Lambda_{\tau_{k+1}}(x)+\sum_{\substack{\tau_{k+1}\in\calt^\Lambda_{k+1}(\tau_k)\\\tau_{k+1}\nsim\tau_{k+1}(x)}}f^\Lambda_{\tau_{k+1}}(x),$$ where $\tau_{k+1}(x)$ is defined by the relation $\displaystyle|f^\Lambda_{\tau_{k+1}(x)}(x)|=\max_{\tau_{k+1}\in\calt^\Lambda_{k+1}(\tau_k)}|f^\Lambda_{\tau_{k+1}}(x)|$. Following this definition, we have $$|f^\Lambda_{\tau_k}(x)|\leq K|f^\Lambda_{\tau_{k+1}(x)}(x)|+\sum_{\substack{\tau_{k+1}\in\calt^\Lambda_{k+1}(\tau_k)\\\tau_{k+1}\nsim\tau_{k+1}(x)}}|f^\Lambda_{\tau_{k+1}}(x)|^{1/2}|f^\Lambda_{\tau_{k+1}(x)}(x)|^{1/2},$$ where by definition we have $\#\{\tau_{k+1}\sim\tau_{k+1}(x)\}\leq K$.
    Using the fact $\|\cdot\|_{\ell^\infty}\leq\|\cdot\|_{\ell^6}$ on the first term, and relaxing the summation over the second term we get $$|f^\Lambda_{\tau_k}(x)|\leq K\big(\sum_{\tau_{k+1}\in\calt^\Lambda_{k+1}(\tau_k)}|f^\Lambda_{\tau_{k+1}}(x)|^6\big)^{1/6}+\sum_{\substack{\tau_{k+1},\tau_{k+1}'\in\calt^\Lambda_{k+1}(\tau_k)\\\tau_{k+1}\nsim\tau_{k+1}'}}|f^\Lambda_{\tau_{k+1}}(x)|^{1/2}|f^\Lambda_{\tau_{k+1}'}(x)|^{1/2},$$ proving the claim.
\end{proof}
Equipped with Proposition \ref{prop: k-broad bound} and Lemma \ref{lemma: broad narrow}, we are now ready to prove Theorem \ref{thm: convex decoupling ver 2 pigeonholed}. 
\begin{proof}[Proof of Theorem \ref{thm: convex decoupling ver 2 pigeonholed}] 
Let us recall the definition of the broad-set $B_{\tau_0}$ given by \begin{align*}
    B_{\tau_0}=\bigg\{x\in B_R: |{f^\Lambda}(x)|\leq 2KR^{2\epsilon}&\max_{\substack{\tau_1,\tau_1'\in\calt^\Lambda_1(\Gamma)\\\tau_1\nsim\tau_1'}}|f^\Lambda_{\tau_1}(x)|^{1/2}|f^\Lambda_{\tau_1'}(x)|^{1/2},\\&\big(\sum_{\tau_1\in\calt^\Lambda_1(\Gamma)}|f^\Lambda_{\tau_1}(x)|^6\big)^{1/6}\leq 2KR^{2\epsilon}\max_{\substack{\tau_1,\tau_1'\in\calt^\Lambda_1(\Gamma)\\\tau_1\nsim\tau_1'}}|f^\Lambda_{\tau_1}(x)|^{1/2}|f^\Lambda_{\tau_1'}(x)|^{1/2}\bigg\}.
\end{align*}
Set
$N_{\tau_0}:=B_R\setminus B_{\tau_0}$, so that for all $x\in N_{\tau_0}$, one of the following must hold:
\begin{itemize}
    \item  $\displaystyle|f^\Lambda(x)|> 2KR^{2\epsilon}\max_{\substack{\tau_1,\tau_1'\in\calt^\Lambda_1(\Gamma)\\\tau_1\nsim\tau_1'}}|f^\Lambda_{\tau_1}(x)|^{1/2}|f^\Lambda_{\tau_1'}(x)|^{1/2}$,
    \item  $\displaystyle|f^\Lambda(x)|\leq \big(\sum_{\tau_1\in\calt^\Lambda_1(\Gamma)}|f^\Lambda_{\tau_1}(x)|^6\big)^{1/6}$.
\end{itemize}
In the former case, the broad/narrow inequality \eqref{eqn: broad narrow} gives  $$|f^\Lambda(x)|\leq 2K\big(\sum_{\tau_1\in\calt^\Lambda_1(\Gamma)}|f^\Lambda_{\tau_1}(x)|^6\big)^{1/6}.$$ 
Hence, in either case we have
\begin{equation}
    \label{eqn: narrow step 0}
   \int_{B_R}|f^\Lambda|^6\leq \int_{B_{\tau_0}}|f^\Lambda|^6+(2K)^6\sum_{\tau_1\in\calt^\Lambda_1(\Gamma)}\int_{N_{\tau_0}}|f^\Lambda_{\tau_1}(x)|^6\dd{x}.
\end{equation}
Suppose at the end of $k$ iterations, we have the following inequality 
\begin{align}
    \label{eqn: narrow step k}
    \nonumber\int_{B_R}|f^\Lambda|^6\leq \int_{B_{\tau_0}}|f^\Lambda|^6+(2K)^6\sum_{\tau_1\in\calt^\Lambda_1(\Gamma)}\int_{B_{\tau_1}}|f^\Lambda_{\tau_1}(x)|^6\dd{x}&+\dots+(2K)^{6(k-1)}\sum_{\tau_{k-1}\in\calt_{k-1}^\Lambda(\Gamma)}\int_{B_{\tau_{k-1}}}|f^\Lambda_{\tau_{k-1}}(x)|^6\dd{x}\\&+(2K)^{6k}\sum_{\tau_k\in\calt^\Lambda_k(\Gamma)}\int_{N_{\tau_{k-1}}}|f^\Lambda_{\tau_k}(x)|^6\dd{x}.
\end{align}
Let us recall the definitions of the broad-sets $B_{\tau_k}'$, given by \begin{align*}
    B'_{\tau_k}=\bigg\{x\in B_R: |f^\Lambda_{\tau_k}(x)|&\leq 2KR^{2\epsilon}\max_{\substack{\tau_{k+1},\tau_{k+1}'\in\calt^\Lambda_{k+1}(\Gamma;\tau_k)\\\tau_{k+1}\nsim\tau_{k+1}'}}|f^\Lambda_{\tau_{k+1}}(x)|^{1/2}|f^\Lambda_{\tau_{k+1}'}(x)|^{1/2},\\&\big(\sum_{\tau_{k+1}\in\calt^\Lambda_{k+1}(\Gamma;\tau_k)}|f^\Lambda_{\tau_{k+1}}(x)|^6\big)^{1/6}\leq 2KR^{2\epsilon}\max_{\substack{\tau_{k+1},\tau_{k+1}'\in\calt^\Lambda_{k+1}(\Gamma;\tau_k)\\\tau_{k+1}\nsim\tau_{k+1}'}}|f^\Lambda_{\tau_{k+1}}(x)|^{1/2}|f^\Lambda_{\tau_{k+1}'}(x)|^{1/2}\bigg\}.
\end{align*}
For each $\tau_k$, we decompose $N_{\tau_{k-1}}:=B_{\tau_k}\cup N_{\tau_k},$ where $$B_{\tau_k}:=N_{\tau_{k-1}}\cap B_{\tau_k}'\quad\text{and}\quad N_{\tau_k}:=N_{\tau_{k-1}}\setminus B_{\tau_k}'.$$
For all $x\in N_{\tau_k}$ one of the following must hold:
\begin{itemize}
    \item $\displaystyle|f^\Lambda_{\tau_k}(x)|> 2KR^{2\epsilon}\max_{\substack{\tau_{k+1},\tau_{k+1}'\in\calt^\Lambda_{k+1}(\tau_k)\\\tau_{k+1}\nsim\tau_{k+1}'}}|f^\Lambda_{\tau_{k+1}}(x)|^{1/2}|f^\Lambda_{\tau_{k+1}'}(x)|^{1/2}$,
    \item $\displaystyle|f^\Lambda_{\tau_k}(x)|\leq \big(\sum_{\tau_{k+1}\in\calt^\Lambda_{k+1}(\tau_k)}|f^\Lambda_{\tau_{k+1}}(x)|^6\big)^{1/6}$.
\end{itemize}
Using the broad/narrow inequality \eqref{eqn: broad narrow} in the former case, we get $$|f^\Lambda_{\tau_k}(x)|\leq 2K\big(\sum_{\tau_{k+1}\in\calt^\Lambda_{k+1}(\tau_k)}|f^\Lambda_{\tau_{k+1}}(x)|^6\big)^{1/6}.$$ Hence, we update \eqref{eqn: narrow step k} to the following 
\begin{align}
    \label{eqn: narrow step k+1}
    \nonumber \int_{B_R}|f^\Lambda|^6\leq \int_{B_{\tau_0}}|f^\Lambda|^6+(2K)^6\sum_{\tau_1\in\calt^\Lambda_1(\Gamma)}\int_{B_{\tau_1}}|f^\Lambda_{\tau_1}(x)|^6\dd{x}&+\dots+(2K)^{6k}\sum_{\tau_{k}\in\calt^\Lambda_k(\Gamma)}\int_{B_{\tau_{k}}}|f^\Lambda_{\tau_{k}}(x)|^6\dd{x}\\&+(2K)^{6(k+1)}\sum_{\tau_{k+1}\in\calt_{k+1}^\Lambda(\Gamma)}\int_{N_{\tau_{k}}}|f^\Lambda_{\tau_{k+1}}(x)|^6\dd{x}.
\end{align}
The process terminates after $N$ steps, where our final inequality is the following 
\begin{align}
    \label{eqn: narrow step N}
     \nonumber\int_{B_R}|f^\Lambda|^6\leq \int_{B_{\tau_0}}|f^\Lambda|^6+(2K)^6\sum_{\tau_1\in\calt^\Lambda_1(\Gamma)}\int_{B_{\tau_1}}|f^\Lambda_{\tau_1}(x)|^6\dd{x}&+\dots+(2K)^{6(N-1)}\sum_{\tau_{N-1}\in\calt^\Lambda_{N-1}(\Gamma)}\int_{B_{\tau_{N-1}}}|f^\Lambda_{\tau_{N-1}}(x)|^6\dd{x}\\&+(2K)^{6N}\sum_{\tau_{N}\in\calt^\Lambda_N(\Gamma)}\int_{N_{\tau_{N-1}}}|f^\Lambda_{\tau_{N}}(x)|^6\dd{x}.
\end{align}
Since $B_{\tau_k}\subseteq B_{\tau_k'}$, the integrals over the sets $B_{\tau_k}$ ($k=0,\dots,N-1$) can be treated using Proposition \ref{prop: k-broad bound} as shown below
\begin{equation}
    \label{eqn: broad term bound}
    \sum_{\tau_k\in\calt^\Lambda_k(\Gamma)}\int_{B_{\tau_k}}|f^\Lambda_{\tau_k}(x)|^6\dd{x}\leq C_{\epsilon}^{\text{br}} R^{20\epsilon}\big(\sum_{{\tau_N}\in{\calt^\Lambda_N(\Gamma)}}\|f_{\tau_N}\|_{L^\infty}^2\big)^2\big(\sum_{{\tau_N}\in{\calt^\Lambda_N(\Gamma)}}\|f_{\tau_N}\|_{L^2}^2\big).
\end{equation}
For the remaining integral over the narrow set $N_{\tau_{N-1}}$, an application of H\"older's inequality gives us 
\begin{align}
    \label{eqn: narrow term bound}
    \sum_{\tau_{N}\in\calt^\Lambda_{N}(\Gamma)}\int_{N_{\tau_{N-1}}}|f^\Lambda_{\tau_{N}}(x)|^6\dd{x}&\leq\sum_{{\tau_N}\in{\calt^\Lambda_N(\Gamma)}}\|f^\Lambda_{\tau_N}\|_{L^6(B_R)}^6\leq\sum_{{\tau_N}\in{\calt^\Lambda_N(\Gamma)}}\|f^\Lambda_{\tau_N}\|_{L^\infty}^4\|f^\Lambda_{\tau_N}\|_{L^2}^2\nonumber\\&\leq\big(\sum_{{\tau_N}\in{\calt^\Lambda_N(\Gamma)}}\|f^\Lambda_{\tau_N}\|_{L^\infty}^2\big)^2\big(\sum_{{\tau_N}\in{\calt^\Lambda_N(\Gamma)}}\|f^\Lambda_{\tau_N}\|_{L^2}^2\big).
\end{align}
Combining \eqref{eqn: broad term bound} with \eqref{eqn: narrow term bound} in \eqref{eqn: narrow step N}, we obtain 
\begin{align*}
     \int_{B_R}|f^\Lambda|^6&\leq N(2K)^{6N}C_{\epsilon}^{\text{br}}R^{20\epsilon}\big(\sum_{{\tau_N}\in{\calt^\Lambda_N(\Gamma)}}\|f^\Lambda_{\tau_N}\|_{L^\infty}^2\big)^2\big(\sum_{{\tau_N}\in{\calt^\Lambda_N(\Gamma)}}\|f^\Lambda_{\tau_N}\|_{L^2}^2\big)\\&\lesssim_\epsilon R^{20\epsilon}\big(\sum_{{\tau_N}\in{\calt^\Lambda_N(\Gamma)}}\|f^\Lambda_{\tau_N}\|_{L^\infty}^2\big)^2\big(\sum_{{\tau_N}\in{\calt^\Lambda_N(\Gamma)}}\|f_{\tau_N}\|_{L^2}^2\big),
\end{align*} which concludes the proof.
\end{proof}
\section{Special examples}\label{sec: applications and examples}
By interpolation, Theorem \ref{thm: convex decoupling} implies that $$\|f\|_{L^q(\R^2)}\lesssim_\epsilon R^\epsilon\big(\sum_{\theta}\|f_\theta\|_{L^q(\R^2)}^2\big)^{1/2}\quad\text{for all}\quad 2\leq q\leq 6,$$ for a canonical box-decomposition $\{\theta\}$ of $\mathcal{N}_{R^{-1}}(\Gamma)$. It is well-known that the range of exponents $2\leq q\leq 6$ cannot be improved when $\Gamma$ has non-vanishing curvature. However, for general curves, this range is not sharp. A trivial example is when the affine dimension $\kappa_\Gamma=0$. In this case, the Cauchy--Schwarz inequality gives $$ \|f\|_{L^q(\R^2)}\lesssim_{\epsilon,\Gamma} R^\epsilon \big(\sum_{\theta}\|f_\theta\|_{L^q(\R^2)}^2\big)^{1/2}\quad\text{for all}\quad 2\leq q\leq \infty.$$
The following provides non-trivial examples of curves where the range can be improved.
\begin{prop}
    [Non-sharp examples, \cite{roy}]\label{prop: lambda(p) decoupling}
    For each $p>2$ and $\epsilon>0$, there exists a convex curve $\Gamma=\Gamma(p,\epsilon)$ satisfying $\kappa_\Gamma=1/p$, such that 
    \begin{equation}
        \label{eqn: lambda(p) decoupling}
         \|f\|_{L^q(\R^2)}\lesssim_{\epsilon,p} R^\epsilon\big(\sum_{\theta}\|f_\theta\|_{L^q(\R^2)}^2\big)^{1/2} \quad\text{for all}\quad 2\leq q\leq 3p.
    \end{equation}
\end{prop}
The details of the construction can be found in \cite[\S 3.1]{roy}, and the proof of the decoupling bound \eqref{eqn: lambda(p) decoupling} can be found in \cite[Corollary 4.3]{roy}. Here, we briefly discuss the main ideas.

Decoupling for the parabola exploits curvature. The curvature gives the parabola a non-linear structure, which is favourable for cancellation. Contrastingly, decoupling along a line can be far more restrictive. To illustrate this, let $N$ be a large integer, and define the intervals $J_k:=[k/N,(k+1)/N]$ for $k=0,\dots, N-1$. By taking $f$ to be a sum of bump functions adapted to $J_k$, we see that $$\|f\|_{L^q(\R)}\lesssim_\epsilon N^\epsilon\big(\sum_{k=0}^{N-1}\|f_{J_k}\|_{L^q(\R)}^2\big)^{1/2}\quad\text{fails for all}\quad q>2.$$ In fact, the best universal decoupling bound over the real line is Plancherel's theorem ($\ell^2 L^2$ decoupling). In \cite{guo}, the authors give an alternative proof of the classical $\ell^2 L^6$ decoupling for the parabola by `lifting' the one-dimensional $\ell^2 L^2$ decoupling over the real line. In some sense, this quantifies the role of curvature in decoupling. A similar heuristic is true for general convex curves, where we `mod out' the flat parts of the curve by encoding the curvature into the definition of the canonical boxes, which allows us to prove Theorem \ref{thm: convex decoupling}. However, the universal estimate of Theorem \ref{thm: convex decoupling} does not take into account any special arithmetic structure of the curve, which is another source of cancellation. The examples constructed in Proposition \ref{prop: lambda(p) decoupling} exhibit such properties. 

To construct these examples, we first use Bourgain's $\Lambda(p)$ sets to obtain a family $\{I_k\}_{k=0}^{N-1}$ of $N$ sub-intervals of $[-1/2,1/2]$, each of length $N^{-2/p}$, such that the following decoupling estimate holds 
\begin{equation}
    \label{eqn: 1st level decoupling}
    \|f\|_{L^q(\R)}\lesssim_p \big(\sum_{k=0}^{N-1}\|f_{I_k}\|_{L^q(\R)}^2\big)^{1/2}\quad\text{for all}\quad 2\leq q\leq p.
\end{equation} 
In fact, even a generic family satisfying the same conditions would also lead to the same estimate. We then construct a \textit{generalized Cantor set} $C$ from these intervals, and the curve $\Gamma$ is obtained by lifting $C$ to the parabola and taking its convex hull, whose affine dimension will be $1/p$. For each canonical box $\theta$, denote its projection onto the horizontal coordinate axis by $J_\theta$. By choosing the parameter $N$ sufficiently large (depending on $\epsilon$), the self-similarity of $C$ and \eqref{eqn: 1st level decoupling} implies the following one-dimensional decoupling estimate 
\begin{equation}
    \label{eqn: Cantor set decoupling}
      \|f\|_{L^q(\R)}\lesssim_{\epsilon,p} R^\epsilon\big(\sum_{\theta}\|f_{J_\theta}\|_{L^q(\R)}^2\big)^{1/2}\quad\text{for all}\quad 2\leq q\leq p.
\end{equation}
By the lifting heuristic discussed above, we expect the lift to the parabola to satisfy $\ell^2$-decoupling on a wider range of exponents. This is made precise by \cite[Theorem 1.1]{chang}, using which we obtain 
\begin{equation*}
      \|f\|_{L^q(\R^2)}\lesssim_{\epsilon,p} R^\epsilon\big(\sum_{\theta}\|f_{\theta}\|_{L^q(\R^2)}^2\big)^{1/2}\quad\text{for all}\quad 2\leq q\leq 3p,
\end{equation*}
which is precisely the estimate appearing in Proposition \ref{prop: lambda(p) decoupling}.
\label{sec: examples}
\appendix
\section{}
\subsection{Essential Fourier support} Here we give a proof of Lemma \ref{lemma: properties of f_k} (iii). As the arguments for both inequalities are very similar, we prove only the former.
\begin{lemma}
    \label{lemma: rapid decay induction}
    For all $k\in\{1,\dots,N\}$ and each $\tau_{k}$ we have $$|f_{k+1,\tau_{k}}(x)-f_{k+1,\tau_{k}}*\rho_{\tau_{k}}(x)|\lesssim R^{-100(k+1)}\quad\text{for all}\quad x\in\R^2.$$
\end{lemma}
\begin{proof}
    We prove the claim recursively, starting from $k=N$.
    Recall that $f_{N+1,\tau_N}=f_{\tau_N}$ is compactly Fourier supported in $\tau_N$, by definition. By Lemma \ref{lemma: rho}, it follows that $$\widehat f_{\tau_N}=\widehat f_{\tau_N}\cdot\widehat\rho_{\tau_N}\quad\text{so that}\quad f_{\tau_N}=f_{\tau_N}*\rho_{\tau_N},$$ establishing the base case.
    
    Now assume the conclusion holds for some $k>1$. We have $$f_{k+1,\tau_{k}}(x)-f_{k+1,\tau_{k}}*\rho_{\tau_{k}}(x)=\sum_{\tau_{k+1}\in\calt_{k+1}(\tau_k)}f_{k+1,\tau_{k+1}}(x)-f_{k+1,\tau_{k+1}}*\rho_{\tau_{k}}(x).$$ Denote $\mathcal{E}_{\tau_{k+1}}:=f_{k+2,\tau_{k+1}}-f_{k+2,\tau_{k+1}}*\rho_{\tau_{k+1}}$. By definition, $$f_{k+1,\tau_{k+1}}(x)-f_{k+1,\tau_{k+1}}*\rho_{\tau_{k}}(x)=\sum_{T\in\T^{\mathrm{g}}_{\tau_{k+1}}}(\psi_T\cdot f_{k+2,\tau_{k+1}})(x)-(\psi_T\cdot f_{k+2,\tau_{k+1}})*\rho_{\tau_k}(x),$$ which we split into two sums $S_1(x;\tau_{k+1})+S_2(x;\tau_{k+1})$, where $$S_1(x;\tau_{k+1}):=\sum_{T\in\T^{\mathrm{g}}_{\tau_{k+1}}}(\psi_T\cdot( f_{k+2,\tau_{k+1}}*\rho_{\tau_{k+1}}))(x)-(\psi_T\cdot(f_{k+2,\tau_{k+1}}*\rho_{\tau_{k+1}}))*\rho_{\tau_k}(x),$$ and $$S_2(x;\tau_{k+1})=\sum_{T\in\T^{\mathrm{g}}_{\tau_{k+1}}}(\psi_T\cdot\mathcal{E}_{\tau_{k+1}})(x)-(\psi_T\cdot\mathcal{E}_{\tau_{k+1}})*\rho_{\tau_k}(x).$$
    By the induction hypothesis, we have $|\mathcal{E}_{\tau_{k+1}}(x)|\lesssim_\epsilon R^{-100(k+2)}$, so that $$|(\psi_T\cdot\mathcal{E}_{\tau_{k+1}})(x)|\lesssim_\epsilon R^{-100(k+2)},$$ and  $$|(\psi_T\cdot\mathcal{E}_{\tau_{k+1}})*\rho_{\tau_k}(x)|\leq \int_{\R^2}|\psi_T(y)\mathcal{E}_{\tau_{k+1}}(y)||\rho_{\tau_k}(x-y)|\dd{y}\lesssim_\epsilon R^{-100(k+2)},$$ where we have used the fact that $\|\psi_T\|_{L^\infty}\leq 1$, and $\|\rho_{\tau_k}\|_{L^1}\lesssim 1$. Since $\#\T^{\mathrm{g}}_{\tau_{k+1}}\lesssim R^{20}$, this gives $$|S_2(x;\tau_{k+1})|\lesssim R^{-100(k+1)-1}.$$ For $S_1(x;\tau_{k+1})$, we use the Fourier inversion formula and write $$S_1(x;\tau_{k+1})=\sum_{T\in\T^{\mathrm{g}}_{\tau_{k+1}}}\int_{\xi\notin R^\delta 32^{N-k}\cdot\tau_k}e^{2\pi ix\cdot\xi}\big(\widehat{\psi}_T*(\widehat{f}_{k+2,\tau_{k+1}}\cdot\widehat{\rho}_{\tau_{k+1}})\big)(\xi)(1-\widehat{\rho}_{\tau_k}(\xi))\dd{\xi},$$ noting that $\widehat{\rho}_{\tau_k}(\xi)=1$ for $\xi\in R^\delta32^{N-k}\cdot\tau_k$ (Lemma \ref{lemma: rho}). Expanding the convolution, we can express the integral as $$\int_{\xi\notin R^\delta32^{N-k}\cdot\tau_k}e^{2\pi ix\cdot\xi}\bigg(\int_{\xi-\eta\in 2R^\delta32^{N-k-1}\cdot\tau_{k+1}}\widehat{\psi}_T(\eta)\widehat{f}_{k+2,\tau_{k+1}}(\xi-\eta)\widehat{\rho}_{\tau_{k+1}}(\xi-\eta)\dd{\eta}\bigg)(1-\widehat{\rho}_{\tau_k}(\xi))\dd{\xi},$$ where $\widehat{\rho}_{\tau_{k+1}}(\xi-\eta)=0$ for $\xi-\eta\notin 2R^\delta32^{N-k-1}\cdot\tau_{k+1}$ (Lemma \ref{lemma: rho}). This forces the integral to vanish in the region $\eta\in R^\delta\cdot\tau_k^{**}$. To see why, let $\eta\in R^\delta\cdot\tau_k^{**}$, so that in particular $\eta\in 16R^\delta32^{N-k-1}\cdot\tau_k^{**}$. In light of \eqref{eqn: tau k+1 containment amendment}, a simple computation reveals that $$2R^\delta32^{N-k-1}\cdot\tau_{k+1}\subseteq 16R^\delta32^{N-k-1}\cdot\tau_k.$$ 
    Combining with the above, this would imply that $\xi=(\xi-\eta)+\eta\in R^\delta 32^{N-k}\cdot\tau_k$, and so the integral vanishes.
    Interchanging the order of integration we get $$\int_{\eta\notin R^\delta\cdot\tau_k^{**}}\widehat{\psi}_T(\eta)\bigg(\int_{\R^2}e^{2\pi ix\cdot\xi}\widehat{f}_{k+2,\tau_{k+1}}(\xi-\eta)\widehat{\rho}_{\tau_{k+1}}(\xi-\eta)(1-\widehat{\rho}_{\tau_k}(\xi))\dd{\xi}\bigg)\dd{\eta}.$$ The inner integral is equal to $$e^{2\pi ix\cdot\eta}f_{k+2,\tau_{k+1}}*\rho_{\tau_{k+1}}(x)-\int_{\R^2}e^{2\pi i\eta\cdot(x-y)}(f_{k+2,\tau_{k+1}}*\rho_{\tau_{k+1}})(x-y)\rho_{\tau_k}(y)\dd{y},$$ the absolute value of which is controlled by $|f_{k+2,\tau_{k+1}}*\rho_{\tau_{k+1}}(x)|+|(f_{k+2,\tau_{k+1}}*\rho_{\tau_{k+1}})*\rho_{\tau_k}(x)|$. 
    Recall that $T\in\T^{\mathrm{g}}_{\tau_{k+1}}$ is a translate of $\tau_{k+1}^*$, and $\psi_T$ is a bump function adapted to $2T$. Hence, $\widehat{\psi}_T$ is rapidly decaying away from $\tau_{k+1}^{**}$. Another application of \eqref{eqn: tau k+1 containment amendment} shows that $\tau_{k+1}^{**}\subseteq 8\cdot\tau_k^{**}$. Using this, and the fact that $\#\T^{\mathrm{g}}_{\tau_{k+1}}\lesssim R^{20}$, we get $$|S_1(x;\tau_{k+1})|\lesssim R^{-100(k+1)-10}(|f_{k+2,\tau_{k+1}}*\rho_{\tau_{k+1}}(x)|+|(f_{k+2,\tau_{k+1}}*\rho_{\tau_{k+1}})*\rho_{\tau_k}(x)|).$$
    By Lemma \ref{lemma: rho}, we have $$|f_{k+2,\tau_{k+1}}*\rho_{\tau_{k+1}}(x)|+|(f_{k+2,\tau_{k+1}}*\rho_{\tau_{k+1}})*\rho_{\tau_k}(x)|\lesssim\|f_{k+2,\tau_{k+1}}\|_{L^\infty}.$$ By the definition of the pruned functions, we have $f_{k+1,\tau_{k+1}}=\sum_{\tau_{k+2}\in\calt_{k+2}(\tau_{k+1})}f_{k+2,\tau_{k+2}}$, so that by Lemma \ref{lemma: properties of f_k} (ii) and Lemma \ref{lemma: number of boxes}, we get $$\|f_{k+2,\tau_{k+1}}\|_{L^\infty}\lesssim R^{\epsilon/2}G.$$ By the definition of $G$ from Notation \ref{not: G}, we have that $G\lesssim R^2.$  Combining everything, we get $$|S_1(x;\tau_{k+1})|\lesssim R^{-100(k+1)-1}.$$
     The conclusion follows from the bounds on $S_1,S_2$, and Lemma \ref{lemma: number of boxes}.
\end{proof}
The bound \eqref{eqn: local constancy use 1} then follows if we have the additional normalisation \eqref{eqn: normalisation in sup norm}. Let us also record the following bound, which was obtained during the proof of the Lemma above. 
\begin{corollary}
    For all $k\in\{1,\dots,N\}$ and each $\tau_{k+1}\in\calt_{k+1}(\tau_k)$, we have $$|f_{k+1,\tau_{k+1}}(x)-f_{k+1,\tau_{k+1}}*\rho_{\tau_{k}}(x)|\lesssim R^{-100}\sup_\theta\|f_\theta\|_{L^\infty}\quad\text{for all}\quad x\in\R^2.$$
\end{corollary}
\subsection{Local orthogonality} 
Here we provide a proof of the local $L^2$-orthogonality heuristic used in the argument.
\begin{lemma}
    [Local $L^2$-orthogonality]\label{lemma: local orthogonality}
    Let $h_k:\R^2\to\mathbb{C}$ be Schwarz functions having Fourier support in $X_k\subset\R^2$, and let $W_U$ be a function with Fourier support in a set $U\subset\R^2$ containing the origin. If the maximum overlap of the sets $X_k+U$ is $L$, then $$\bigg|\int_{\R^2}|\sum_kh_k|^2W_U\bigg|\lesssim L\sum_k\int_{\R^2}|h_k|^2|W_U|.$$
\end{lemma}
\begin{proof}
   First we write $$\int_{\R^2}|\sum_kh_k|^2W_U=\sum_l\sum_k\int_{\R^2}\bar h_kh_lW_U,$$ and then by Plancherel's theorem we have $$\int_{\R^2}\bar h_k\cdot(h_lW_U)=\int_{\R^2}\overline{{\widehat h}_k}\cdot(\widehat h_l*\widehat W_U).$$ Now $\overline{{\widehat h}_k}$ is supported in $X_k\subset X_k+U$, and $\widehat h_l*\widehat W_U$ is supported in $X_l+U$, and so the integral above is non-zero only when $X_k+U$ and $X_l+U$ have a non-empty intersection. By the hypothesis of the lemma, it follows that $$\#\{l:(X_k+U)\cap(X_l+U)\neq\emptyset\}\leq L\quad\text{for all}\quad k.$$ 
   Thus for all $k$, there exists indices $k_1,\dots,k_L$ such that  $$\int_{\R^2}\bar h_kh_lW_U\neq 0\quad\text{only if}\quad l=k_1,\dots,k_L.$$
   Thus, $$\bigg|\int_{\R^2}|\sum_kh_k|^2W_U\bigg|=\bigg|\sum_{j=1}^L\sum_k\int_{\R^2}\bar h_kh_{k_j}W_U\bigg|\leq\int_{\R^2}\sum_{j=1}^L\sum_k|\bar h_{k}h_{k_j}||W_U|,$$ where the last inequality follows from the triangle inequality and Fubini's theorem.
    Finally, applying Cauchy--Schwarz, we get $$\int_{\R^2}\sum_{j=1}^L\sum_k|\bar h_{k}h_{k_j}||W_U|\leq  L\int_{\R^2}\sum_k|h_k|^2|W_U|,$$ which leads us to the desired bound.
\end{proof}
\subsection{A commutative diagram}
Here we give a proof of the claim from Lemma \ref{lemma: rescaling}.
\begin{lemma}\label{lemma: commutative diagram}
    Let $1\leq k\leq N$ and $\tau_k\in\calt_k^\Lambda(\Gamma)$. For $1\leq l\leq N-k$ let $\tau_{k+l}\in\calt_{k+l}^\Lambda(\Gamma;\tau_k)$. By the diagonal case of Lemma \ref{lemma: rescaling}, $\tau^k_l:=\cal_{\tau_k}(\tau_{k+l})\in\calt_l^{\Lambda'}(\Gamma_{\tau_k})$. The following diagram commutes.
    \begin{figure}[H]
        \centering
        \includegraphics[width=0.35\linewidth]{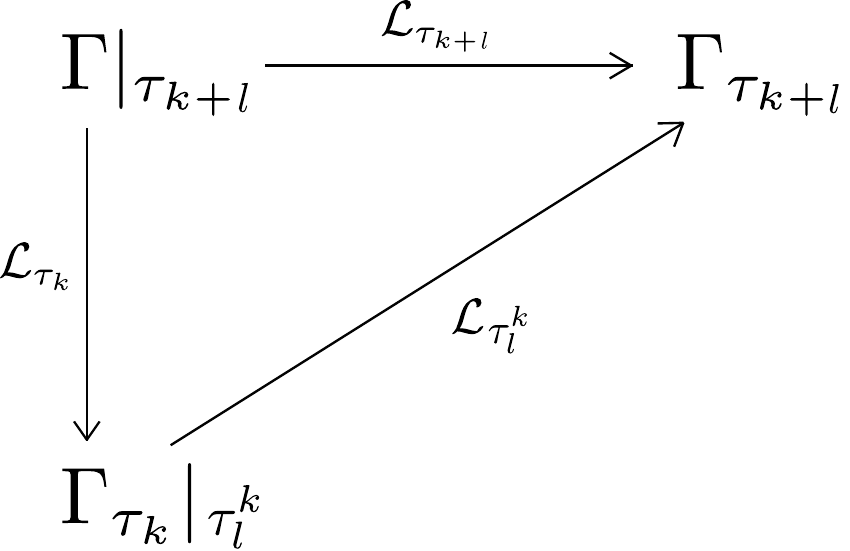}
    \end{figure}
\end{lemma}
\begin{proof}
    We argue by inducting on $l$. For $l=1$, this is follows immediately from Lemma \ref{lemma: cal composition law}. Suppose the claim is true for some $l$. Let $\tau_{k+l+1}\in\calt^\Lambda_{k+l+1}(\Gamma;\tau_k)$ and $\tau^k_{l+1}:=\cal_{\tau_k}(\tau_{k+l+1})\in\calt^{\Lambda'}_{l+1}(\Gamma_{\tau_k})$, which holds by the diagonal case of Lemma \ref{lemma: rescaling}. By the definition given in \eqref{eqn: pigeonholed boxes}, there exists some $\tau_{k+1}\in\calt^\Lambda_{k+1}(\Gamma;\tau_k)$ such that $\tau_{k+l+1}\in\calt^\Lambda_{k+l+1}(\Gamma;\tau_{k+1})$. By the definition in \eqref{eqn: calt k+1 def}, $\sigma_{k+1}:=\cal_{\tau_k}(\tau_{k+1})\in\calt_1(\Gamma_{\tau_k})$. In fact, $\sigma_{k+1}\in\calt_1^{\Lambda'}(\Gamma_{\tau_k})$. Again, by the diagonal case of Lemma \ref{lemma: rescaling}, we know that $\tau^{k+1}_l:=\cal_{\tau_{k+1}}(\tau_{k+l+1})\in\calt^{\Lambda'}_l(\Gamma_{\tau_{k+1}})$. The $l+1$ case of the lemma follows from the commutative diagrams below:    \begin{figure}[H]
        \begin{tabular}{cc}
           \includegraphics[width=0.48\linewidth]{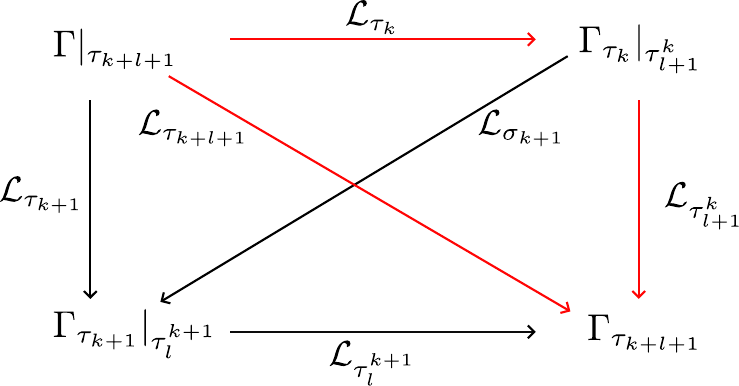} & \includegraphics[width=0.48\linewidth]{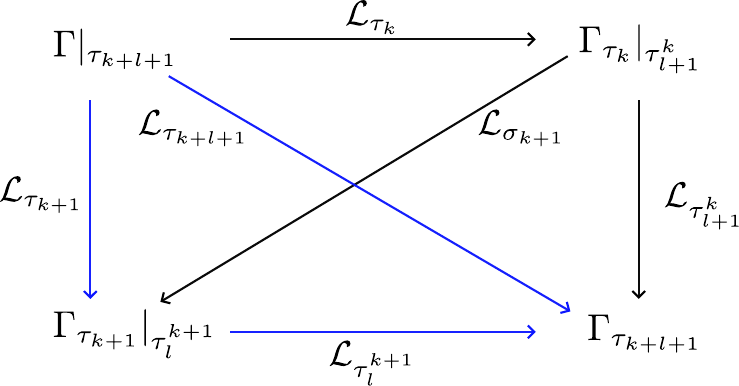} \\
           Diagram (a) & Diagram (b)\\\\\\
            \includegraphics[width=0.48\linewidth]{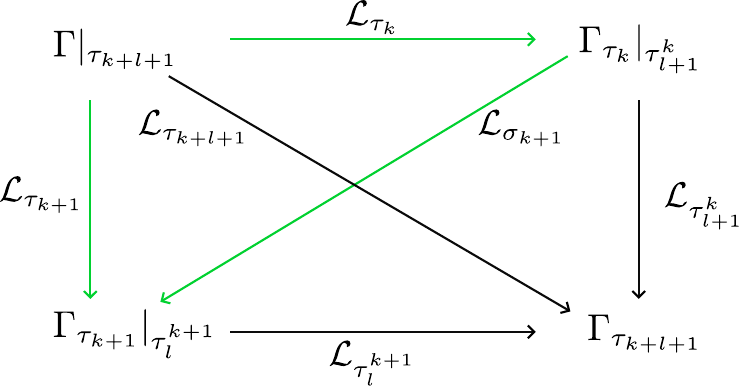} & \includegraphics[width=0.48\linewidth]{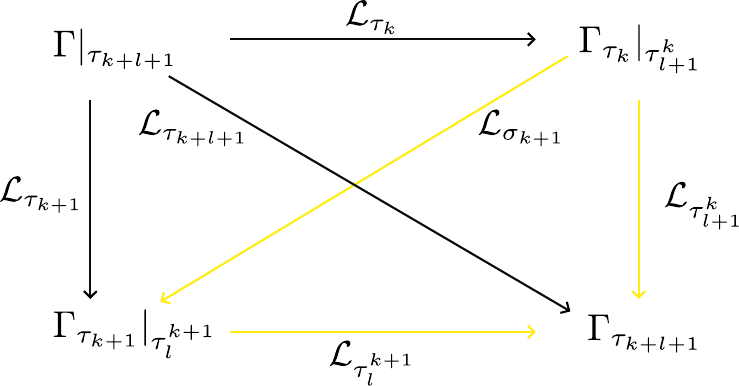} \\
            Diagram (c) & Diagram (d)
        \end{tabular}\end{figure} 
        \noindent The commutativity of Diagram (a) is the desired result of the $l+1$ case, and it follows from the commutativity of Diagram (b), (c), and (d). To see that Diagram (b) commutes, we write the index $k+l+1$ as $(k+1)+l$, and apply the induction hypothesis. To see why Diagram (c) commutes, we recall from above that 
        \begin{equation}
            \label{eqn: key observation for diagram c}
            \cal_{\tau_k}(\tau_{k+l+1})=\tau^k_{l+1},\quad\text{and}\quad \cal_{\tau_{k+1}}(\tau_{k+l+1})=\tau^{k+1}_l.
        \end{equation}
         Now we apply Lemma \ref{lemma: cal composition law}, and the rest follows from the relation above. As a consequence, we have that 
        \begin{equation}
            \label{eqn: diagram c corollary}
            \cal_{\sigma_{k+1}}(\tau^k_{l+1})=\tau^{k+1}_l.
        \end{equation}
        For Diagram (d), we only need to show that
    \begin{equation}
        \label{eqn: commutative diagram key observation}
        \cal_{\tau^k_{l+1}}=\cal_{\tau^{k+1}_l}\circ\cal_{\sigma_{k+1}}.
    \end{equation}
    This can be seen by applying the composition on the right-hand side of \eqref{eqn: commutative diagram key observation} to the box $\tau^k_{l+1}$. By \eqref{eqn: diagram c corollary} and the definition of the similarity transformations, we get $$\cal_{\tau^{k+1}_l}\circ\cal_{\sigma_{k+1}}(\tau^{k}_{l+1})=\cal_{\tau^{k+1}_l}(\tau^{k+1}_l)=J_{\tau^{k+1}_l}\times [-1,1],$$ where $J_{\tau}:=[0,\len(\tau)/\wid(\tau)]$. Now \eqref{eqn: key observation for diagram c} tells us that $\tau^k_{l+1}$ and $\tau^{k+1}_l$ are similar, and so $$\len(\tau^{k+1}_l)/\wid(\tau^{k+1}_l)=\len(\tau^k_{l+1})/\wid(\tau^k_{l+1}),\quad\text{so that}\quad J_{\tau^{k+1}_l}=J_{\tau^k_{l+1}}.$$ Using this in the display above, we find $$\cal_{\tau^{k+1}_l}\circ\cal_{\sigma_{k+1}}(\tau^{k}_{l+1})=J_{\tau^k_{l+1}}\times [-1,1]=\cal_{\tau^k_{l+1}}(\tau^k_{l+1}),$$ where the last equality follows from the definition of $\cal_{\tau^k_{l+1}}$. Thus $\cal_{\tau^{k+1}_l}\circ\cal_{\sigma_{k+1}}=\cal_{\tau^k_{l+1}}$, up to a symmetry of $\tau^k_{l+1}$. The similarities cannot differ by a single reflection, since they all have positive determinants, by definition \eqref{eqn: cal tau k def}. They also cannot differ by a $180^\degree$ rotation, as the angles used to define the similarities in \eqref{eqn: cal tau k def} are all acute angles (so their sum is strictly less than $180^\degree$). Therefore, \eqref{eqn: commutative diagram key observation} holds and Diagram (d) commutes. Combining everything, we get that Diagram (a) is commutative, which establishes the $l+1$ case of the lemma.
\end{proof} 
\bibliographystyle{plain}
\bibliography{Citation}
\end{document}